\documentclass[12pt,twoside]{amsart}
\usepackage{amssymb,amsmath,amsthm, amscd, enumerate, mathrsfs}
\usepackage{graphicx, hhline}
\usepackage[all]{xy}

\title[Basic slc-trivial fibrations]
{Fundamental properties of basic slc-trivial fibrations}
\author{Osamu Fujino}
\date{2020/3/10, version 0.30}
\subjclass[2010]{Primary 14N30; Secondary 14D07, 32G20, 14E30}
\keywords{canonical bundle formula, quasi-log canonical pair, 
semipositivity theorem, variation of mixed Hodge structure, subadjunction}
\address{Department of Mathematics, Graduate School of Science, 
Osaka University, Toyonaka, Osaka 560-0043, Japan}
\email{fujino@math.sci.osaka-u.ac.jp}
\DeclareMathOperator{\Nqklt}{Nqklt}
\DeclareMathOperator{\Gr}{Gr}
\DeclareMathOperator{\mld}{mld}
\DeclareMathOperator{\nqklt}{Nqklt}
\DeclareMathOperator{\Spec}{Spec}
\DeclareMathOperator{\Supp}{Supp}
\DeclareMathOperator{\mult}{mult}
\DeclareMathOperator{\rank}{rank}
\DeclareMathOperator{\codim}{codim}

\newtheorem{thm}{Theorem}[section]
\newtheorem{lem}[thm]{Lemma}
\newtheorem{prop}[thm]{Proposition}
\newtheorem{conj}[thm]{Conjecture}
\newtheorem{cor}[thm]{Corollary}
\newtheorem*{claim}{Claim}

\theoremstyle{definition}
\newtheorem{ex}[thm]{Example}
\newtheorem{defn}[thm]{Definition}
\newtheorem{rem}[thm]{Remark}
\newtheorem*{ack}{Acknowledgments}  
\newtheorem*{conventions}{Conventions}        
\newtheorem{say}[thm]{}
\newtheorem{step}{Step}
\newcommand{\Div}[0]{\operatorname{Div\!}}
\newcommand{\pd}{\Delta\!^{\ast}}
\newcommand{\rec}{_{\rm rec}}

\makeatletter
    
    \@addtoreset{equation}{section}
\makeatother

\begin{document}

\begin{abstract}
We introduce the notion of basic slc-trivial fibrations. 
It is a generalization of that of Ambro's lc-trivial fibrations. 
Then we study fundamental properties of basic slc-trivial 
fibrations by using the theory of variations of 
mixed Hodge structure on cohomology with compact support. 
More precisely, we prove that 
the moduli part of a basic slc-trivial fibration is b-potentially 
nef. 
Note that the notion of basic slc-trivial fibrations is closely related to 
that of normal irreducible quasi-log canonical pairs. 
So the results obtained in this paper will play an important role in 
the theory of quasi-log schemes. 
Here we give a structure theorem for normal irreducible quasi-log canonical 
pairs as an application of the main theorem. 
This result makes the theory of quasi-log schemes more powerful and more 
flexible. 
\end{abstract}

\maketitle 
\tableofcontents

\section{Introduction}\label{c-sec1}

Let us introduce {\em{basic slc-trivial fibrations}} 
$f:(X, B)\to Y$. 
They consist of a projective surjective morphism 
$f:X\to Y$ from a simple normal crossing variety $X$ to a 
normal 
irreducible variety $Y$ such that $(X, B)$ is a simple 
normal crossing pair and that $K_X+B$ is 
$\mathbb Q$-linearly trivial over $Y$. 
More precisely, we assume:  
\begin{itemize}
\item[(1)] $Y$ is a normal irreducible variety, 
\item[(2)] every stratum of $X$ is 
dominant onto $Y$ and $f_*\mathcal O_X\simeq 
\mathcal O_Y$, where $f:X\to Y$ is a projective surjective morphism, 
\item[(3)] $B$ is a $\mathbb Q$-divisor on $X$ such that 
$(X, B)$ is a simple normal crossing pair and that 
$B=B^{\leq 1}$ holds over the generic point of $Y$, 
\item[(4)] there exists a $\mathbb Q$-Cartier $\mathbb Q$-divisor 
$D$ on $Y$ such that 
$
K_X+B\sim _{\mathbb Q}f^*D$, 
and 
\item[(5)] $\rank f_*\mathcal O_X(\lceil -(B^{<1})\rceil)=1$. 
\end{itemize}
We note that $X$ is not necessarily irreducible in the above setup. 
It may be a {\em{reducible}} simple normal crossing variety. 
Of course, we are mainly interested in the case where $X$ is reducible. 
The notion of basic slc-trivial fibrations is a natural generalization of that 
of {\em{lc-trivial fibrations}} (see \cite{ambro3} and 
\cite{fujino-gongyo}) and will suit the theory of quasi-log schemes 
very well. 

\medskip 

In the above setup, 
let $\sigma:Y'\to Y$ be a birational morphism from a normal irreducible 
variety $Y'$. Then we can construct the following commutative 
diagram of basic slc-trivial fibrations: 
$$
\xymatrix{
   (X', B_{X'}) \ar[r]^{\mu} \ar[d]_{f'} & (X, B)\ar[d]^{f} \\
   Y' \ar[r]_{\sigma} & Y,  
} 
$$
where $B_{X'}$ is defined by $K_{X'}+B_{X'}=\mu^*(K_X+B)$ and 
$f': (X', B_{X'})\to Y'$ is nothing but the base change of $f:(X, B)\to Y$ by $\sigma: 
Y'\to Y$ on a nonempty Zariski open set of $Y'$. 
We call $f': (X', B_{X'})\to Y'$ an {\em{induced 
basic slc-trivial fibration}} of $f:(X, B)\to Y$ by $\sigma: Y'\to Y$. 
As for lc-trivial fibrations, we can define a {\em{discriminant 
$\mathbb Q$-b-divisor}} $\mathbf B$ and a {\em{moduli $\mathbb Q$-b-divisor}} 
$\mathbf M$ on $Y$ associated to $f:(X, B)\to Y$ 
(see \ref{c-say4.5}). 

\medskip

Before we state the main theorem of this paper, 
we have to introduce the notion of 
{\em{potentially nef}} $\mathbb Q$-divisors. 

\begin{defn}[Potentially nef divisors, 
see Definition \ref{c-def2.5}]\label{c-def1.1}
Let $X$ be a normal irreducible variety 
and let $D$ be a divisor on $X$. 
If there exist a normal complete variety $\overline X$ which 
contains $X$ as a dense Zariski open set 
and a nef divisor $\overline D$ on $\overline X$ 
such that $D=\overline D|_X$, then 
$D$ is called a {\em{potentially nef}} divisor on $X$. 
A finite $\mathbb R_{>0}$-linear (resp.~$\mathbb Q_{>0}$-linear) 
combination of potentially nef divisors is called 
a {\em{potentially nef}} $\mathbb R$-divisor (resp.~$\mathbb Q$-divisor). 
\end{defn}

Let us state the main theorem of this paper, 
which is a generalization of \cite[Theorem 
0.2]{ambro3} (see also \cite[Theorem 3.6]{fujino-gongyo}). 

\begin{thm}[Main Theorem]\label{c-thm1.2}
Let $f:(X, B)\to Y$ be a basic slc-trivial fibration and 
let $\mathbf B$ and $\mathbf M$ be the induced 
discriminant and moduli $\mathbb Q$-b-divisors of $Y$ respectively. 
Then we have the following properties: 
\begin{itemize}
\item[(i)] $\mathbf K+\mathbf B$ is $\mathbb Q$-b-Cartier, and 
\item[(ii)] $\mathbf M$ is b-potentially nef, that is,  
there exists a proper birational morphism $\sigma:Y'\to Y$ 
from a normal variety $Y'$ such that 
$\mathbf M_{Y'}$ is a potentially nef $\mathbb Q$-divisor on $Y'$ and 
that $\mathbf M=\overline{\mathbf M_{Y'}}$. 
\end{itemize}
\end{thm}

We note that $\mathbf K$ in Theorem \ref{c-thm1.2} 
is the {\em{canonical b-divisor}} of $Y$. 
For the precise definition of {\em{$\mathbb Q$-b-divisors}} and 
{\em{b-potentially nef divisors}}, see Definition \ref{c-def2.12} below. 

\medskip

Theorem \ref{c-thm1.2} can be restated as follows 
without using b-divisors. 

\begin{thm}\label{c-thm1.3}
Let $f:(X, B)\to Y$ be a basic slc-trivial fibration. 
Then there is a proper birational morphism $\sigma:Y'\to Y$ from a 
normal variety $Y'$ such that 
\begin{itemize}
\item[(i)] $K_{Y'}+B_{Y'}$ is $\mathbb Q$-Cartier 
and $\nu^*(K_{Y'}+B_{Y'})=K_{Y''}+B_{Y''}$ for every 
proper birational morphism $\nu:Y''\to Y'$ from a normal 
variety $Y''$, and  
\item[(ii)] $M_{Y'}$ is a $\mathbb Q$-Cartier $\mathbb Q$-divisor 
on $Y'$ that is potentially nef and 
$\nu^*M_{Y'}=M_{Y''}$ for every 
proper birational morphism $\nu:Y''\to Y'$ from a normal 
variety $Y''$. 
\end{itemize}
\end{thm}

We note that $B_{Y'}$ (resp.~$B_{Y''}$) is the discriminant 
$\mathbb Q$-divisor on $Y'$ (resp.~$Y''$) and 
that $M_{Y'}$ (resp.~$M_{Y''}$) is the 
moduli $\mathbb Q$-divisor on $Y'$ (resp.~$Y''$) in Theorem \ref{c-thm1.3}. 

\medskip 

In \cite{ambro3}, Florin Ambro established Theorem \ref{c-thm1.2} 
under the assumption that $(X, B)$ has only 
sub kawamata log terminal singularities over 
the generic point of $Y$. The case 
where $(X, B)$ has only sub log canonical 
singularities over the generic point of 
$Y$ was proved in \cite{fujino-gongyo}. 
Note that Ambro used the theory of variations of 
Hodge structure in \cite{ambro3} and 
Gongyo and the author used the theory of variations of 
mixed Hodge structure in \cite{fujino-gongyo}. 

\medskip 

On moduli $\mathbb Q$-b-divisors, we pose the following 
conjecture.  

\begin{conj}[b-semi-ampleness conjecture]\label{c-conj1.4} 
Let $f:(X, B)\to Y$ be a basic 
slc-trivial fibration. 
Then the moduli part $\mathbf M$ is 
b-semi-ample. 
\end{conj}

By Lemma \ref{d-lem4.12} below, we see that it is sufficient to 
prove Conjecture \ref{c-conj1.4} under the extra assumption that 
$Y$ is complete. 
Conjecture \ref{c-conj1.4} is still widely open even when 
$X$ is a smooth irreducible variety and $B=0$. 
For some known cases and related topics, we recommend the 
reader to see \cite{kawamata-sub1}, 
\cite{fujino-certain}, 
\cite{ambro5}, 
\cite{prokhorov-shokurov}, 
\cite{fujino-gongyo}, and so on. In a 
joint paper with Taro Fujisawa and Haidong Liu (see \cite{fujino-fujisawa-liu}), 
we will prove: 

\begin{thm}\label{c-thm1.5} 
If $Y$ is complete and $\mathbf M_{Y'}$ is numerically trivial in 
Theorem \ref{c-thm1.2}, 
then $\mathbf M_{Y'}\sim _{\mathbb Q}0$ holds. 
\end{thm}

As an easy consequence of Theorem \ref{c-thm1.5}, 
we have the following result. 

\begin{cor}\label{c-cor1.6} 
Conjecture \ref{c-conj1.4} holds true when $Y$ is a curve. 
\end{cor}

We note that Theorem \ref{c-thm1.5} is a generalization of 
\cite[Theorem 3.5]{ambro5} and \cite[Theorem 1.3]{floris} and 
that Corollary \ref{c-cor1.6} is a generalization of 
\cite[Theorem 0.1]{ambro3}. 
For the details of Theorem \ref{c-thm1.5} and Corollary \ref{c-cor1.6}, 
see \cite{fujino-fujisawa-liu}. 

\medskip 

As an application of Theorem \ref{c-thm1.3}, we will prove the 
following theorem. 
Theorem \ref{c-thm1.7}, which is 
one of the main motivations of this paper, 
will play an crucial role in the theory of 
quasi-log schemes. 

\begin{thm}[Structure theorem for 
normal irreducible quasi-log canonical pairs]\label{c-thm1.7}
Let $[X, \omega]$ be a quasi-log canonical pair such that 
$X$ is a normal irreducible variety. 
Then there exists a projective 
birational morphism 
$p:X'\to X$ from a smooth quasi-projective 
variety $X'$ such that 
$$
K_{X'}+B_{X'}+M_{X'}=p^*\omega, 
$$ 
where $B_{X'}$ is a subboundary $\mathbb R$-divisor, that is, 
$B_{X'}=B^{\leq 1}_{X'}$, 
such that $\Supp B_{X'}$ is a simple normal crossing 
divisor and that $B^{<0}_{X'}$ is 
$p$-exceptional, 
and $M_{X'}$ is a potentially nef $\mathbb R$-divisor on $X'$. 
Furthermore, we can make $B_{X'}$ satisfy 
$p(B^{=1}_{X'})=\nqklt(X, \omega)$.  

We further assume that $[X, \omega]$ has a $\mathbb Q$-structure. 
Then we can make $B_{X'}$ and $M_{X'}$ $\mathbb Q$-divisors in 
the above statement. 
\end{thm}

We note that there are many examples of quasi-log 
canonical 
pairs in the theory of minimal models. 

\begin{ex}\label{c-ex1.8}
(1) Let $(X, \Delta)$ be a quasi-projective semi-log 
canonical pair. 
Then $[X, \omega]$, where $\omega=K_X+\Delta$, is 
a quasi-log canonical pair such that 
$W$ is a qlc stratum of $[X, \omega]$ if and only if 
$W$ is an slc stratum of $(X, \Delta)$. 
For the details, see \cite[Theorem 1.2]{fujino-slc}. 

(2) Let $W$ be a qlc stratum of a quasi-log canonical pair 
$[X, \omega]$. 
Then $[W, \omega|_W]$ is also a quasi-log canonical 
pair by adjunction (see \cite[Theorem 6.3.5]
{fujino-foundations}). 

(3) Let $[X, \omega]$ be a quasi-log canonical 
pair such that $X$ is irreducible. 
Let $\nu:X^\nu\to X$ be the normalization of $X$. 
Then we can prove that 
$[X^\nu, \nu^*\omega]$ is a quasi-log canonical 
pair. For the details, see \cite[Theorem 1.1]{fujino-haidong}. 

(4) Let $W$ be an slc stratum of a quasi-projective 
semi-log canonical 
pair $(X, \Delta)$. Then, by (1), (2), and (3) above, 
we see that 
$[W, \omega|_W]$ and $[W^\nu, \nu^*(\omega|_W)]$ are 
quasi-log canonical pairs, where $\omega=K_X+\Delta$ and 
$\nu:W^\nu\to W$ is the normalization of $W$.  
\end{ex}

Here we give an important remark on Theorem \ref{c-thm1.7}. 

\begin{rem}[Generalized polarized pairs]\label{c-rem1.9}
We put $B_X=p_*B_{X'}$ and $M_X=p_*M_{X'}$ in Theorem 
\ref{c-thm1.7}. 
Then $B_X$ is a boundary $\mathbb R$-divisor on $X$, 
that is, an effective $\mathbb R$-divisor on $X$ with 
$B_X=B^{\leq 1}_X$, 
since $B^{<0}_{X'}$ is $p$-exceptional. 
Of course, $K_X+B_X+M_X$ is $\mathbb R$-Cartier by construction. 
Let $X\to S$ be any projective 
morphism between quasi-projective varieties. 
Then, $(X, B_X+M_X)$ is a {\em{generalized 
polarized pair}} which comes with the data $X'\overset{p}\longrightarrow 
X\longrightarrow S$ and $M_{X'}$ as in 
\cite[Definition 1.4]{birkar-zhang}. 
Moreover, we can easily check that 
$(X, B_X+M_X)$ is {\em{generalized lc}} in the sense of 
\cite[Definition 4.1]{birkar-zhang}. 
We note that $(X, B_X+M_X)$ is {\em{generalized klt}} in the 
sense of \cite[Definition 4.1]{birkar-zhang} 
when $\Nqklt (X, \omega)=\emptyset$. 
Since $M_{X'}$ is potentially nef $\mathbb R$-divisor, 
$M_{X'}$ is a finite $\mathbb R_{>0}$-linear 
combination of relatively nef Cartier divisors 
over $S$. Hence $(X, B_X+M_X)$ is an NQC g-pair 
in the sense of \cite[Definition 2.13]{han-li}. 
For the details of generalized polarized pairs, 
we recommend the reader to see 
\cite[Section 4]{birkar-zhang} and \cite{han-li}. 
\end{rem}

By Theorem \ref{c-thm1.7}, 
we can prove a kind of subadjunction formula for minimal qlc strata of 
quasi-log canonical pairs. Corollary \ref{c-cor1.10} is a complete 
generalization of \cite[Theorem 1]{kawamata}. For a different generalization 
of \cite[Theorem 1]{kawamata}, see \cite[Theorem 1.2]{fujino-gongyo-sub}. 
We also recommend the reader to see \cite{fujino-wenfei2} 
for a generalization of Corollary \ref{c-cor1.10}. 

\begin{cor}[Subadjunction for minimal qlc strata]\label{c-cor1.10}
Let $[X, \omega]$ be a quasi-log canonical pair and let $W$ be a minimal 
qlc stratum of $[X, \omega]$. 
We assume that $W$ is quasi-projective and 
$H$ is any ample $\mathbb R$-divisor on $W$. 
Then we can construct an effective $\mathbb R$-divisor 
$\Delta_W$ on $W$ such that $(W, \Delta_W)$ is kawamata log 
terminal with $K_W+\Delta_W\sim _{\mathbb R} \omega|_W+H$. 
We further assume that $[X, \omega]$ has a $\mathbb Q$-structure 
and $H$ is an ample $\mathbb Q$-divisor on $W$. 
Then we can make $\Delta_W$ a $\mathbb Q$-divisor with 
$K_W+\Delta_W\sim _{\mathbb Q}\omega|_W+H$.  
\end{cor}

As an application of Theorem \ref{c-thm1.7}, 
we prove: 

\begin{cor}[\cite{fujino-haidong2}]\label{c-cor1.11} 
Every quasi-log canonical pair has only Du Bois singularities. 
\end{cor}

Corollary \ref{c-cor1.11} is a complete 
generalization of \cite[Corollary 6.32]{kollar}. 
We discuss Corollary \ref{c-cor1.11} and some related 
topics in a joint paper with Haidong Liu (see \cite{fujino-haidong2}). 
We note that the arguments in \cite{fujino-haidong2} and 
this paper are free from the minimal model program. 

\medskip

By using Theorem \ref{c-thm1.7}, we will also prove: 

\begin{cor}[Simply connectedness and 
rationally chain connectedness of quasi-log canonical 
Fano pairs, \cite{fujino-wenfei2}]\label{c-cor1.12} 
Let $[X, \omega]$ be a connected projective 
quasi-log canonical pair. 
Assume that $-\omega$ is ample. 
Then $X$ is simply connected and rationally chain connected. 
\end{cor}

\begin{cor}[Lengths of extremal rational curves, 
\cite{fujino-wenfei2}]\label{c-cor1.13} 
Let $[X, \omega]$ be a quasi-log canonical 
pair and let $\pi:X\to S$ be a projective morphism 
onto a variety $S$. 
Then every $\omega$-negative extremal ray $R$ of the 
relative Kleiman--Mori cone 
$\overline{NE}(X/S)$ is spanned by a rational curve $C$ 
with $0<-\omega\cdot C\leq 2\dim X$. 
\end{cor}

We will discuss a generalization of Corollary 
\ref{c-cor1.10}, Corollaries \ref{c-cor1.12} 
and \ref{c-cor1.13}, and some other applications in \cite{fujino-wenfei2}. 

\medskip 

Finally, as an application of Theorem \ref{c-thm1.7}, 
we prove the following Fujita-type freeness for 
quasi-log canonical surfaces in a 
joint paper with Haidong Liu (see \cite{fujino-haidong3}). 

\begin{cor}[{\cite{fujino-haidong3}}]\label{c-cor1.14} 
Let $[X, \omega]$ be a projective 
quasi-log canonical pair of dimension two and 
let $M$ be a Cartier divisor on $X$. 
We put $N=M-\omega$. 
Assume that $N^2\cdot X_i>4$ for every 
two-dimensional irreducible 
component $X_i$ of $X$ and 
that $N\cdot C\geq 2$ for every curve $C$ on $X$. 
Then the complete linear system $|M|$ is 
basepoint-free. 
\end{cor}

Corollary \ref{c-cor1.14} is a generalization of the result 
for semi-log canonical surfaces obtained in \cite{fujino-slc-surface}. 

\medskip 

We strongly recommend the reader 
to see \cite{fujino-haidong2}, \cite{fujino-haidong3}, \cite{fujino-fujisawa-liu},  
and \cite{fujino-wenfei2} after reading this paper. 

\begin{say}[Historical comments on 
related papers, and so on]\label{c-say1.15}
One of the starting points of this paper is 
Mori's work in \cite[Section 5, Part II]{mori}. 
It is a prototype of the so-called Fujino--Mori canonical bundle 
formula (see \cite{fujino-mori}). 
We note that \cite{fujino-mori} is an expanded version of 
Mori's unpublished preprint written and circulated around 1994. 
We also note that the moduli part is called the {\em{semistable 
part}} in \cite{fujino-mori}. 
In \cite[Theorem 2]{kawamata}, Kawamata 
essentially proved that the moduli part of a klt-trivial fibration is 
nef. 
After the author learned \cite[Theorem 2]{kawamata}, 
he soon got some applications of Kawamata's 
result in \cite{fujino-applications} and 
then obtained the so-called Fujino--Mori 
canonical bundle formula with Shigefumi Mori 
by combining Mori's unpublished preprint 
with \cite[Theorem 2]{kawamata}. 
Then the author discussed the semi-ampleness of semistable 
parts for certain algebraic fiber spaces in \cite{fujino-certain} 
and also proved that the semistable part 
behaves very well under pull-back in \cite[Section 4]{fujino-certain}. 
In \cite[Section 4]{fujino-higher-pre}, 
he essentially proved that the moduli part of an lc-trivial 
fibration is nef. 
This result is a direct generalization of \cite[Theorem 2]{kawamata}. 
From the Hodge theoretic viewpoint, 
\cite{kawamata} is {\em{pure}} and \cite[Section 4]{fujino-higher-pre} 
is {\em{mixed}}. We note that 
\cite[Sections 4 and 5]{fujino-higher-pre} was not published. 
If the author remembers correctly, 
he planned to divide \cite{fujino-higher-pre} into 
two papers following the editor's recommendation (see 
\cite[Remark 1.1]{fujino-higher}). 
On the other hand, Ambro started to study 
some applications of \cite[Theorem 2]{kawamata} 
in his thesis (see \cite{ambro1}) independently. 
Then he formulated lc-trivial fibrations, which are now called 
klt-trivial fibrations in this paper, 
and proved that the moduli part is b-nef (see \cite{ambro3}). 
His result recovers \cite[Theorem 2]{kawamata}. 
However, his proof is different from Kawamata's original one 
in \cite{kawamata} 
and is essentially the same as 
the arguments in \cite[Section 5, Part II]{mori} and 
\cite[Section 4]{fujino-certain}. 
Moreover, in \cite{ambro5}, Ambro proved that the moduli 
part of a klt-trivial fibration is {\em{b-nef and abundant}} 
under some mild assumptions. 
Note that this deep result was generalized 
for lc-trivial fibrations by \cite{fujino-gongyo}. 
More precisely, in \cite{fujino-gongyo}, 
Gongyo and the author showed how to reduce some problems for 
lc-trivial fibrations to those for klt-trivial fibrations. 
On the semi-ampleness, Kawamata essentially proved that 
the moduli part of an lc-trivial fibration is semi-ample when the 
dimension of general fibers is one in \cite{kawamata-sub1} 
(see also \cite{prokhorov-shokurov}). 
As we mentioned before, 
the b-semi-ampleness conjecture (see 
Conjecture \ref{c-conj1.4}) is still widely open. 
We recommend the reader to see \cite{fujino-some}, 
where the author discussed various topics around lc-trivial 
fibrations. 
Roughly speaking, in \cite{fujino-some}, 
the author formulated lc-trivial fibrations for K\"ahler manifolds 
and proved the finite generation of 
canonical rings for compact K\"ahler manifolds. 
We also recommend the reader to see \cite{fujino-zucker65} 
for a survey on some 
related topics. 
Finally, we note that Koll\'ar surveys lc-trivial 
fibrations in \cite{kollar1}. 
His treatment is slightly different from others. 

From 2006 to 2007, the author wrote a preprint \cite{fujino-lmmp}, 
where he obtained some generalizations of Koll\'ar's injectivity, 
vanishing, and torsion-free theorems by using the theory of 
mixed Hodge structures on cohomology with 
compact support. 
Note that a completely revised and expanded version of 
\cite{fujino-lmmp} is now published as Chapter 5 of \cite{fujino-foundations} 
(see also \cite{fujino-vanishing} and \cite{fujino-injectivity}). 
The main motivation of \cite{fujino-lmmp} 
is to establish some generalizations 
of Koll\'ar's theorems for the theory of 
quasi-log schemes introduced by Florin Ambro (see \cite{ambro2}). 
In 2009, he wrote a very preliminary version of \cite{fujino-fujisawa} 
and started a joint work with Taro Fujisawa. 
One of his motivations of \cite{fujino-fujisawa} 
is to formulate an ultimate 
generalization of the Fujita--Zucker--Kawamata 
semipositivity theorem and obtain some kind of canonical 
bundle formula for reducible varieties 
by using the theory of variations of mixed Hodge 
structure on cohomology with compact support. 
Soon after they released a preprint version of \cite{fujino-fujisawa} 
in 2012, the author 
got the projectivity of the coarse moduli spaces of stable 
varieties in \cite{fujino-projectivity} as an easy application of 
\cite{fujino-fujisawa}. 
He thought that the paper \cite{fujino-projectivity} was an 
important unexpected application of \cite{fujino-fujisawa} 
because everyone thought that the projectivity 
of the coarse moduli spaces of stable varieties had been 
already proved in \cite{kollar-projectivity}. 
Note that the main result of \cite{fujino-projectivity} 
now can be proved without using the theory of 
variations of mixed Hodge structure (see \cite{fujino-vani-semi}). 
The proof in \cite{fujino-vani-semi} 
uses the Koll\'ar--Ohsawa type vanishing theorem 
for simple normal crossing pairs. 

In this paper, we discuss a kind of canonical 
bundle formula for reducible varieties, which we call a 
basic slc-trivial firbration, as an application of 
\cite{fujino-fujisawa}. 
This paper relates the theory of variations of mixed Hodge 
structure on cohomology with compact support discussed in 
\cite{fujino-fujisawa} to the theory of 
quasi-log schemes established in \cite[Chapter 6]{fujino-foundations}. 
Therefore, the results in this paper will 
play a crucial role in the study of quasi-log schemes. 
\end{say}

We briefly explain the organization of this paper. 
In Section \ref{c-sec2}, we fix the notation and recall 
various basic results for the reader's convenience. 
Here we introduce the notion of potentially nef divisors 
and explain some basic properties. 
Section \ref{c-sec3} is a short section on the theory of variations of 
mixed Hodge structure on cohomology with compact support. 
We explain some results in \cite{fujino-fujisawa}. 
Note that Theorem \ref{c-thm3.1} is the main ingredient of this paper. 
Theorem \ref{c-thm3.1} is a generalization of 
the Fujita--Zucker--Kawamata semipositivity theorem. 
In Section \ref{c-sec4}, we introduce the notion of (pre-)basic 
slc-trivial fibrations, define discriminant $\mathbb Q$-b-divisors and 
moduli $\mathbb Q$-b-divisors, and study some basic properties. 
The notion of basic slc-trivial fibrations is a generalization of 
that of Ambro's lc-trivial fibrations. 
In Section \ref{c-sec5}, we treat an inversion of adjunction 
for pre-basic slc-trivial fibrations under some assumptions. 
Although we do not need the result in Section \ref{c-sec5} explicitly 
in this paper, the calculation 
in Section \ref{c-sec5} may help the reader understand 
Theorem \ref{c-thm1.7}. 
In Section \ref{c-sec6}, we take a cyclic cover of 
the generic fiber of a given basic slc-trivial fibration 
to construct a new pre-basic slc-trivial fibration. 
Then we interpret the moduli part of a given 
basic slc-trivial fibration Hodge theoretically. 
In Section \ref{c-sec7}, we discuss various covering lemmas 
essentially due to Yujiro Kawamata. 
We will use them in the subsequent sections. 
In Section \ref{c-sec8}, we prove that the moduli part of 
a basic slc-trivial fibration behaves very well under pull-back by 
generically finite surjective 
morphisms with some mild assumptions. 
Section \ref{c-sec9} is devoted to the proof of the main 
theorem:~Theorem \ref{c-thm1.2}. 
In Section \ref{c-sec10}, we treat normal irreducible quasi-log canonical 
pairs. By the main result in Section \ref{c-sec10}, we see that 
a normal irreducible quasi-log canonical pair with $\mathbb Q$-structure 
can be seen as a basic slc-trivial fibration. 
This fact is one of the main motivations to introduce the notion of 
basic slc-trivial fibrations. 
In Section \ref{c-sec11}, we prove Theorem \ref{c-thm1.7} 
as an application of Theorem \ref{c-thm1.2}. 
By this theorem, we see that normal irreducible 
quasi-log canonical pairs are similar to log canonical 
pairs. 
Section \ref{c-sec12} is a short section on a remark 
about the basepoint-free theorem 
for quasi-log canonical pairs. 
In the final section:~Section \ref{c-sec13}, 
we give some supplementary remarks on \cite{fujino-fujisawa}, which is 
one of the main ingredients of this paper, for the reader's 
convenience. 

\begin{ack}
The author was partially 
supported by JSPS KAKENHI Grant Numbers 
JP16H03925, JP16H06337. 
He would like to thank Yoshinori Gongyo, 
Wenfei Liu, Takeshi Abe, Kenta Hashizume, 
and Haidong Liu for discussions. 
He also would like to thank Professor Taro Fujisawa very 
much for useful discussions and advice, and for allowing him to 
include \cite{fujino-fujisawa3} in this paper. 
Finally, he thanks the referee very much for useful comments.  
\end{ack}

\begin{conventions}
We will work over $\mathbb C$, the complex number field, throughout 
this paper. We will freely use the basic 
notation of the minimal model program as in 
\cite{fujino-fund} and \cite{fujino-foundations}. 
A {\em{scheme}} means a separated scheme of 
finite type over $\mathbb C$. 
A {\em{variety}} means a reduced scheme, that is, 
a reduced separated scheme of finite type over $\mathbb C$. 
In this paper, a variety may be reducible. 
However, we sometimes assume that a variety is irreducible without 
mentioning it explicitly if there is no danger of confusion. 
The set of integers (resp.~rational numbers or real numbers) 
is denoted by $\mathbb Z$ (resp.~$\mathbb Q$ or $\mathbb R$). 
The set of nonnegative (resp.~positive) rational numbers 
is denoted by $\mathbb Q_{\geq 0}$ (resp.~$\mathbb Q_{>0}$). 
We use $\mathbb Z_{\geq 0}$, $\mathbb Z_{>0}$, 
$\mathbb R_{\geq 0}$, and $\mathbb R_{>0}$ similarly. 
\end{conventions}

\section{Preliminaries}\label{c-sec2}
In this section, we fix the notation and 
recall some basic results for the reader's convenience. 

\begin{say}[Divisors]\label{c-say2.1} 
Let $X$ be a scheme with structure sheaf $\mathcal O_X$ and let 
$\mathcal K_X$ be the sheaf of total quotient rings of $\mathcal O_X$. 
Let $\mathcal K^*_X$ denote the (multiplicative) 
sheaf of invertible elements in $\mathcal K_X$, 
and $\mathcal O^*_X$ the sheaf of invertible elements in $\mathcal O_X$. 
We note that $\mathcal O_X\subset \mathcal K_X$ and $\mathcal O^*_X\subset 
\mathcal K^*_X$ hold. 
A {\em{Cartier divisor}} $D$ on $X$ is a global section of 
$\mathcal K^*_X/\mathcal O^*_X$, that is, 
$D$ is an element of $\Gamma(X, \mathcal K^*_X/\mathcal O^*_X)$. 
A {\em{$\mathbb Q$-Cartier divisor}} (resp.~An 
{\em{$\mathbb R$-Cartier divisor}}) is an element of 
$\Gamma (X, \mathcal K^*_X/\mathcal O^*_X)\otimes 
_{\mathbb Z}\mathbb Q$ (resp.~$\Gamma (X, \mathcal K^*_X/\mathcal O^*_X)
\otimes _{\mathbb Z} \mathbb R$). 

Let $D_1$ and $D_2$ be two $\mathbb R$-Cartier divisors 
on $X$. Then $D_1$ is {\em{linearly}} 
(resp.~{\em{$\mathbb Q$-linearly}}, or {\em{$\mathbb R$-linearly}}) 
{\em{equivalent}} to $D_2$, denoted by $D_1\sim D_2$ (resp.~$D_1
\sim _{\mathbb Q}D_2$, or $D_1\sim _{\mathbb R} D_2$) if 
$$
D_1=D_2+\sum _{i=1}^k r_i (f_i)
$$ 
such that $f_i\in \Gamma (X, \mathcal K^*_X)$ and $r_i\in \mathbb Z$ 
(resp.~$r_i\in \mathbb Q$, or $r_i\in \mathbb R$) for 
every $i$. 
We note that $(f_i)$ is a {\em{principal Cartier divisor}} 
associated to $f_i$, that is, 
the image of $f_i$ by 
$$
\Gamma (X, \mathcal K^*_X)\to \Gamma(X, \mathcal K^*_X/\mathcal O^*_X). 
$$
Let $f:X\to Y$ be a morphism between schemes. 
If there exists an $\mathbb R$-Cartier 
(resp.~a $\mathbb Q$-Cartier) divisor $B$ on $Y$ such that 
$D_1\sim _{\mathbb R} D_2+f^*B$ (resp.~$D_1\sim _{\mathbb Q} 
D_2+f^*B$), then $D_1$ is said to be {\em{relatively $\mathbb R$-linearly 
{\em{(}}resp.~$\mathbb Q$-linearly{\em{)}} equivalent}} to $D_2$. 
It is denoted by $D_1\sim _{\mathbb R, f}D_2$ or 
$D_1\sim _{\mathbb R, Y} D_2$ (resp.~$D_1\sim 
_{\mathbb Q, f}D_2$ or $D_1\sim _{\mathbb Q, Y}D_2$). 

\medskip

From now on, let $X$ be an equidimensional scheme. We note 
that $X$ is not necessarily regular in codimension one. 
A ({\em{Weil}}) {\em{divisor}} $D$ on $X$ is a finite formal 
sum 
$$
D=\sum _i d_iD_i, 
$$
where $D_i$ is an irreducible reduced closed subscheme of $X$ 
of pure codimension one and $d_i$ is an integer 
for every $i$ such that $D_i\ne D_j$ for every $i\ne j$. 
If $d_i \in \mathbb Q$ (resp.~$d_i \in \mathbb R$) for every $i$, 
then $D$ is called a {\em{$\mathbb Q$-divisor}} 
(resp.~an {\em{$\mathbb R$-divisor}}). 
Let $D=\sum _i d_i D_i$ be an $\mathbb R$-divisor as above. 
We put 
\begin{equation*}
D^{\leq 1}=\sum _{d_i\leq 1}d_i D_i, \quad 
D^{<1} =\sum _{d_i<1}d_iD_i, \quad 
D^{= 1}=\sum _{d_i= 1} D_i, \quad \text{and} \quad
\lceil D\rceil =\sum _i \lceil d_i \rceil D_i, 
\end{equation*}
where $\lceil d_i\rceil$ is the integer defined by $d_i\leq 
\lceil d_i\rceil <d_i+1$. 
Moreover, we put $\lfloor D\rfloor =-\lceil -D\rceil$ and 
$\{D\}=D-\lfloor D\rfloor$. 
Let $D$ be an $\mathbb R$-divisor.  
We call $D$ a {\em{subboundary}} 
$\mathbb R$-divisor if $D=D^{\leq 1}$ holds. 
When $D$ is effective and $D=D^{\leq 1}$ holds, 
we call $D$ a {\em{boundary}} $\mathbb R$-divisor. 

We further assume that 
$f:X\to Y$ is a surjective morphism onto an irreducible 
variety $Y$. 
Then we put 
$$
D^v=\sum _{f(D_i)\subsetneq Y}d_i D_i \quad 
\text{and} \quad D^h=D-D^v, 
$$
and call $D^v$ the {\em{vertical part}} 
and $D^h$ the {\em{horizontal part}} of $D$ 
with respect to $f:X\to Y$, respectively. 
\end{say}

\begin{say}[Singularities of pairs]\label{c-say2.2}
A pair $(X, \Delta)$ consists of a normal variety $X$ and 
an $\mathbb R$-divisor $\Delta$ on $X$ such that 
$K_X+\Delta$ is $\mathbb R$-Cartier. 
A pair $(X, \Delta)$ is called {\em{sub kawamata log terminal}} 
(resp.~{\em{sub log canonical}}) if for any 
proper birational morphism 
$f:Y\to X$ from a normal variety $Y$, every coefficient 
of $\Delta_Y$ is $<1$ (resp.~$\leq 1$) where 
$$K_Y+\Delta_Y:=f^*(K_X+\Delta).$$ 
A pair 
$(X, \Delta)$ is called {\em{kawamata log terminal}} 
(resp.~{\em{log canonical}}) if $(X, \Delta)$ is sub kawamata log terminal 
(resp.~sub log canonical) and $\Delta$ is effective. 

Let $(X, \Delta)$ be a sub log canonical pair and let $W$ be a closed 
subset of $X$. 
Then $W$ is called a {\em{log canonical center}} of $(X, \Delta)$ if there 
exist  
a proper birational morphism 
$f:Y\to X$ from a normal variety $Y$ and a prime divisor 
$E$ on $Y$ such that $\mult _E\Delta_Y=1$ and $f(E)=W$. 

We note that $-\mult _E\Delta_Y$ is denoted by $a(E, X, \Delta)$ 
for any prime divisor $E$ on $Y$ 
and 
is called the {\em{discrepancy coefficient}} of $E$ with respect to $(X, \Delta)$. 

\medskip 

Let $X$ be a normal variety and let $\Delta$ 
be an $\mathbb R$-divisor on $X$ such that 
$K_X+\Delta$ is $\mathbb R$-Cartier. 
Under this assumption, we can define the discrepancy 
coefficient $a(E, X, \Delta)$ for any prime divisor 
$E$ {\em{over}} $X$ by taking a suitable resolution of singularities. 
The {\em{minimal log discrepancy}} of $(X, \Delta)$ in a closed 
subset $Z\subsetneq X$ is 
$$
\mld _Z(X, \Delta):=\inf _{c_X(E)\subset Z}a(E, X, \Delta)+1, 
$$ 
where $E$ is a prime divisor over $X$ and $c_X(E)$ is the center of 
$E$ on $X$. 

\medskip 

In this paper, we mainly treat reducible varieties. 
So we need the notion of ({\em{sub}}) {\em{semi-log canonical singularities}}. 
 
\begin{defn}[Semi-log canonical singularities]\label{c-def2.3}
Let $X$ be an equidimensional variety that 
satisfies Serre's $S_2$ condition and is normal crossing in 
codimension one. 
Let $\Delta$ be an $\mathbb R$-divisor on $X$ such that 
no irreducible component of $\Supp \Delta$ 
is contained in the singular locus 
of $X$ and that $K_X+\Delta$ is $\mathbb R$-Cartier. 
We say that $(X, \Delta)$ has only {\em{sub semi-log canonical}} 
({\em{sub slc}}, for short) singularities if $(X^\nu, \Delta_{X^\nu})$ 
is sub log canonical, where $\nu: X^\nu \to X$ is the normalization of $X$ 
and $K_{X^\nu}+\Delta_{X^\nu}=\nu^*(K_X+\Delta)$, 
that is, $\Delta_{X^\nu}$ is the sum of the inverse images of $\Delta$ and 
the conductor of $X$. 
An {\em{slc center}} of $(X, \Delta)$ is the $\nu$-image of an lc 
center of $(X^\nu, \Delta_{X^\nu})$. 
An {\em{slc stratum}} of $(X, \Delta)$ means either 
an slc center of $(X, \Delta)$ or an 
irreducible component of $X$. 
If $(X, \Delta)$ has only sub semi-log canonical 
singularities and $\Delta$ is effective, then we say that 
$(X, \Delta)$ has only {\em{semi-log canonical}} 
({\em{slc}}, for short) singularities. 
\end{defn}
If $(X, \Delta)$ is (sub) semi-log canonical 
and $X$ is normal, then $(X, \Delta)$ is (sub) log canonical 
by definition. 

\medskip 

For the details of semi-log canonical 
singularities, see \cite{fujino-slc} and \cite{kollar}. 
\end{say}

\begin{say}[Potentially nef divisors]\label{c-say2.4}
Let us introduce the notion of {\em{potentially nef}} divisors. 
It is indispensable for the main theorem of this paper:~Theorem \ref{c-thm1.2}. 

\begin{defn}[Potentially nef divisors]\label{c-def2.5}
Let $X$ be a normal 
irreducible variety and let $D$ be a divisor on $X$. 
If there exist a completion $\overline X$ of $X$, 
that is, $\overline X$ is a normal complete 
variety and contains $X$ as a dense Zariski open set, and 
a nef divisor $\overline D$ on $\overline X$ such that 
$D=\overline D|_X$, then $D$ is called 
a {\em{potentially nef}} divisor on $X$. 
A finite $\mathbb R_{>0}$-linear (resp.~$\mathbb Q_{>0}$-linear) 
combination of potentially nef divisors is called 
a {\em{potentially nef}} $\mathbb R$-divisor 
(resp.~$\mathbb Q$-divisor). 
\end{defn}

The following easy lemma is very important in some 
applications. 

\begin{lem}\label{c-lem2.6}
Let $X$ be a normal irreducible quasi-projective 
variety, let $D$ be a potentially nef divisor on $X$, 
and let $H$ be an ample 
divisor on $X$. 
Then $D+H$ is ample. 
\end{lem}

We give a detailed proof for the 
reader's convenience. 

\begin{proof}[Proof of Lemma \ref{c-lem2.6}]
It is sufficient to prove that the line bundle 
$\mathcal O_X(D+H)$ is ample. 
Therefore, by replacing $D$ and $H$ with $mD$ and $mH$ for some 
positive integer $m$, respectively, 
we may assume that $H$ is very ample (see, for example, 
\cite[Chapter II, Theorem 7.6]
{hartshorne}). 
Thus, there exists an embedding 
$i:X\hookrightarrow \mathbb P^N$ such that 
$\mathcal O_X(H)\simeq i^*\mathcal O_{\mathbb P^N}(1)$. 
Let $X^\dag$ be the closure of $X$ in $\mathbb P^N$. 
Let $\overline X$ be a completion of 
$X$ on which there is a nef divisor $\overline D$ such that 
$D=\overline D|_X$. 
By \cite[Lemma 2.2]{lut}, which is an easy application of 
the flattening theorem (see \cite[Th\'eor\`eme (5.2.2)]{raynaud-g}), 
we can take an ideal sheaf 
$\mathcal I$ on $X^\dag$ 
with $\Supp \mathcal O_{X^\dag}/\mathcal I\subset X^\dag \setminus X$ 
such that the blow-up of $X^\dag$ along 
$\mathcal I$ eliminates the indeterminacy of $X^\dag\dashrightarrow \overline X$. 
Therefore, by taking the normalization of the blow-up of 
$X^\dag$ along $\mathcal I$, we get a projective 
birational morphism $\alpha:\widetilde X\to X^\dag$, 
which is an isomorphism over $X$, 
from a normal variety $\widetilde X$ such that 
the induced birational map $\beta:\widetilde X\dashrightarrow \overline X$ 
is a morphism, and an effective divisor $E$ on $\widetilde X$ such 
that $\Supp E \subset \widetilde X\setminus X$ and $-E$ is $\alpha$-ample. 
Note that we can see $X$ as a Zariski open set of $\widetilde X$. 
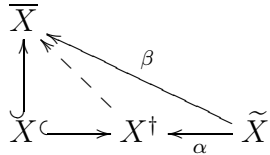
$$
\xymatrix{
\overline X & & \\ 
X \ar@{^{(}->}[u]\ar@{^{(}->}[r]
& X^\dag\ar@{-->}[ul] & \widetilde X\ar[ull]_-\beta\ar[l]^-\alpha
}
$$
Therefore, we can construct an ample line bundle $\mathcal L$ on $\widetilde X$ 
such that 
$\mathcal L|_X\simeq \mathcal O_X(lH)$ for some 
positive integer $l$. We consider 
the nef divisor $\beta^*\overline D$ on $\widetilde X$. 
Since $\widetilde X$ is projective and $\mathcal L$ is an ample 
line bundle on $\widetilde X$, 
$\mathcal L\otimes \mathcal O_{\widetilde X} (l\beta^*\overline D)$ 
is ample. By restricting it to 
$X$, 
we obtain that $\mathcal O_X(lD+lH)$ is an ample 
line bundle on $X$. 
Thus, $\mathcal O_X(D+H)$ is ample. 
This is what we wanted. 
\end{proof}

We note that 
any Cartier divisor is potentially nef when $X$ is affine. 

\begin{lem}\label{c-lem2.7}
Let $X$ be a normal irreducible 
affine variety and let $D$ be a {\em{(}}not necessarily effective{\em{)}} 
Weil divisor on $X$ which is Cartier. 
Then there exist a normal irreducible 
projective variety $\overline X$ containing $X$ as a 
dense Zariski open set and a Weil divisor $\overline D$ on $\overline X$ such that 
$D=\overline D|_X$ and that $\mathcal O_{\overline X}(\overline D)$ is a very 
ample line bundle on $\overline X$. 
In particular, $D$ is potentially nef. 
\end{lem}

\begin{proof}
We fix a closed embedding $X\subset \mathbb C^N$. 
Then we take the closure $X_1$ of $X$ in $\mathbb P^N$. 
Note that there exists a hyperplane $H$ on $\mathbb P^N$ such that 
\begin{equation}\label{c-eq2.1}
\Supp H|_{X_1}=X_1\setminus X. 
\end{equation} 
Let $X_2$ be the normalization of $X_1$. 
In this situation, we can see $X$ as a dense 
Zariski open set of $X_2$. 
Let $D_2$ be the closure of $D$ in $X_2$. 
We take an ample Cartier divisor $H^\dag$ on $X_2$ and a sufficiently 
large positive integer $l$. 
Then we can take an effective Weil divisor 
$\Gamma$ which is linearly equivalent to $D_2+lH^\dag$, that is, 
$\Gamma-D_2\sim lH^\dag$. 
We take the normalization of the blow-up of $X_2$ along the 
ideal sheaf $\mathcal O_{X_2}(-\Gamma)$. Then 
we get a projective 
birational morphism $p:X_3\to X_1$ from a normal variety 
$X_3$ and a Weil divisor $D_3$ on $X_3$ such that 
$p$ is an isomorphism 
over $X$ and that $D_3$ is a Cartier divisor 
satisfying 
$D=D_3|_X$. 
Note that we saw $X$ as a dense Zariski open set of $X_3$. 
As in the proof of Lemma \ref{c-lem2.6}, 
we take the normalization of the blow-up of $X_1$ along a suitable 
ideal sheaf on $X_1$ to eliminate the indeterminacy 
of $p^{-1}: X_1\dashrightarrow X_3$. 
Then we get the following commutative diagram 
$$
\xymatrix{
& \overline X\ar[dr]^-\beta\ar[dl]_-\alpha & \\ 
X_1\ar@{-->}[rr]_-{p^{-1}}& & X_3,  
}
$$ 
where $\alpha: \overline X\to X_1$ is a projective 
birational morphism from a normal irreducible variety $\overline X$ such that 
$\alpha$ is an isomorphism over $X$. 
By using \eqref{c-eq2.1}, 
we can construct an 
ample divisor $A$ on $\overline X$ with $\Supp A=\overline X\setminus X$. 
Of course, we saw $X$ as a dense Zariski open set of $\overline X$. 
We put $\overline D:=\beta^*D_3+mA$ for some sufficiently 
large positive integer $m$. 
Then $\overline D$ is very ample and $\overline D|_X=D$ by construction. 
Therefore, we see that $D$ is potentially nef. 
\end{proof}

We prepare one more easy lemma on ample divisors. 

\begin{lem}\label{c-lem2.8}
Let $f:X\to Y$ be a projective 
morphism between quasi-projective varieties. 
Let $D$ be an $f$-ample 
Cartier {\em{(}}resp.~$\mathbb Q$-Cartier or $\mathbb R$-Cartier{\em{)}} 
divisor on $X$ and let $H$ be an ample divisor 
on $Y$. 
Then $D+mf^*H$ is an ample 
divisor {\em{(}}resp.~$\mathbb Q$-divisor 
or $\mathbb R$-divisor{\em{)}} for every sufficiently 
large positive 
integer $m$. 
\end{lem}
\begin{proof}
If $D$ is $\mathbb Q$-Cartier, then it is well known that 
$D+mf^*H$ is ample for every sufficiently large positive integer $m$. 
When $D$ is an $f$-ample $\mathbb R$-divisor, 
we can write $D=\sum _{i=1}^k d_i D_i$ where $d_i\in \mathbb R_{>0}$ 
and $D_i$ is an $f$-ample divisor for every $i$. 
Then 
$$
D+mf^*H=\sum _{i=1}^k d_i (D_i+m_if^*H)+\left(m-\sum _{i=1}^k 
m_i d_i\right)f^*H, 
$$
where $m_i$ is a positive integer such that $D_i+m_if^*H$ is ample for every 
$i$. 
Therefore, it is sufficient to prove that $aA+bf^*H$ 
is ample, where $a$ and $b$ are positive real numbers 
and $A$ is an ample divisor on $X$. 
We fix a positive integer $l$ such that $A+lf^*H$ is 
ample. We take a positive real number $c$ such that $0<c\ll 1$ 
and $b-c\in \mathbb Q_{>0}$. 
Then we take a positive rational number $d$ with 
$0<d\ll 1$. 
In this situation, we can write 
$$
aA+bf^*H=\frac{c}{l}(A+lf^*H)+(dA+(b-c)f^*H)
+\left( a-\frac{c}{l}-d\right)A. 
$$ 
This means that $aA+bf^*H$ is ample. 
Thus, we obtain that $D+mf^*H$ 
always can be written as a finite 
$\mathbb R_{>0}$-linear 
combination of ample divisors on $X$ 
for every sufficiently large positive integer $m$. 
This is what we wanted. 
\end{proof}

We give some remarks on potentially nef divisors. 

\begin{rem}\label{c-rem2.9}
(1) Let $X$ be a normal irreducible 
variety and let $D$ be a potentially nef $\mathbb R$-divisor 
on $X$. 
Then $D\cdot C\geq 0$ for every complete 
integral curve $C$ on $X$. 
In particular, $D$ is $\pi$-nef for any proper morphism 
$\pi:X\to S$ onto a variety $S$. 

(2) Let $\pi:X\to S$ be a projective morphism 
from a normal quasi-projective irreducible variety onto 
a quasi-projective variety $S$. 
Let $D$ be a $\pi$-nef $\mathbb R$-divisor 
on $X$, let $A$ be a $\pi$-ample 
$\mathbb R$-divisor on $X$, 
and let $H$ be an ample divisor on $S$. 
Then, by Lemma \ref{c-lem2.8}, 
we can easily see that $D+A+m\pi^*H$ is an ample 
$\mathbb R$-divisor on $X$, 
that is, a finite $\mathbb R_{>0}$-linear 
combination of ample divisors on $X$, for 
every sufficiently large positive integer $m$. 
We note that $D+A$ is a $\pi$-ample $\mathbb R$-divisor 
on $X$.  
\end{rem}
\end{say}

\begin{say}[b-divisors]\label{c-say2.10} 
Let us quickly recall the notion of {\em{b-divisors}} introduced by 
Shokurov (see \cite[Section 1]{shokurov}). 
We note that a b-divisor was originally called a {\em{bi-divisor}} 
in \cite{shokurov}. 

Let $X$ be a normal variety and let $\Div X$ be the space of Weil 
divisors on $X$. A {\em{b-divisor}} on $X$ is an element: 
$$
\mathbf D\in \mathbf{Div} X =\lim_{Y\to X} \Div Y, 
$$ 
where the (projective) limit is taken over all proper birational 
morphism $f: Y\to X$ from a normal variety $Y$ under the 
pushforward homomorphism $f_*: \Div Y\to \Div X$. 
We can define {\em{$\mathbb Q$-b-divisors}} on $X$ similarly.  
If $\mathbf D=\sum d_\Gamma \Gamma$ is a ($\mathbb Q$-)b-divisor 
on a normal variety $X$ and $f:Y\to X$ is a proper birational morphism 
from a normal variety $Y$, then the {\em{trace}} of $\mathbf D$ on $Y$ is 
the ($\mathbb Q$-)divisor 
$$
\mathbf D_Y:=\sum_{\text{$\Gamma$ is a divisor on $Y$}}
d_\Gamma \Gamma. 
$$

\medskip 

The {\em{$\mathbb Q$-Cartier closure}} of a $\mathbb Q$-Cartier 
($\mathbb Q$-)divisor $D$ on a normal variety $X$ is the $\mathbb Q$-b-divisor 
$\overline D$ with trace 
$$
\overline D _Y=f^*D, 
$$ 
where $f:Y\to X$ is a proper birational morphism from 
a normal variety $Y$. 

\begin{defn}[Canonical b-divisor]\label{c-def2.11}
Let $X$ be a normal variety and let 
$\omega$ be a top rational differential 
form of $X$. 
Then $(\omega)$ defines a 
b-divisor $\mathbf K$. We call $\mathbf K$ 
the {\em{canonical b-divisor}} of $X$. 
\end{defn}

We need the following definition for Theorem \ref{c-thm1.2} and 
Conjecture \ref{c-conj1.4}. 

\begin{defn}
[b-potentially nef and b-semi-ample 
$\mathbb Q$-b-divisors, and $\mathbb Q$-b-Cartier 
divisors]\label{c-def2.12}
Let $X$ be a normal variety. 
A $\mathbb Q$-b-divisor $\mathbf D$ of $X$ 
is {\em{b-potentially nef}} 
(resp.~{\em{b-semi-ample}}) if there 
exists a proper birational morphism $X'\to X$ from a normal 
variety $X'$ such that $\mathbf D=\overline {\mathbf D_{X'}}$ 
and $\mathbf D_{X'}$ is potentially nef 
(resp.~semi-ample). 
A $\mathbb Q$-b-divisor $\mathbf D$ of $X$ is {\em{$\mathbb Q$-b-Cartier}} 
if there is a proper birational morphism $X'\to X$ from a normal 
variety $X'$ such that $\mathbf D=\overline{\mathbf D_{X'}}$. 
\end{defn}

\begin{lem}\label{d-lem2.13} 
Let $\mathbf D_1$ and $\mathbf D_2$ be $\mathbb Q$-b-divisors on 
a normal variety $X$. Assume that 
$$
\mathbf D_1=\mathbf D_2+r\overline{(g)}
$$ 
holds, where $\overline{(g)}$ is a $\mathbb Q$-Cartier 
closure of a principal Cartier divisor $(g)$ associated 
to $g\in \Gamma (X, \mathcal K^*_X)$ and $r$ is a rational 
number. In this situation, if $\mathbf D_1$ is $\mathbb Q$-b-Cartier, 
that is, $\mathbf D_1=\overline {\mathbf D_{1X'}}$ for some 
proper birational morphism $X'\to X$ from 
a normal variety $X'$, then $\mathbf D_2 =\overline 
{\mathbf D_{2X'}}$ holds. 
\end{lem}
\begin{proof}
It is obvious by definition. 
\end{proof}

For more details on b-divisors, see, 
for example, \cite[2.3.2 b-divisors]{corti}. 
\end{say}

\begin{say}[Simple normal crossing pairs]\label{d-say2.14}
In this paper, we will mainly treat simple normal crossing 
pairs. 

\begin{defn}\label{d-def2.15}
We say that the pair $(X, D)$ is {\em{simple normal crossing}} 
at a point $a\in X$ if $X$ has a Zariski open neighborhood 
$U$ of $a$ that can be embedded in a smooth 
variety $Y$, where $Y$ has a regular system of parameters 
$(x_1, \ldots, x_p, y_1, \ldots, y_r)$ at $a=0$ in 
which $U$ is defined by a monomial equation 
$$
x_1\cdots x_p=0
$$ 
and 
$$
D=\sum _{i=1}^r \alpha_i (y_i=0)|_U, \quad 
\alpha_i\in \mathbb R.
$$ 
We say that $(X, D)$ is a {\em{simple normal crossing pair}} 
if it is simple normal crossing at every point of $X$. 
If $(X, 0)$ is a simple normal crossing pair, then 
$X$ is called a {\em{simple normal crossing variety}}. 
If $(X, D)$ is a simple normal crossing pair and 
$D$ is reduced, then $D$ is called a {\em{simple 
normal crossing divisor}} on $X$. 
Let $(X, D)$ be a simple normal crossing pair 
such that $D=D^{\leq 1}$ holds. 
Then it is easy to see that $(X, D)$ is sub slc 
in the sense of Definition \ref{c-def2.3}. 
In this situation, we simply say that $W$ is 
a {\em{stratum}} of $(X, D)$ if $W$ is 
an slc stratum of $(X, D)$ in the sense of Definition 
\ref{c-def2.3}. We note that 
a stratum of a simple normal crossing variety $X$ 
means a stratum of a simple normal crossing pair 
$(X, 0)$. 
\end{defn}

Let $X$ be a simple normal crossing variety and 
let $\Delta$ be an $\mathbb R$-divisor 
on $X$ such that no irreducible component of $\Supp \Delta$ 
is contained in 
the singular locus of $X$ and that $K_X+\Delta$ is $\mathbb R$-Cartier. 
Let $Z$ be a closed subset $Z\subsetneq X$ such that 
$Z$ contains no stratum of $X$. 
Then we put 
$$
\mld _Z(X, \Delta):=\mld _{\nu^{-1}(Z)}(X^\nu, \Theta),  
$$ 
where $\nu:X^\nu\to X$ is the normalization and $K_{X^\nu}+\Theta
=\nu^*(K_X+\Delta)$, that is, 
$\Theta$ is the sum of the inverse images of $\Delta$ and the singular 
locus of $X$. We call $\mld _Z(X, \Delta)$ the {\em{minimal log discrepancy}} 
of $(X, \Delta)$ in a closed subset $Z$. We will 
use it in Theorem \ref{c-thm5.1} below. 

\medskip 

We close this section with a useful lemma. 
We will often use it in the subsequent sections 
without mentioning it explicitly. We note that the classical topology 
means the Euclidean topology in Lemma \ref{d-lem2.16}. 

\begin{lem}\label{d-lem2.16}
Let $(X, D)$ be a simple normal crossing pair with $\dim X=n$ and 
let $f:X\to Z$ be a morphism 
onto an $m$-dimensional 
smooth irreducible variety $Z$. 
Assume that every stratum of $(X, \Supp D)$ is smooth 
over $Z$. 
Let $a\in X$ be any closed point. 
Then we have the following local analytic description 
of $f:(X, D)\to Z$ in a neighborhood 
of $a\in X$. 
\begin{itemize}
\item[(i)] $U$ and $V$ are open neighborhoods of 
$a\in X$ and $f(a)\in Z$ in the classical 
topology, respectively. 
\item[(ii)] $W$ is an open set of $\mathbb C^{n+1}$ in 
the classical topology. 
\item[(iii)] $(z_1, \ldots, z_m)$ and $(z_1, \ldots, z_{n+1})$ are systems of 
local analytic coordinates of $V$ and $W$, respectively. 
\item[(iv)] $p:W\to V$ is the projection given by $(z_1, \ldots, z_{n+1})
\mapsto (z_1, \ldots, z_m)$. 
\item[(v)] $U$ is defined by a monomial equation $z_{m+1}\cdots z_{m+p}=0$ 
in $W$ and $a=(0, \ldots, 0)\in W$. 
\item[(vi)] $D|_U=\sum _{i=1}^r\alpha_i (z_{m+p+i}=0)|_U$ with 
$\alpha_i\in \mathbb R$. 
\item[(vii)] $f|_U=p\circ \iota$, where 
$\iota$ is the natural closed embedding $U\hookrightarrow 
W$. 
\end{itemize}
$$
\xymatrix{U \ar@{^{(}->}[r]^-\iota\ar[dr]_-{f|_U}& W \ar[d]^-p\\ 
& V
}
$$
Let $\rho:Z'\to Z$ be a morphism from a smooth irreducible variety $Z'$. 
We put $X'=X\times _Z Z'$ and consider the following 
commutative diagram. 
$$
\xymatrix{
X \ar[d]_-f& X'\ar[d]^-{f'} \ar[l]_-{\rho'}\\ 
Z & Z'\ar[l]^-{\rho}
}
$$ 
Let $D'$ be the pull-back of $D$ on $X'$ by $\rho'$. 
Then we can easily see that $(X', D')$ is a simple normal 
crossing pair and every stratum of $(X', \Supp D')$ is smooth 
over $Z'$ by 
the above local analytic description of $f:(X, D)\to Z$. 
\end{lem}
\begin{proof}
By definition, $X$ is Zariski locally a simple normal crossing divisor 
on a smooth variety $Y$ in a neighborhood of 
$a\in X$ (see Definition \ref{d-def2.15}). 
By taking a small open set $W$ of $Y$ containing 
$a$ in the classical topology, $f|_U:U\to Z$, where 
$U:=X\cap W$, 
extends to a holomorphic map $W\to Z$ (see, for example, 
\cite[0.22.~Corollary 2]{fischer}). 
Since every stratum of $(X, \Supp D)$ is smooth over $Z$ by assumption, 
we obtain the desired local analytic description 
by shrinking $W$ suitably around $a$ and 
taking a small open neighborhood $V$ of $f(a)$ in $Z$ in the classical 
topology. By this local analytic description, 
we can easily see that $f: (X, D)\to Z$ behaves well under base change. 
\end{proof}
\end{say}

\section{Variations of mixed Hodge structure}\label{c-sec3} 

In this section, let us quickly recall the main result of \cite{fujino-fujisawa} 
(see also \cite{ffs}). We note that Theorem 
\ref{c-thm3.1} is the main ingredient of 
Theorem \ref{c-thm1.2}. 
Theorem \ref{c-thm3.1} follows from the 
theory of variations of mixed Hodge structure on cohomology 
with compact support. 

\begin{thm}[{\cite[Theorems 7.1 and 7.3]{fujino-fujisawa}}]
\label{c-thm3.1}
Let $(X, D)$ be a simple normal crossing pair such that 
$D$ is reduced and let $f:X\to Y$ be a projective 
surjective morphism onto a smooth variety $Y$ such that 
every stratum of $(X, D)$ is dominant onto $Y$. 
Assume that there exists a simple normal crossing 
divisor $\Sigma_Y$ on $Y$ such that every stratum of $(X, D)$ is smooth 
over $Y^*=Y\setminus \Sigma_Y$. 
Then we have 
\begin{itemize}
\item[(i)] $f_*\omega_{X/Y}(D)$ is a locally free sheaf on $Y$. 
\end{itemize} 
We further assume that all the local monodromies 
on the local system 
$R^d(f|_{X^*})_*\iota_!\mathbb Q_{X^*\setminus 
D^*}$ around $\Sigma_Y$ are unipotent, where $d=\dim X-\dim Y$, 
$X^*=f^{-1}(Y^*)$, $D^*=D|_{X^*}$, and $\iota:X^*\setminus D^*
\hookrightarrow X^*$. 
Then we have the following properties. 
\begin{itemize}
\item[(ii)] $\left(f_*\omega_{X/Y}(D)\right)|_V$ 
is a nef locally free sheaf 
on $V$, where $V$ is any complete subvariety of $Y$. 
\item[(iii)] Let $\rho:Y'\to Y$ be a morphism 
from a smooth 
variety $Y'$ such that $\rho^{-1}(\Sigma_Y)$ is a simple 
normal crossing divisor on $Y'$. 
Let $(X', D')$ be a simple normal crossing pair 
and let $f':X'\to Y'$ be a projective 
surjective morphism onto $Y'$ such that 
$f': (X', D')\to Y'$ is nothing but the base change of $f:(X, D)\to 
Y$ by $\rho:Y'\to Y$ over $Y\setminus \Sigma_Y$ and 
that every stratum of $(X', D')$ is dominant onto $Y'$. 
Then there exists a natural isomorphism 
$\rho^*(f_*\omega_{X/Y}(D))\simeq 
f'_*\omega_{X'/Y'}(D')$ of locally free sheaves 
which extends the base change 
isomorphism over $Y\setminus \Sigma_Y$. 
\end{itemize}
\end{thm}

We sketch the proof of Theorem \ref{c-thm3.1} 
for the reader's convenience. 
The details are contained in \cite{fujino-fujisawa} (see also 
Section \ref{c-sec13} for some supplementary remarks). 

\begin{proof}[Sketch of proof]
By \cite[Theorem 4.15]{fujino-fujisawa}, 
the local system 
$R^d(f|_{X^*})_*\iota_!\mathbb Q_{X^*\setminus D^*}$ 
underlies a graded polarizable 
variation of $\mathbb Q$-mixed Hodge structure on $Y^*$. 
Moreover, it is admissible (see, for example, 
\cite[Definition 3.11]{fujino-fujisawa}). 
We put 
$$\mathcal V^d_{Y^*}=R^d(f|_{X^*})_*\iota_!\mathbb Q_{X^*\setminus 
D^*}\otimes \mathcal O_{Y^*}. 
$$ 
Let 
$$
\cdots \subset F^{p+1}(\mathcal V^d_{Y^*})
\subset F^p(\mathcal V^d_{Y^*})
\subset F^{p-1}(\mathcal V^d_{Y^*})\subset \cdots
$$ 
be the Hodge filtration. 
By \cite[Theorem 7.3 (b)]{fujino-fujisawa}, 
we obtain that $f_*\omega_{X/Y}(D)$ is 
isomorphic to the upper canonical extension 
of 
$$
\left(\Gr^0_F(\mathcal V^d_{Y^*})\right)^*
=\mathcal {H}om_{\mathcal O_{Y^*}}\!
\left(\Gr^0_F(\mathcal V^d_{Y^*}), \mathcal O_{Y^*}\right). 
$$ 
In particular, $f_*\omega_{X/Y}(D)$ is a locally free sheaf on $Y$. 
For the details of the (upper) canonical extensions of Hodge 
bundles, see \cite[Remark 7.4]{fujino-fujisawa}. 
Hence, we get (i). 
When all the local monodromies on the local 
system $R^d(f|_{X^*})_*\iota_!\mathbb Q_{X^*\setminus 
D^*}$ around $\Sigma_Y$ are unipotent, 
$f_*\omega_{X/Y}(D)$ is the canonical extension of 
$$
\left(\Gr^0_F(\mathcal V^d_{Y^*})\right)^*
=\mathcal {H}om_{\mathcal O_{Y^*}}
\!\left(\Gr^0_F(\mathcal V^d_{Y^*}), \mathcal O_{Y^*}\right). 
$$
Therefore $f_*\omega_{X/Y}(D)\simeq 
\left(\Gr^0_F(\mathcal V^d_Y)\right)^*$, 
where $\Gr^0_F(\mathcal V^d_Y)$ is the canonical 
extension of 
$$
\Gr^0_F(\mathcal V^d_{Y^*})=F^0(\mathcal V^d_{Y^*})
/F^1(\mathcal V^d_{Y^*}). 
$$ 
Note that $\Gr^0_F(\mathcal V^d_Y)$ is isomorphic to $R^df_*\mathcal 
O_X(-D)$ by 
\cite[Theorem 7.1 (2)]{fujino-fujisawa}. 
Thus we obtain that $\left(f_*\omega_{X/Y}(D)\right)|_V$ is 
a nef locally free sheaf on $V$ for any complete subvariety $V$ of $Y$ (see
\cite[Remark 5.22, Corollary 5.23, and Theorem 7.1 (4)]
{fujino-fujisawa}). So we get (ii). 
As we saw above, $f_*\omega_{X/Y}(D)$ can be 
characterized by using canonical extensions of Hodge 
bundles. We note that canonical extensions 
of Hodge bundles behave well 
under pull-back by 
$\rho:Y'\to Y$ such that $\rho^{-1}(\Sigma_Y)$ is a simple 
normal crossing divisor on $Y'$. 
More precisely, we see that the pull-back of 
$\left(\Gr^0_F(\mathcal V^d_Y)\right)^*$ is isomorphic 
to the canonical extension of the pull-back 
of $\left(\Gr^0_F(\mathcal V^d_{Y^*})\right)^*$. 
Therefore, we get a natural 
isomorphism $\rho^*\left(f_*\omega_{X/Y}(D)\right)
\simeq f'_*\omega_{X'/Y'}(D')$, 
which is nothing but (iii). 
\end{proof}
\begin{rem}\label{c-rem3.2} 
In Theorem 
\ref{c-thm3.1}, the 
same results hold for $R^if_*\omega_{X/Y}(D)$ for 
every $i$. We only treat the case where $i=0$ since 
it is sufficient for our purposes in this paper. 
For the details of the cases where $i\ne 0$, 
see \cite{fujino-fujisawa}. 
\end{rem}

We recommend the interested reader to see \cite{fujino-fujisawa} 
for the details of Theorem \ref{c-thm3.1}. 

\section{Basic slc-trivial fibrations}\label{c-sec4} 
In this section, we introduce the notion of {\em{$($pre-$)$basic slc-trivial 
fibrations}} and define {\em{discriminant $\mathbb Q$-b-divisors}} and 
{\em{moduli $\mathbb Q$-b-divisors}} for (pre-)basic 
slc-trivial fibrations. 

\medskip 

Let us start with the definition of (pre-)basic slc-trivial fibrations. 

\begin{defn}[Basic slc-trivial fibration]\label{c-def4.1}
A {\em{pre-basic slc-trivial fibration}} $f:(X, B)\to Y$ consists of 
a projective surjective morphism 
$f:X\to Y$ and a simple normal crossing pair $(X, B)$ satisfying 
the following properties: 
\begin{itemize}
\item[(1)] $Y$ is a normal irreducible variety, 
\item[(2)] every stratum of $X$ is dominant onto $Y$ and 
$f_*\mathcal O_X\simeq \mathcal O_Y$, 
\item[(3)] $B$ is a $\mathbb Q$-divisor such that $B=B^{\leq 1}$ holds 
over 
the generic point of $Y$, and 
\item[(4)] there exists 
a $\mathbb Q$-Cartier $\mathbb Q$-divisor $D$ on $Y$ such that 
$$
K_X+B\sim _{\mathbb Q}f^*D. 
$$ 
\end{itemize}
If a pre-basic slc-trivial fibration $f:(X, B)\to Y$ also satisfies 
\begin{itemize}
\item[(5)] $\rank f_*\mathcal O_X(\lceil -B^{<1}\rceil)=1$, 
\end{itemize}
then it is called a {\em{basic slc-trivial fibration}}. 
\end{defn}

Before we study basic slc-trivial fibrations, we make a remark on 
lc-trivial fibrations and klt-trivial fibrations for the reader's convenience. 

\begin{rem}[Lc-trivial fibrations and klt-trivial fibrations]\label{c-rem4.2} 
Let $f:(X, B)\to Y$ be a basic slc-trivial fibration. 
Roughly speaking, if $X$ is irreducible and 
$(X, B)$ is sub log canonical (resp.~sub kawamata log 
terminal) over the generic point of $Y$, 
then $f:(X, B)\to Y$ is called an {\em{lc-trivial fibration}} 
(resp.~a {\em{klt-trivial fibration}}). 
We note that a klt-trivial fibration is called an lc-trivial fibration 
in \cite{ambro3} (see \cite[Definition 2.1]{ambro3}). 
For the details, see \cite[Definitions 3.1 and 3.2]{fujino-gongyo}.   
\end{rem}

The notion of basic slc-trivial fibrations is a 
generalization of that of lc-trivial fibrations. 

\begin{say}[Induced (pre-)basic slc-tirival fibrations]\label{c-say4.3}
Let $f:(X, B)\to Y$ be a (pre-)basic slc-trivial fibration 
and let $\sigma:Y'\to Y$ be a generically finite surjective 
morphism from a normal irreducible variety $Y'$. 
Then we have an {\em{induced {\em{(}}pre-{\em{)}}basic slc-trivial fibration}} 
$f':(X', B_{X'})\to Y'$, where 
$B_{X'}$ is defined by $\mu^*(K_X+B)=K_{X'}+B_{X'}$, with 
the following commutative diagram: 
$$
\xymatrix{
   (X', B_{X'}) \ar[r]^{\mu} \ar[d]_{f'} & (X, B)\ar[d]^{f} \\
   Y' \ar[r]_{\sigma} & Y, 
} 
$$
where $X'$ coincides with 
$X\times _{Y}Y'$ over a nonempty Zariski open set of $Y'$. 
More precisely, $X'$ is a simple normal crossing variety with a morphism 
$X'\to X\times _Y Y'$ that is an isomorphism over 
a nonempty Zariski open set of $Y'$ such that 
$X'$ is projective over $Y'$ and that every stratum of $X'$ is dominant onto 
$Y'$. 

\begin{lem}\label{c-lem4.4}
Let $f'_i: (X'_i, B_{X'_i})\to Y'$ be an induced {\em{(}}pre-{\em{)}}basic 
slc-trivial fibration for $i=1, 2$. 
Then there exist an induced {\em{(}}pre-{\em{)}}basic slc-trivial fibration $f'_3: 
(X'_3, B_{X'_3})\to Y'$ and a commutative diagram 
$$
\xymatrix{& X'_3\ar[dl]_-{p_1}\ar[dr]^-{p_2}& 
\\ X'_1 \ar[dr]_-{f'_1}&& X'_2\ar[dl]^-{f'_2} \\ 
&Y'&
}
$$
such that $p_i$ induces a birational correspondence between 
each stratum of $X'_3$ and $X'_i$ and that 
$K_{X'_3}+B_{X'_3}=p^*_i(K_{X'_i}+B_{X'_i})$ holds for $i=1, 2$. 
\end{lem}
\begin{proof}
By definition, there exists a nonempty Zariski open set 
$U$ of $Y'$ such that $X'_1$ and $X'_2$ coincide with 
$X\times _YY'$ over $U$. 
By \cite[Theorem 1.4]{bierstone-vera}, 
we can take a common partial resolution $X'_3$ of $X'_1$ and $X'_2$, 
which coincides with $X\times _Y Y'$ over $U$,  
with the desired properties. 
\end{proof}
\end{say}

\begin{say}[Discriminant and 
moduli $\mathbb Q$-b-divisors]\label{c-say4.5} 
Let $f:(X, B)\to Y$ be a (pre-)basic slc-trivial fibration as in Definition \ref{c-def4.1}. 
Let $P$ be a prime divisor on $Y$. 
By shrinking $Y$ around the generic point of $P$, 
we assume that $P$ is Cartier. We set 
$$
b_P:=\max \left\{t \in \mathbb Q\, \left|\, 
\begin{array}{l}  {\text{$(X, B+tf^*P)$ is sub slc over}}\\
{\text{the generic point of $P$}} 
\end{array}\right. \right\} 
$$ 
and 
set $$
B_Y=\sum _P (1-b_P)P, 
$$ 
where $P$ runs over prime divisors on $Y$. Equivalently, we have 
$$
b_P=\max \left\{t \in \mathbb Q\, \left|\, 
\begin{array}{l}  {\text{$(X^\nu, \Theta+t\nu^*f^*P)$ is sub log canonical}}\\
{\text{over the generic point of $P$}} 
\end{array}\right. \right\},  
$$ 
where $\nu:X^\nu\to X$ is the normalization and 
$K_{X^\nu}+\Theta=\nu^*(K_X+B)$, that is, 
$\Theta$ is the sum of the inverse images of $B$ and the singular 
locus of $X$. 
Then it is easy to  see that 
$B_Y$ is a well-defined $\mathbb Q$-divisor 
on $Y$ and is called the {\em{discriminant 
$\mathbb Q$-divisor}} of $f:(X, B)\to Y$. We set 
$$
M_Y=D-K_Y-B_Y
$$ 
and call $M_Y$ the {\em{moduli $\mathbb Q$-divisor}} of $f:(X, B)\to Y$. 
By definition, we have 
$$
K_X+B\sim _{\mathbb Q}f^*(K_Y+B_Y+M_Y). 
$$

Let $\sigma:Y'\to Y$ be a proper birational morphism 
from a normal variety $Y'$ and let $f':(X', B_{X'})\to Y'$ be 
an induced (pre-)basic slc-trivial fibration 
by $\sigma:Y'\to Y$.  
We can define $B_{Y'}$, $K_{Y'}$ and $M_{Y'}$ such that 
$\sigma^*D=K_{Y'}+B_{Y'}+M_{Y'}$, 
$\sigma_*B_{Y'}=B_Y$, $\sigma _*K_{Y'}=K_Y$ 
and $\sigma_*M_{Y'}=M_Y$. We note that 
$B_{Y'}$ is independent of the choice of $(X', B_{X'})$, 
that is, $B_{Y'}$ is well defined,  
by Lemma \ref{c-lem4.4} above and Lemma \ref{c-lem4.6} below. 
Hence 
there exist a unique $\mathbb Q$-b-divisor $\mathbf B$ 
such that 
$\mathbf B_{Y'}=B_{Y'}$ for every $\sigma:Y'\to Y$ and a unique 
$\mathbb Q$-b-divisor $\mathbf M$ such that $\mathbf M_{Y'}=M_{Y'}$ for 
every $\sigma:Y'\to Y$. 
Note that $\mathbf B$ is called 
the {\em{discriminant $\mathbb Q$-b-divisor}} and 
that $\mathbf M$ is called 
the {\em{moduli $\mathbb Q$-b-divisor}} associated to $f:(X, B)\to Y$. 
We sometimes simply say that $\mathbf M$ is 
the {\em{moduli part}} of $f:(X, B)\to Y$. 
\end{say}

The following lemma has already been used in the 
definition of discriminant $\mathbb Q$-b-divisors in \ref{c-say4.5}. 

\begin{lem}\label{c-lem4.6}
Let $f_i: (X_i, B_i)\to Y$ be a pre-basic 
slc-trivial fibration for 
$i=1, 2$. 
Assume that there exists a morphism 
$p: X_2\to X_1$ over $Y$ 
which induces a birational correspondence between 
each irreducible component of $X_1$ and $X_2$ such that 
$K_{X_2}+B_2=p^*(K_{X_1}+B_1)$ holds. 
Then $f_1: (X_1, B_1)\to Y$ and $f_2: (X_2, B_2)\to Y$ 
induce the same discriminant $\mathbb Q$-divisor 
on $Y$. 
\end{lem}
\begin{proof} 
Let $P$ be a prime divisor on $Y$. 
We may assume that $P$ is Cartier by shrinking $Y$ around $P$ as 
above. 
Since $(X_1, B_1+tf_1^*P)$ is sub slc over the 
generic point of $P$ if and only if 
$(X_2, B_2+tf_2^*P)$ is sub slc over the generic point of $P$ 
for every $t\in \mathbb Q$. 
Therefore, $f_1: (X_1, B_1)\to Y$ and $f_2: (X_2, B_2)\to Y$ 
induce the same discriminant $\mathbb Q$-divisor 
on $Y$ by the definition of discriminant $\mathbb Q$-divisors. 
\end{proof}

When $(X, \Supp B+\Supp f^*P)$ is a simple normal crossing pair, 
we can explicitly write down $b_P$. 

\begin{rem}[{\cite[Theorem 2]{kawamata} and 
\cite[Remark 3.1]{ambro1}}]\label{c-rem4.7} 
Let $f:(X, B)\to Y$ be a pre-basic slc-trivial 
fibration and let $P$ be a prime divisor on $Y$. 
By shrinking $Y$ around the generic point of $P$, we assume that 
$P$ is Cartier. 
If $(X, \Supp B+\Supp f^*P)$ is a simple normal crossing pair and the 
irreducible decomposition $f^*P=\sum _j w_j Q_j$ satisfies 
$f(Q_j)=P$ for every $j$, then we can explicitly write 
\begin{equation}\label{c-eq4.1}
b_P=\min_j \frac{1-d_j}{w_j}, 
\end{equation} 
where $d_j=\mult _{Q_j}B$ for every $j$, by direct calculations. 
Equivalently, we have 
\begin{equation}\label{c-eq4.2}
\mult _P B_Y=1-b_P=\max_j \frac{d_j+w_j-1}{w_j}. 
\end{equation} 
Note that \eqref{c-eq4.2} plays a crucial role when we compare 
the minimal log discrepancy of $(X, B)$ with that of 
$(Y, B_Y)$. 
See, for example, the proof of Theorem \ref{c-thm5.1} below. 
\end{rem}

We give a small remark on the definition of 
discriminant $\mathbb Q$-divisors. 

\begin{rem}\label{c-rem4.8}
Let $f:(X, B)\to Y$ be a pre-basic 
slc-trivial fibration. We do not need condition (4) in 
Definition \ref{c-def4.1} in order to 
define the discriminant $\mathbb Q$-divisor $B_Y$. 
\end{rem}

We will use condition (5) in Definition \ref{c-def4.1} 
to relate the moduli $\mathbb Q$-divisor 
$M_Y$ with some Hodge bundles (see Proposition 
\ref{d-prop6.3} below). 

\begin{rem}\label{d-rem4.9} 
Let $f:(X, B)\to Y$ be a pre-basic slc-trivial fibration. 
We start with $K_X+B\sim _{\mathbb Q} f^*D^\dag$ 
for some $\mathbb Q$-Cartier $\mathbb Q$-divisor $D^\dag$ 
on $Y$. 
Note that $D\sim _{\mathbb Q}D^\dag$ holds since 
$f_*\mathcal O_X\simeq \mathcal O_Y$. 
In this setting, we put 
$$
M^\dag_Y:=D^\dag-K_Y-B_Y 
$$ 
and obtain the moduli $\mathbb Q$-divisor 
$\mathbf M^\dag$ associated to $f:(X, B)\to Y$ 
as in \ref{c-say4.5} above.  
We note that $K_Y$ is well defined modulo linear equivalence 
and the discriminant $\mathbb Q$-b-divisor 
$\mathbf B$ is independent of $D$ and $D^\dag$ 
(see Remark \ref{c-rem4.8}). 
Therefore, we have $g\in \Gamma (Y, \mathcal K^*_Y)$ and 
a rational number $r$ such that 
$$
\mathbf M=\mathbf M^\dag+r\overline{(g)}
$$
holds. Therefore, if $\mathbf M^\dag=\overline {\mathbf M^\dag_{Y'}}$ 
holds 
for some proper birational morphism $\sigma: Y'\to Y$ 
from a normal variety $Y'$, 
then $\mathbf M=\overline {\mathbf M_{Y'}}$ holds true by 
Lemma \ref{d-lem2.13}. 
\end{rem}

We prepare an elementary finite base change 
formula, which will be used in Sections \ref{c-sec8} and \ref{c-sec9}. 

\begin{lem}[{\cite[Theorem 3.2]{ambro1}}]\label{d-lem4.10}
Let us consider a commutative diagram: 
$$
\xymatrix{
   (X', B_{X'}) \ar[r]^{\mu} \ar[d]_{f'} & (X, B)\ar[d]^{f} \\
   Y' \ar[r]_{\sigma} & Y, 
} 
$$
where $f:(X, B)\to Y$ is a pre-basic slc-trivial fibration, 
$\sigma:Y'\to Y$ is a finite surjective 
morphism of normal irreducible varieties, 
and $f':(X', B_{X'})\to Y'$ is an induced pre-basic 
slc-trivial fibration. 
Then $\sigma^*(K_Y+B_Y)=K_{Y'}+B_{Y'}$ holds, 
where $B_Y$ {\em{(}}resp.~$B_{Y'}${\em{)}} is 
the discriminant $\mathbb Q$-divisor of 
$f:(X, B)\to Y$ {\em{(}}resp.~$ f':(X', B_{X'})\to Y'${\em{)}}. 
\end{lem}

\begin{rem}\label{d-rem4.11}
In Lemma \ref{d-lem4.10}, $K_Y+B_Y$ is not necessarily 
$\mathbb Q$-Cartier. However, 
we can define $\sigma^*(K_Y+B_Y)$ since $\sigma$ is 
a finite surjective morphism between normal varieties. 
\end{rem} 
 
\begin{proof}[Proof of Lemma \ref{d-lem4.10}]
Without loss of generality, we may assume that 
$Y$ and $Y'$ are both smooth by shrinking $Y$ suitably. 
Let $P'$ be a prime divisor on $Y'$. 
We put $P=\sigma(P')$ and $w=\mult _{P'}\sigma^*P$. 
Then it is sufficient to see $wb_P=b_{P'}$ because 
$\sigma^*(K_Y+P)=K_{Y'}+P'$ holds in a neighborhood of 
the generic point of $P'$. 
By the definition of discriminant $\mathbb Q$-divisors, 
we may assume that $X$ is smooth by replacing 
$(X, B)$ with $(X^\nu, \Theta)$, where $\nu:X^\nu \to X$ 
is the normalization with 
$K_{X^\nu}+\Theta=\nu^*(K_X+B)$ as usual. 

We take any $c\leq b_P$. Then 
$K_X+B+cf^*P$ is sub log canonical 
over the generic point of $P$. Therefore, 
$K_{X'}+B_{X'}+c(f\circ \mu)^*P=K_{X'}+B_{X'}+c(f')^*\sigma^*P$ is 
sub log canonical over the generic point of $P'$. 
Since $\sigma^*P=wP'$, 
$K_{X'}+B_{X'}+cw(f')^*P'$ is sub log canonical 
over the generic point of $P'$. 
This implies that $cw\leq b_{P'}$. 
Thus we get $b_{P'}\geq wb_P$. 

We take any $c\geq b_P$. By taking a suitable 
birational modification of $X$, 
we may assume that there exists a prime divisor $E$ on $X$ such that 
$a(E, X, B+cf^*P)\leq -1$ and $f(E)=P$. 
Since $X'$ is a resolution of $X\times _Y Y'$, we can find a prime divisor 
$E'$ on $X'$ such that $\mu(E')=E$, $f'(E')=P'$, and 
$a(E', X', B_{X'}+cw(f')^*P')
=a(E', X', B_{X'}+c(f\circ \sigma)^*P)\leq -1$. 
Therefore, we get $cw\geq b_{P'}$. 
This implies $wb_P\geq b_{P'}$. 

Thus we obtain $wb_P=b_{P'}$. This is what we wanted, that is, 
$\sigma^*(K_Y+B_Y)=K_{Y'}+B_{Y'}$. 
\end{proof}

We close this section with the following easy lemma. 

\begin{lem}\label{d-lem4.12}
Let $f:(X, B)\to Y$ be a {\em{(}}pre-{\em{)}}basic slc-trivial 
fibration. Then there 
exists a {\em{(}}pre-{\em{)}}basic 
slc-trivial fibration $\overline f: (\overline X, \overline B)
\to \overline Y$ such that 
\begin{itemize}
\item[(i)] $\overline Y$ is a normal complete variety which 
contains $Y$ as a dense Zariski open set, and 
\item[(ii)] the restriction of $\overline f: (\overline X, \overline B)
\to \overline Y$ to $Y$ coincides with 
$f:(X, B)\to Y$. 
\end{itemize}
\end{lem}
\begin{proof}
We can write $K_X+B+r(\varphi)=f^*D$ for some 
$\mathbb Q$-Cartier $\mathbb Q$-divisor $D$ on $Y$, 
$r\in \mathbb Q$, and $\varphi\in \Gamma (X, \mathcal K^*_X)$. 
We take a normal complete irreducible variety $\overline Y$ which 
contains $Y$ as a dense Zariski open set. 
By taking a suitable birational modification 
(see \cite[Th\'eor\`eme (5.2.2)]{raynaud-g}), we may assume that 
there exists a $\mathbb Q$-Cartier $\mathbb Q$-divisor 
$\overline D$ on $\overline Y$ with $\overline D|_Y=D$. 
By using \cite[Theorem 1.4]{bierstone-vera}, 
we can construct a complete 
simple normal crossing variety $\overline X$ which contains $X$ 
as a dense Zariski open set and 
a projective morphism $\overline f: \overline X\to \overline Y$ which is an 
extension of $f:X\to Y$. 
By \cite[Theorem 1.4]{bierstone-vera}, we may assume that 
every stratum of $\overline X$ is dominant onto $\overline Y$ and 
that $\Sigma:=\overline X\setminus X$ is a simple 
normal crossing divisor on $\overline X$. 
We consider the Stein factorization 
$$
\overline f: \overline X\longrightarrow\overline Z:=\Spec _{\overline Y} 
\overline f_* \mathcal O_{\overline X} \overset{\alpha}{\longrightarrow} 
\overline Y
$$
of $\overline f: \overline X\to \overline Y$. Note that 
$\alpha$ is an isomorphism over $Y$ by construction. 
Since every irreducible component of $X$ is dominant 
onto $\overline Y$, $\overline Z$ is an irreducible variety. 
Therefore, by Zarisiki's main theorem, 
$\alpha$ is an isomorphism. 
This means that $\overline f_*\mathcal O_{\overline X}
\simeq \mathcal O_{\overline Y}$ holds. We may further assume that 
$(\overline X, \Sigma+\Supp B')$ is a simple normal crossing 
pair, where $B'$ is the closure of $B$ on $\overline X$. 
We put 
$\overline B:=\overline f^*\overline D-K_{\overline X} 
-r(\varphi)$. 
Note that we can see $\varphi$ as an element of $\Gamma (\overline 
X, \mathcal K^*_{\overline X})$. 
Then $\overline f: (\overline X, \overline B)\to \overline Y$ satisfies 
the desired properties. 
\end{proof}

Lemma \ref{d-lem4.12} is indispensable 
for the proof of Theorem \ref{c-thm1.2} (ii). 
 
\section{Inversion of adjunction}\label{c-sec5} 

In this section, we prove 
the following theorem, which is 
essentially the same as \cite[Theorem 3.1]{ambro3}. 
Although we do not use Theorem \ref{c-thm5.1} explicitly in this paper, 
the arguments in the proof of Theorem \ref{c-thm5.1} below 
may help the reader understand the proof of 
Theorem \ref{c-thm1.7} in Section \ref{c-sec11}. 

\begin{thm}[Inversion of adjunction]\label{c-thm5.1}
Let $f:(X, B)\to Y$ be a pre-basic 
slc-trivial fibration such that 
$\mathbf {K}+\mathbf{B}=
\overline{K_Y+B_Y}$, where $\mathbf K$ is the canonical 
b-divisor of $Y$ and $\mathbf B$ is 
the discriminant $\mathbb Q$-b-divisor of $f:(X, B)\to Y$. 
Then there is a positive integer $N$ such that 
$$
\frac{1}{N}\mld _{f^{-1}(Z)}(X, B)
\leq \mld _Z(Y, B_Y)\leq \mld _{f^{-1}(Z)} (X, B)
$$ 
for every closed subset $Z\subsetneq Y$. 
\end{thm}
\begin{proof} 
We take a proper birational morphism 
$\sigma:Y'\to Y$ from a smooth variety $Y'$ such that 
$\sigma^{-1}(Z)$ is a divisor on $Y'$ and 
that $\Supp\sigma^{-1}(Z)\cup \Supp \mathbf B_{Y'}$ is included 
in a simple normal crossing divisor $\Sigma_{Y'}$. 
Let $f': (X', B_{X'})\to Y'$ be an induced 
pre-basic slc-trivial fibration with 
$$
\xymatrix{(X, B) \ar[d]_-f& (X', B_{X'})\ar[d]^-{f'}\ar[l]_\mu\\ 
Y & Y'. \ar[l]^-\sigma
}
$$
We may further assume that $\Supp B_{X'}\cup \Supp (f')^*\Sigma_{Y'}$ 
is included in a simple normal crossing divisor $\Sigma_{X'}$. 
Let $\Sigma_{Y'}=\sum _l P_l$ (resp.~$\Sigma_{X'}=\sum 
_j Q_j$) be the irreducible decomposition of $\Sigma_{Y'}$ (resp.~$\Sigma_{X'}$). 
We may assume that there exists $j_0$ such that 
$Q_{j_0}\subset (\sigma\circ f')^{-1}(Z)$ and 
$a(Q_{j_0}, X, B)+1=\mld _{f^{-1}(Z)}(X, B)$ when 
$\mld _{f^{-1}(Z)}(X, B)\geq 0$. 
When $\mld _{f^{-1}(Z)}(X, B)=-\infty$, we assume that 
$a(Q_{j_0}, X, B)+1<0$ holds. 
If we need, we take more blow-ups of $Y'$ and may assume that 
$f'(Q_{j_0})=P_{l_0}$ for some $l_0$ with the aid of the 
flattening theorem (see \cite[Th\'eor\`eme (5.2.2)]{raynaud-g}). 
By \eqref{c-eq4.1}, we obtain 
$$
a(P_{l_0}, Y, B_Y)+1=a(P_{l_0}, Y', \mathbf B_{Y'})+1\leq 
a(Q_{j_0}, X', B_{X'})+1=a(Q_{j_0}, X, B)+1. 
$$ 
Therefore, if $\mld_{f^{-1}(Z)}(X, B)\geq 0$, then 
$a(P_{l_0}, Y, B_Y)+1\leq \mld _{f^{-1}(Z)}(X, B)$. 
When $\mld_{f^{-1}(Z)}(X, B)=-\infty$, we get $a(P_{l_0}, Y, B_Y)+1<0$. 
Hence, we obtain that 
$$\mld_Z( Y, B_Y)\leq \mld _{f^{-1}(Z)}(X, B)$$ always holds. 

If $\mld _{f^{-1}(Z)}(X, B)=-\infty$, then 
$$
\frac{1}{N}\mld _{f^{-1}(Z)}(X, B)\leq \mld _Z(Y, B_Y)
$$ 
obviously holds for any positive integer $N$. 
Therefore, from now on, we may assume that $\mld _{f^{-1}(Z)}(X, B)\geq 0$. 
Let $P_l$ be any prime divisor contained in $\sigma^{-1}(Z)$. 
Then 
\begin{equation*}
\begin{split}
a(P_l, Y, B_Y)+1&=a(P_l, Y', \mathbf B_{Y'})+1\\& \geq \frac{1}{N_l}
\left(\min _{f'(Q_j)=P_l}
a(Q_j, X', B_{X'})+1\right)\\& \geq \frac{1}{N_l}\mld _{f^{-1}(Z)}(X, B)
\end{split}
\end{equation*} 
for some positive integer $N_l$ 
by \eqref{c-eq4.1}. By \cite[Theorem 2.3]{ambro-mld}, 
we can check that 
$$
\left\{\mld_{f^{-1}(Z)}(X, B)\, |\, Z\subsetneq Y\right\}
$$ 
is a finite subset of $\mathbb Q_{\geq 0}\cup \{-\infty\}$. 
Therefore, we can take a positive integer $N$ satisfying the 
desired properties. 
\end{proof}
 
\section{Cyclic cover of the generic fiber}\label{c-sec6}

The main purpose of this section is to 
interpret moduli parts of 
basic slc-trivial fibrations Hodge theoretically. 
We closely follow the formulation in \cite[Section 5]{ambro3}. 
The approach in \cite[Section 5]{ambro3} 
is essentially the same as those in \cite[Section 5, Part II]{mori} and 
\cite[Section 4]{fujino-certain}. 

\medskip 

Let $f:(X, B)\to Y$ be a basic slc-trivial fibration 
such that $Y$ is quasi-projective. 
Let $F$ be a general fiber of $f:X\to Y$. 
We put $$
b(F, B_F):=\min \{ m\in \mathbb Z_{>0} \, | \, m(K_F+B_F)\sim 0\} 
$$ 
where $K_F+B_F=(K_X+B)|_F$. 
Since 
$Y$ is quasi-projective, 
we can take a $\mathbb Q$-Cartier $\mathbb Q$-divisor 
$D$ on $Y$ and $\varphi\in \Gamma (X, \mathcal K^*_X)$ such that 
$$
K_X+B+\frac{1}{b}(\varphi)=f^*D
$$
holds, where $b=b(F, B_F)$. 
Therefore, we have 
\begin{equation}\label{d-eq6.1}
K_X+B+\frac{1}{b}(\varphi)=
f^*(K_Y+B_Y+M_Y),  
\end{equation} 
as in \ref{c-say4.5}. 

\begin{say}[Cyclic cover of the generic fiber under the 
assumption that $Y$ is smooth]\label{d-say6.1} 
From now on, we assume that $Y$ is smooth. In particular, 
$K_Y$ is Cartier. 
By taking some suitable blow-ups, 
we may assume that $\Supp\left(B-f^*(B_Y+M_Y)\right)$ 
is a simple normal crossing 
divisor on $X$, $(B^h)^{=1}$ is Cartier, 
and every stratum of $(X, (B^h)^{=1})$ is dominant onto $Y$ 
(see, for example, 
\cite[Section 8]{bierstone-vera} and 
\cite[Lemma 2.11]{fujino-projectivity}). 
Let $\pi:\widetilde X\to X$ be the $b$-fold cyclic cover 
associated to \eqref{d-eq6.1}. 
Then we have the following commutative diagram. 
$$
\xymatrix{
(X, B)\ar[d]_-f & \widetilde X \ar[dl]^-{\widetilde f}
\ar[l]_-\pi\\ 
Y & 
}
$$ 
More explicitly, we put 
$$
\Delta=K_{X/Y}+B-f^*(B_Y+M_Y),  
$$ 
where $K_{X/Y}=K_X-f^*K_Y$. 
Then $b\Delta=-(\varphi)\sim 0$ holds by definition. 
We note that the support of $\{\Delta\}$ is a simple normal crossing divisor 
on $X$. 
We can define an $\mathcal O_X$-algebra 
structure of $\bigoplus _{i=0}^{b-1}\mathcal 
O_X(\lfloor i\Delta\rfloor)$ by $b\Delta=(\varphi^{-1})\sim 0$. 
We note that 
$$
\mathcal O_X(\lfloor i\Delta\rfloor)\times 
\mathcal O_X(\lfloor j\Delta\rfloor)\to \mathcal O_X
(\lfloor (i+j)\Delta\rfloor)
$$ 
is well defined for $0\leq i, j\leq b-1$ by 
$\lfloor i\Delta\rfloor +\lfloor j\Delta\rfloor\leq 
\lfloor (i+j)\Delta\rfloor$ and that 
$$
\mathcal O_X(\lfloor (i+j)\Delta\rfloor)
\simeq \mathcal O_X(\lfloor (i+j-b)\Delta\rfloor)
$$ 
for $i+j\geq b$ by $b\Delta=(\varphi^{-1})\sim 0$. 
In this situation, we have the following description of 
$\widetilde X$: 
\begin{equation}\label{d-eq6.2}
\widetilde X=\Spec _X\bigoplus _{i=0}^{b-1} \mathcal O_X(\lfloor i\Delta
\rfloor). 
\end{equation} 
Let $\zeta$ be a fixed primitive $b$-th root of unity and let $G=\langle 
\rho\rangle$ be the cyclic group $\mathbb Z/b\mathbb Z$. 
Then $G$ acts on $\bigoplus _{i=0}^{b-1}\mathcal O_X(\lfloor 
i\Delta\rfloor)$ by $\mathcal O_X$-algebra 
homomorphisms defined by: 
$$
\rho(l)=\zeta^i l
$$ 
for a local section $l$ of $\mathcal O_X(\lfloor i\Delta\rfloor)$. 

Here, we give an alternative 
description of $\widetilde X$ for the reader's convenience, 
which is more familiar than \eqref{d-eq6.2}. 
We put $\mathcal L=\mathcal O_X(-\lfloor \Delta\rfloor)$. 
Then we see that  
$b\{\Delta\}=(\varphi^{-1})-b\lfloor \Delta\rfloor\in |\mathcal L^b|$. 
In this notation, we have  
\begin{equation}\label{d-eq6.3}
\widetilde X=\Spec _X\bigoplus _{i=0}^{b-1} \mathcal L^{-i}
\left(\left\lfloor \frac{ib\{\Delta\}}{b}\right\rfloor\right)=
\Spec_X
\bigoplus _{i=0}^{b-1} \mathcal L^{-i} (\lfloor i\{\Delta\}\rfloor). 
\end{equation}
We note that 
$$
\mathcal L^{-i} (\lfloor i\{\Delta\}\rfloor)=\mathcal O_X(i\lfloor \Delta\rfloor 
+\lfloor i\{\Delta\}\rfloor)=\mathcal O_X(\lfloor i\Delta\rfloor). 
$$ 
Thus this usual description of the $b$-fold cyclic cover \eqref{d-eq6.3} 
coincides with the above description \eqref{d-eq6.2}. 
We note that \cite[2.3 Ramified covers]{kollar} may be 
helpful. 

By construction, $\pi:\widetilde X\to X$ is \'etale outside 
$\Supp \{\Delta\}$. We note that, over a neighborhood of 
the generic point of every irreducible component of 
$\Supp \{\Delta\}$, 
$\widetilde X$ is normal and 
$\pi:\widetilde X\to X$ is a well-known $b$-fold 
cyclic cover of $X$ associated to $b\{\Delta\}\in |\mathcal L^b|$. 
By construction again, there exists 
$\widetilde \varphi\in \Gamma (\widetilde X, 
\mathcal K_{\widetilde X}^*)$ such that 
$\pi^*\varphi=\widetilde \varphi^b$ in 
$\Gamma (\widetilde X, 
\mathcal K_{\widetilde X}^*)$. 
We note that $\widetilde X$ is connected by the definition of $b=b(F, B_F)$. 
We also note that $G=\langle \rho\rangle$ acts on $\widetilde \varphi$ by 
$\rho(\widetilde \varphi)=\zeta^{-1}\widetilde \varphi$. 

We define $B_{\widetilde X}$ by 
the formula 
$K_{\widetilde X}+B_{\widetilde X}=\pi^*(K_X+B)$. We can 
easily see that $(B^h_{\widetilde X})^{=1}=\pi^*((B^h)^{=1})$ holds. 
We can also check that 
$(\widetilde X, (B^h_{\widetilde X})^{=1})$ is semi-log canonical and 
that every slc stratum of $(\widetilde X, (B^h_{\widetilde X})^{=1})$ 
is dominant onto $Y$. 
Let $d:V\to \widetilde X$ be a projective birational morphism 
from a simple normal crossing variety $V$. 
Then we have the following commutative diagram. 
\begin{equation}\label{d-eq6.4}
\xymatrix{
(X, B)\ar[d]_-f & \widetilde X \ar[dl]_-{\widetilde f}
\ar[l]_-\pi& (V, B_V)\ar[dll]^-h\ar[l]_-d\\ 
Y & & 
}
\end{equation}
We assume that $d$ is an isomorphism over the generic point of 
every slc stratum of $(\widetilde X, (B^h_{\widetilde X})^{=1})$. 
We put $g:=\pi\circ d: V\to X$ and 
$$
K_V+B_V=d^*(K_{\widetilde X}+B_{\widetilde X})=g^*(K_X+B). 
$$ 
By taking some more blow-ups if necessary, 
we may assume that $(B^h_V)^{=1}$ is Cartier (see 
\cite[Section 8]{bierstone-vera} and 
\cite[Lemma 2.11]{fujino-projectivity}). 
We put $\psi=d^*\widetilde \varphi \in \Gamma (V, \mathcal K_V^*)$. 
Thus we have $g^*\varphi=\psi^b\in \Gamma (V, \mathcal K_V^*)$. 
Therefore, 
\begin{equation}\label{c-eq6.5}
K_V+B_V+(\psi)=h^*(K_Y+B_Y+M_Y)
\end{equation}
holds. We further assume that $(V, B_V)$ is a simple normal crossing 
pair. 
By construction, 
$$
\pi_*\omega_{\widetilde X/Y}((B^h_{\widetilde X})^{=1})
=\bigoplus _{i=0}^{b-1} \omega_{X/Y}((B^h)^{=1})\otimes 
\mathcal O_X(\lceil -i\Delta\rceil). 
$$ 
We note that $(B^h_{\widetilde X})^{=1}=\pi^*((B^h)^{=1})$ holds and 
that $G$ acts on $\pi_*\omega _{\widetilde X/Y}((B^h_{\widetilde X})^{=1})$ 
naturally. 
Since $K_V+(B^h_V)^{=1}=d^*(K_{\widetilde X}+(B^h_{\widetilde X})^{=1})
+E$, where $E$ is a $d$-exceptional 
$\mathbb Q$-divisor such that 
$\lceil E\rceil \geq 0$, 
$d_*\omega_{V/Y}((B^h_V)^{=1})=
\omega_{\widetilde X/Y}((B^h_{\widetilde X})
^{=1})$ holds. 
Therefore, the following eigensheaf decomposition holds: 
\begin{equation}\label{c-eq6.6}
\begin{split}
h_*\omega_{V/Y}((B^h_V)^{=1})
&=\widetilde f_*\omega_{\widetilde X/Y} ((B^h_{\widetilde X})^{=1})
\\ &=\bigoplus _{i=0}^{b-1} f_*\mathcal O_X(\lceil 
(1-i)K_{X/Y}
-iB+if^*B_Y+if^*M_Y\rceil +(B^h)^{=1})
\end{split}
\end{equation} 
since $\Delta=K_{X/Y}+B-f^*(B_Y+M_Y)$. 
We note that $$\rank f_*\mathcal O_X(\lceil -B+f^*B_Y+f^*M_Y\rceil+
(B^h)^{=1})=\rank f_*\mathcal O_X(\lceil -(B^{<1})\rceil)=1$$ by 
Definition \ref{c-def4.1} (5). 
\end{say}

\begin{say}
[Cyclic cover of the generic fiber when $Y$ may be singular]
\label{d-say6.2} 
Let $Y_0$ be a nonempty Zariski open set 
of $Y$ such that $Y_0$ is smooth. 
By restricting \eqref{d-eq6.1} to $Y_0$, 
we have 
\begin{equation}\label{d-eq6.7}
K_{X_0}+B_0+\frac{1}{b}(\varphi)=f_0^*(K_{Y_0}+B_{Y_0}+M_{Y_0})
\end{equation} and 
can construct the following commutative diagram 
\begin{equation}\label{d-eq6.8}
\xymatrix{
(X_0, B_0)\ar[d]_-{f_0} & \widetilde X_0 \ar[dl]_-{\widetilde f_0}
\ar[l]_-{\pi_0}& (V_0, B_{V_0})\ar[dll]^-{h_0}\ar[l]_-{d_0}\\ 
Y_0& & 
}
\end{equation} 
similar to \eqref{d-eq6.4} since $Y_0$ is smooth. 
We take a projective birational modification $\sigma:Y'\to Y$ 
and construct an induced basic slc-trivial fibration $f':(X', B_{X'})
\to Y'$ of $f:(X, B)\to Y$ by $\sigma: Y'\to Y$, and so on. 
Then we get the following commutative diagram similar to 
the diagram \eqref{d-eq6.4}.   
\begin{equation}\label{d-eq6.9}
\xymatrix{(X', B_{X'})\ar[d]_-{f'} & \widetilde X' \ar[dl]_-{\widetilde f'}
\ar[l]_-{\pi'}& (V', B_{V'})\ar[dll]^-{h'}\ar[l]_-{d'}\\ 
Y' & & 
}
\end{equation} 
We can assume that this new diagram \eqref{d-eq6.9} 
coincides with the 
diagram \eqref{d-eq6.8} over some nonempty Zariski open set of $Y'$ 
by \cite[Theorem 1.4]{bierstone-vera}. 
By replacing $f:(X, B)\to Y$ and $(V, B_V)$ with 
$f': (X', B_{X'})\to Y'$ and $(V', B_{V'})$ respectively, 
we further assume that the following properties hold 
for \eqref{d-eq6.4}. 
\begin{itemize}
\item[(a)] $Y$ is a smooth quasi-projective 
irreducible variety, and $X$ and $V$ are quasi-projective 
simple 
normal crossing varieties. 
\item[(b)] there exist simple normal crossing divisors 
$\Sigma_X$, $\Sigma_V$, and $\Sigma_Y$ 
on $X$, $V$, and $Y$, respectively.   
\item[(c)] $f$ and $h$ are projective surjective 
morphisms. 
\item[(d)] the supports of 
$B$, $B_V$, and $B_Y$, $M_Y$ are 
contained in $\Sigma_X$, $\Sigma_V$, and $\Sigma_Y$, respectively. 
\item[(e)] every stratum of $(X, \Sigma^h_X)$ 
and $(V, \Sigma^h_V)$ is smooth 
over $Y\setminus \Sigma_Y$. 
\item[(f)] $f^{-1}(\Sigma_Y)\subset \Sigma_X$, $f(\Sigma^v_X)\subset 
\Sigma_Y$, and $h^{-1}(\Sigma_Y)\subset 
\Sigma_V$, $h(\Sigma^v_V)\subset 
\Sigma_Y$. 
\item[(g)] $(B^h)^{=1}$ and $(B^h_V)^{=1}$ are Cartier. 
\end{itemize}
\end{say}

By definition and construction, 
we can easily check the following basic 
properties of $h: (V, B_V)\overset{g}\longrightarrow (X, B)
\overset{f}\longrightarrow Y$. 
Proposition \ref{d-prop6.3} is the main result of this section. 

\begin{prop}\label{d-prop6.3}
We have the following properties. 
\begin{itemize}
\item[(i)] $\pi:\widetilde X\to X$ is a Galois cover and its Galois 
group $G$ is $\mathbb Z/b\mathbb Z$. 
\item[(ii)] $h:(V, B_V)\to Y$ is a pre-basic slc-trivial fibration. 
\item[(iii)] $f:(X, B)\to Y$ and $h: (V, B_V)\to Y$ induce 
the same discriminant and moduli part on $Y$. 
\item[(iv)] Assume that, for any irreducible component $P$ of 
$\Supp \Sigma_Y$, there exists a prime divisor $Q$ on $V$ such that 
$\mult _Q(-B_V+h^*B_Y)=0$, 
$h(Q)=P$, and 
$\mult _Qh^*P=1$. 
Then $M_Y$ is an integral divisor and $\mathcal O_Y(M_Y)$ is a 
direct summand of $h_*\mathcal O_V(K_{V/Y}+(B^h_V)^{=1})$. 
\item[(v)] In {\em{(iv)}}, we further assume that 
all the local monodromies on the local system 
$$R^{\dim V-\dim Y} (h|_{V^*})_*\iota_!\mathbb Q_{V^*\setminus 
(B_{V^*}^h)^{=1}}$$ around 
$\Sigma_Y$ are unipotent, where 
$Y^*=Y\setminus \Sigma_Y$, $V^*=h^{-1}(Y^*)$, $B_{V^*}=(B_V)|_{V^*}$, 
and $\iota: V^*\setminus (B_{V^*}^h)^{=1}\hookrightarrow V^*$ is the natural 
open immersion. 
Then $$\left(h_*\mathcal O_V(K_{V/Y}+(B^h_V)^{=1})\right)|_W$$ 
is a nef locally free sheaf on $W$, 
where $W$ is any complete subvariety of $Y$. 
In particular, $(M_Y)|_W$ is a nef 
Cartier divisor on $W$. 
\end{itemize}
\end{prop}

\begin{proof}[Proof of Proposition \ref{d-prop6.3}]
In the above construction, we have already described 
the action of $G=\mathbb Z/b\mathbb Z$ on $\widetilde X$ explicitly. 
Therefore, (i) is obvious. 
By the definition of $b=b(F, B_F)$, the general fiber of 
$h:V\to Y$ is connected. 
We consider the Stein factorization 
$$
h: V\longrightarrow Z:= \Spec _Y h_*\mathcal O_V \overset{\alpha}
{\longrightarrow} Y
$$ 
of $h: V\to Y$. 
Note that $\alpha$ is an isomorphism over a nonempty 
Zariski open set of $Y$ since the general fibers 
of $h$ are connected. 
Since every irreducible component of $V$ is dominant 
onto $Y$, $Z$ is an irreducible variety. 
Therefore, by Zariski's main theorem, 
$\alpha$ is an isomorphism. 
Hence we see that the natural map 
$\mathcal O_Y\to h_*\mathcal O_V$ is an isomorphism. 
By construction, $K_V+B_V\sim _{\mathbb Q, f}0$ 
and $B_V=B^{\leq 1}_V$ holds over 
the generic point of $Y$. 
Therefore, 
$h: (V, B_V)\to Y$ is a pre-basic slc-trivial 
fibration. This is (ii). 
By construction again, we see that $(X, B+tf^*P)$ is 
sub slc over the generic point of $P$ if and 
only if $(V, B_V+th^*P)$ is sub slc over the 
generic point of $P$. 
Therefore, $f:(X, B)\to Y$ and $h:(V, B_V)\to Y$ induce the 
same discriminant $\mathbb Q$-divisor $B_Y$. 
This implies (iii) since $M_Y=D-K_Y-B_Y$. 

From now on, we will prove (iv). 
We note that 
$$K_X+B+\frac{1}{b}(\varphi)=f^*(K_Y+B_Y+M_Y)$$ and 
\begin{equation}\label{d-eq6.10}
K_V+B_V+(\psi)=h^*(K_Y+B_Y+M_Y). 
\end{equation}
By \eqref{d-eq6.10} and $\mult _Q (-B_V+h^*B_Y)=0$ in (iv), 
we obtain $\mult _Q h^*M_Y\in \mathbb Z$. 
Since $\mult _Q h^*P=1$ by the assumption in (iv), 
$M_Y$ is an integral divisor on $Y$. 
By Theorem \ref{c-thm3.1}, 
$h_*\mathcal O_V(K_{V/Y}+(B^h_V)^{=1})$ is a locally 
free sheaf. As we saw above, by construction and assumption, we have the 
following eigensheaf decomposition 
\begin{equation}\label{d-eq6.11}
\begin{split}
h_*\omega_{V/Y}((B^h_V)^{=1})
&=\widetilde f_*\omega_{\widetilde X/Y} 
((B^h_{\widetilde X})^{=1})
\\ &=\bigoplus _{i=0}^{b-1} f_*\mathcal O_X(\lceil 
(1-i)K_{X/Y}
-iB+if^*B_Y+if^*M_Y\rceil +(B^h)^{=1}). 
\end{split}
\end{equation}
We note that the eigensheaf corresponding to the eigenvalue 
$\zeta^{-1}$ is 
$$\mathcal N=
f_*\mathcal O_X(\lceil -B+f^*B_Y+f^*M_Y\rceil +(B^h)^{=1}), $$ 
which is 
an invertible sheaf on $Y$. 
From now on, we will prove that $\mathcal O_Y(M_Y)=\mathcal N$ holds. 
Since $\mathcal O_Y(M_Y)$ and $\mathcal N$ are both 
invertible, we can freely replace $Y$ with its Zariski open set $Y^0$ such that 
$\codim _Y(Y\setminus Y^0)\geq 2$. 
Therefore, we can assume that $\lceil -B+f^*B_Y\rceil +(B^h)^{=1}$ 
and $-B_V+h^*B_Y+(B^h_V)^{=1}$ are both effective. 
We have already seen that $M_Y$ is an integral divisor. 
Since $\lceil -B+f^*B_Y\rceil +(B^h)^{=1}$ is effective, 
there exists a natural inclusion: 
$$
\mathcal O_Y(M_Y)\subset \mathcal N=\mathcal O_Y(M_Y)
\otimes f_*\mathcal O_X(\lceil -B+f^*B_Y\rceil +(B^h)^{=1}). 
$$ 
More precisely, we have: 
\begin{claim} The following equality 
$$
\mathcal O_Y(M_Y)=\mathcal N
$$ 
holds under the assumptions in (iv). 
\end{claim}

\begin{proof}[Proof of Claim] 
By \eqref{d-eq6.10}, 
we see that 
$$
K_{V/Y}+(B^h_V)^{=1}+(\psi)
$$ 
is effective over some nonempty Zariski open set 
of $Y$. 
Therefore, it defines a holomorphic section $\tau$ of 
$h_*\omega_{V/Y}((B^h_V)^{=1})=\widetilde f_*\omega
_{\widetilde X/Y}((B^h_{\widetilde X})^{=1})$ on some nonempty 
Zariski open set of $Y$. 
Since $\psi=d^*\widetilde \varphi$, $G=\langle \rho \rangle$ acts on $\widetilde 
\varphi$ by $\rho(\widetilde\varphi)=\zeta^{-1}\widetilde \varphi$, 
and the eigensheaf corresponding to the 
eigenvalue $\zeta^{-1}$ is $\mathcal N$, we can consider $\tau$ a rational 
section of $\mathcal N$. Let $a$ be an element of $\mathbb C(Y)$, that is, 
$a$ is a rational function on $Y$. 
Then, by the above description, a rational 
section $a\cdot \tau$ of $\mathcal N$ corresponds to 
$h^*a\cdot \psi$. 
Let $U$ be any nonempty Zariski open set of $Y$. 

Assume that $a\cdot \tau \in \Gamma (U, \mathcal N)$. Then 
\begin{equation}\label{d-eq6.12}
\left((h^*a\cdot \psi)+K_{V/Y}+(B^h_V)^{=1}\right)|_{h^{-1}(U)}\geq 0. 
\end{equation}
holds. On the other hand, we see that 
\begin{equation}\label{d-eq6.13}
(h^*a\cdot \psi)+K_{V/Y}+(B^h_V)^{=1}=
h^*((a)+M_Y)+(-B_V+h^*B_Y+(B^h_V)^{=1})
\end{equation} 
holds by \eqref{d-eq6.10}. 
Since $-B_V+h^*B_Y$ contains no fibers over any codimension 
one points of $Y$ by the assumption in (iv), 
\eqref{d-eq6.12} and \eqref{d-eq6.13} imply that 
$$
\left((a)+M_Y\right)|_{U}\geq 0.  
$$ 

Assume that 
$\left((a)+M_Y\right)|_{U}\geq 0$ holds. Then we obtain 
\begin{equation}\label{d-eq6.14}
\left((h^*a\cdot \psi)+K_{V/Y}+(B^h_V)^{=1}\right)|_{h^{-1}(U)}\geq 0  
\end{equation} 
by \eqref{d-eq6.13} because 
$-B_V+h^*B_Y+(B^h_V)^{=1}$ is effective. 
This means that $a\cdot \tau \in \Gamma (U, \mathcal N)$. 

Therefore, we obtain that the desired equality $\mathcal O_Y(M_Y)
=\mathcal N$ holds. 
\end{proof}

Hence, we obtain (iv). (v) is a direct 
consequence of (iv) and Theorem \ref{c-thm3.1}. 
\end{proof}

We close this section with a remark on the assumptions 
in (iv) and (v) in Proposition \ref{d-prop6.3}. 
We will implicitly use it in Sections \ref{c-sec7} and \ref{c-sec8}. 

\begin{rem}\label{d-rem6.4}
Let $h: (V, B_V)\to Y$ be a pre-basic slc-trivial fibration satisfying 
the assumptions in (iv) and (v) in Proposition \ref{d-prop6.3}. 
We consider the following commutative diagram of 
pre-basic slc-trivial fibrations: 
$$
\xymatrix{
(V^\dag, B_{V^\dag}) \ar[rr]^-\alpha \ar[dr]_-{h^\dag}
& &\ar[dl]^-h (V, B_V) \\ 
& Y& 
}
$$ 
where $h^\dag: (V^\dag, B_{V^\dag})\to Y$ is a pre-basic 
slc-trivial fibration, $\alpha$ is an isomorphism 
over $Y^*=Y\setminus \Sigma_Y$, and 
$K_{V^\dag}+B_{V^\dag}=\alpha^*(K_V+B_V)$. 
Then it is almost obvious that 
$h^\dag: (V^\dag, B_{V^\dag})\to Y$ also satisfies 
the assumptions in (iv) and (v) in Proposition \ref{d-prop6.3}. 
\end{rem}

\section{Covering lemmas revisited}\label{c-sec7}

In this section, we explain some covering lemmas, 
which are essentially 
due to Yujiro Kawamata (see \cite{kawamata-abel}). 
We will use Lemma \ref{c-lem7.3}, which is the main 
result of this section, in Sections \ref{c-sec8} and \ref{c-sec9}. 

\medskip 

Let us start with a well-known covering lemma in \cite{kawamata-abel}. 

\begin{lem}[{\cite[Theorem 17]{kawamata-abel}}]\label{c-lem7.1}
Let $X$ be a smooth quasi-projective variety and 
let $D$ be 
a simple normal crossing divisor on $X$ such that 
$D=\sum _{j=1}^rD_j$ is the irreducible decomposition. 
Let $N_j$ be a positive integer for $1\leq j\leq r$. 
Then we can construct a finite ramified 
cover $\tau:Z\to X$ satisfying 
the following properties. 
\begin{itemize}
\item[(i)] $Z$ is a smooth quasi-projective 
variety and there is a simple normal crossing divisor 
$\Sigma_X$ on $X$ such that 
$D\leq \Sigma_X$, 
$\tau$ is \'etale over $X\setminus \Sigma_X$, 
$\tau^{-1}(\Sigma_X)$ is a simple normal crossing divisor on $Z$. 
\item[(ii)] We have $\tau^*D_j=N_j\tau^{-1}(D_j)$ for every $1\leq j\leq r$. 
\end{itemize}
\end{lem}
Since it is very important to understand how to 
construct $\tau:Z\to X$, we sketch the proof of Lemma \ref{c-lem7.1} 
for the reader's convenience. 

\begin{proof}[Sketch of proof]
Here, we closely follow the presentation in 
\cite[3.19.~Lemma]{esnault-viehweg} 
and \cite[Lemma 2.5]{viehweg}. 
We take an ample line bundle $\mathcal A$ on $X$ such that 
$\mathcal A^{N_j}\otimes \mathcal O_X(-D_j)$ is generated 
by global sections for $1\leq j\leq r$. 
We put $n=\dim X$. 
We take general members $H^{(j)}_1, \ldots, 
H^{(j)}_n$ of $\left|\mathcal A^{N_j}\otimes 
\mathcal O_X(-D_j)\right|$ for $1\leq j\leq r$ such that 
$D+\sum _{i, j} H^{(j)}_i$ is a simple normal crossing 
divisor on $X$. Let $Z^{(j)}_i$ be the cyclic cover 
obtained by taking the $N_j$-th root out of 
$D_j+H^{(j)}_i$ (see \cite[3.5.~Cyclic covers]{esnault-viehweg} 
and \cite[Lemma 2.3]{viehweg}). 
More explicitly, let $s^{(j)}_i\in \Gamma (X, \mathcal A^{N_j})$ 
be a section whose zero divisor is $D_j+H^{(j)}_i$. 
The dual of $s^{(j)}_i: \mathcal O_X\to \mathcal A^{N_j}$, 
that is, $\left(s^{(j)}_i\right)^{\vee}: \mathcal A^{-N_j}\to \mathcal O_X$, 
defines an $\mathcal O_X$-algebra structure on $\bigoplus _{l=0}^{N_j-1}
\mathcal A^{-l}$. 
Then we can write $Z^{(j)}_i=\Spec _X \bigoplus_{l=0}^{N_j-1}\mathcal A^{-l}$.  
In this situation, we can check that the normalization of 
$$
\left( Z^{(1)}_1\times _X \cdots \times _X Z^{(1)}_n\right) 
\times _X\cdots \times _X \left(Z^{(r)}_1\times _X
\cdots \times _X Z^{(r)}_n\right), 
$$ 
which is denoted by $Z$, has the desired properties. 
For the details, we recommend the reader 
to see \cite[3.19.~Lemma]{esnault-viehweg} 
and \cite[Lemma 2.5]{viehweg}. 
We note that we can take $\Sigma_X=D+\sum _{i, j}H^{(j)}_i$ 
by construction. 
We will use the above 
description of $Z$ in the proof of Lemma \ref{c-lem7.3} below. 
\end{proof}

The following slight generalization of Lemma \ref{c-lem7.1} 
is very important for our applications. 

\begin{lem}[{see \cite[Corollary 19]{kawamata} and 
\cite[Remark 4.2]{ambro3}}]\label{c-lem7.2}
Let $X$, $D$, and $N_1, \ldots, N_r$ be 
as in Lemma \ref{c-lem7.1}. 
Let $\rho:X'\to X$ be a projective surjective 
morphism from a smooth quasi-projective variety $X'$ 
such that $\rho^{-1}(D)$ is a simple normal crossing divisor 
on $X'$. 
Then we may assume that 
$\tau:Z\to X$ in Lemma \ref{c-lem7.1} fits into a commutative 
diagram 
$$
\xymatrix{
Z \ar[d]_-\tau& Z'\ar[d]^-{\tau'}\ar[l] _-{\rho'}\\ 
X & X'\ar[l]^-\rho
}
$$
satisfying the following properties. 
\begin{itemize}
\item[(i)] $\tau'$ is a finite cover and 
$\rho'$ is a projective morphism. 
\item[(ii)] $Z'$ is a smooth quasi-projective 
variety. 
\item[(iii)] There is a simple normal crossing divisor 
$\Sigma_{X'}$ on $X'$ 
such that $\tau'$ is \'etale over $X'\setminus \Sigma_{X'}$, 
$(\tau')^{-1}(\Sigma_{X'})$ is a simple normal crossing divisor, 
and $\rho^{-1}(\Sigma_X)\subset \Sigma_{X'}$, 
where $\Sigma_X$ is the simple normal crossing divisor 
on $X$ in Lemma \ref{c-lem7.1}. 
\end{itemize}
\end{lem}

Although the proof of Lemma \ref{c-lem7.2} is 
more or less well known to the experts, we give a detailed proof 
for the reader's convenience. 

\begin{proof}[Proof of Lemma \ref{c-lem7.2}] 
We closely follow the presentation in \cite[Corollary 2.6]{viehweg}. 
In the proof of Lemma \ref{c-lem7.1}, 
we can choose the divisors $H^{(j)}_i$ on $X$ such that 
$D':=\rho^{-1}\left(D+\sum _{i, j}H^{(j)}_i\right)$ is a simple normal crossing 
divisor on $X'$. Let $Z^\dag$ be the normalization 
of an irreducible component of $Z\times _XX'$. 
Then we get the following commutative diagram: 
$$
\xymatrix{
Z \ar[d]_-\tau& Z^\dag\ar[d]^-{\tau^\dag}\ar[l]_-{\rho^\dag}\\
X & X'\ar[l]^-\rho. 
}
$$
By construction, $\tau^\dag$ is \'etale over $X'\setminus D'$. 
Let $D'=\sum _k D'_k$ be the irreducible decomposition. 
We put 
$$
N'_k:= \underset{l}{\mathrm{lcm}}\{e(\Delta_k^l) \, |\, {\text{$\Delta_k^l$ is an 
irreducible component 
of $(\tau^\dag)^{-1}(D'_k)$}}\}, 
$$
where $e(\Delta_k^l)$ denotes the ramification 
index of $\Delta_k^l$ over $D'_k$. 
Let $\widetilde \tau: \widetilde Z\to X'$ be the finite cover constructed in Lemma 
\ref{c-lem7.1} for $X'$, $D'$, and $N'_k$. 
Let $Z'$ be the normalization of an irreducible component of $\widetilde Z\times _{X'}
Z^\dag$. 
Thus we get the following commutative diagram: 
$$
\xymatrix{
Z \ar[d]_-\tau& Z^\dag \ar[d]_-{\tau^\dag}\ar[l]& Z' 
\ar[dl]_-{\tau'}\ar[d]^-\alpha\ar[l]_-\beta\\
X & \ar[l]^-\rho X' & \widetilde Z\ar[l]^-{\widetilde \tau}. 
}
$$
Since $\widetilde \tau: \widetilde Z\to X'$ is constructed as a chain of 
finite cyclic covers, 
the same holds true for $\beta:Z'\to Z^\dag$. 
The ramification index of a component 
of $\beta^{-1}(\Delta^l_k)$ over $\Delta^l_k$ is 
$N'_k/ e(\Delta^l_k)$ by construction. 
Therefore, the ramification index of an irreducible component of 
$(\tau')^{-1}(D'_k)$ 
over $D'_k$ is given by 
$$
\frac{N'_k}{e(\Delta^l_k)}\cdot e(\Delta^l_k)=N'_k. 
$$ 
By the construction of $\widetilde Z$, this is nothing but 
the ramification index of an irreducible component of 
$(\widetilde \tau)^{-1}(D'_k)$ 
over $D'_k$. Therefore, $\alpha: Z'\to \widetilde Z$ is unramified 
in codimension one. 
Since $\widetilde Z$ is smooth, $\alpha$ is \'etale. 
Thus, we can easily check 
that $\tau': Z'\to X'$ satisfies the desired properties.   
\end{proof}

The following lemma is the main result of this section. 
This somewhat technical covering lemma will play an important role 
in Sections \ref{c-sec8} and \ref{c-sec9}. 
 
\begin{lem}[Unipotent reduction for pre-basic slc-trivial 
fibrations]\label{c-lem7.3}
Let $h:(V, B_V)\to Y$ be a pre-basic slc-trivial fibration such that 
$Y$ is a smooth quasi-projective variety. 
Assume that there are simple normal crossing divisors $\Sigma_V$ and 
$\Sigma_Y$ on $V$ and $Y$ respectively such that 
$h^{-1}(\Sigma_Y)\subset \Sigma_V$, 
$h(\Sigma^v_V)\subset \Sigma_Y$, 
$\Supp B_V\subset 
\Sigma_V$, and every stratum of $(V, \Sigma^h_V)$ is smooth 
over $Y\setminus \Sigma_Y$.  
Then there exist a finite cover $\gamma:Y'\to Y$ from a 
smooth quasi-projective variety $Y'$ such that 
$\Sigma_{Y'}:=\gamma^{-1}(\Sigma_Y)$ is a simple normal crossing 
divisor on $Y'$ and a commutative diagram 
$$
\xymatrix{
V\ar[d]_-h & V\times _Y Y' \ar[d]\ar[l]_-q& V'\ar[l]_-p\ar[dl]^-{h'} \\ 
Y & Y' \ar[l]^-\gamma&
}
$$
with the following properties. 
\begin{itemize}
\item[(i)] $p$ is a projective birational morphism from 
a simple normal crossing variety $V'$ 
which is an isomorphism over $Y'\setminus \Sigma_{Y'}$. 
\item[(ii)] $h': (V', B_{V'})\to Y'$ is a pre-basic slc-trivial 
fibration, where $\gamma':=q\circ p: V'\to V$ and 
$K_{V'}+B_{V'}=(\gamma')^*(K_V+B_V)$. 
\item[(iii)] There exists a simple normal crossing 
divisor $\Sigma_{V'}$ on $V'$ such that 
$(\gamma')^{-1}(\Sigma_{V})\subset \Sigma_{V'}$, 
$\Supp B_{V'}\subset \Sigma_{V'}$, 
$h'(\Sigma^v_{V'})\subset \Sigma_{Y'}$, 
$(h')^{-1}(\Sigma_{Y'})\subset \Sigma_{V'}$, 
and every stratum of $(V', \Sigma^h_{V'})$ is smooth 
over $Y'\setminus \Sigma_{Y'}$. 
\item[(iv)] $h': (V', B_{V'})\to Y'$ 
satisfies the assumptions in {\em{(iv)}} and {\em{(v)}} 
in Proposition \ref{d-prop6.3}. 
\end{itemize}
\end{lem}

\begin{proof}
Let $\Sigma_Y=\sum _{j=1}^{r} P_j$ and $\Sigma_V=\sum _lQ_l$ be 
the irreducible decomposition of $\Sigma_Y$ 
and $\Sigma_V$ respectively. 
In this case, we can write $h^*P_j=\sum _l w^l_j Q_l$ 
with $w^l_j\in \mathbb Z_{\geq 0}$ for every $j$. 
Let $M_j$ be the monodromy matrix on the local 
system 
$$R^{\dim V-\dim Y}(h|_{V^*})_*\iota_!\mathbb Q_{V^*
\setminus (B^h_{V^*})^{=1}}$$ around $P_j$, 
where $Y^*:=Y\setminus \Sigma_Y$, 
$V^*:=h^{-1}(Y^*)$, $B_{V^*}:=B_V|_{V^*}$, 
and $\iota: V^*\setminus (B^h_{V^*})^{=1}\hookrightarrow V^*$. 
By \cite[Theorem 4.15]{fujino-fujisawa}, 
the local system $R^{\dim V-\dim Y}(h|_{V^*})_*\iota_!\mathbb Q_{V^*
\setminus (B^h_{V^*})^{=1}}$ underlies 
a graded polarizable variation of 
$\mathbb Q$-mixed Hodge structure on $Y^*$. 
In particular, $M_j$ is quasi-unipotent. 
We put 
$$
m_j:=\min\{ m\in \mathbb Z_{>0}\, |\, 
{\text{$M_j^m$ is unipotent}}\}
$$ 
and 
$$ 
w_j:=\underset{l}{\mathrm{lcm}}\{w_j^l\, |\, h(Q_l)=P_j\}
$$
Then we set 
$$
N_j:=\mathrm{lcm}\{m_j, w_j\}. 
$$ 
By applying Lemma \ref{c-lem7.1} to $Y$, $\Sigma_Y$, and 
$N_j$, we can construct a finite cover $\gamma: Y'\to Y$. 
More precisely, let $\mathcal A$ be an ample line bundle on $Y$ such 
that $\mathcal A^{N_j}\otimes \mathcal O_Y(-P_j)$ is generated 
by global sections for $1\leq j\leq r$. 
We put $n=\dim Y$. 
We take general members $H^{(j)}_1, \ldots, H^{(j)}_n$ of 
$|\mathcal A^{N_j}\otimes \mathcal O_Y(-P_j)|$ for $1\leq j\leq r$ such that 
$\Sigma_Y+\sum _{i, j}H^{(j)}_i$ is a simple normal crossing divisor on $Y$. 
By the above data $\mathcal A$, $N_j$, and $\Sigma_Y+\sum _{i, j}H^{(j)}_i$, 
we can construct a finite cover $\gamma: Y'\to Y$, which 
is a chain of cyclic covers (see the proof of 
Lemma \ref{c-lem7.1}). 
Let $s\in \Gamma(Y, \mathcal A^{N_j})$ be a section whose 
zero divisor is $P_j+H^{(j)}_1$. 
The dual of $s: \mathcal O_Y\to \mathcal A^{N_j}$, that is, 
$s^{\vee}: \mathcal A^{-N_j}\to \mathcal O_Y$, 
defines an $\mathcal O_Y$-algebra 
structure on $\bigoplus _{i=0}^{N_j-1}\mathcal A^{-i}$. 
From now on, we will look at $\gamma:Y'\to Y$ in a neighborhood of 
the generic point of $P_j$. 
Therefore, by shrinking $Y$ suitably, 
we assume that $(s=0)=P_j$. 
By construction, we can easily see that 
$\gamma: Y'\to Y$ can be decomposed as follows: 
$$
\gamma: Y'\overset{\beta}\longrightarrow \widetilde Y \overset{\alpha}
\longrightarrow Y, 
$$ 
where $\widetilde Y=\Spec _Y \bigoplus _{i=0}^{N_j-1}\mathcal A^{-i}$, 
$\alpha:\widetilde Y\to Y$ is the cyclic cover 
obtained by taking the $N_j$-th root out of $P_j$, 
and $\beta:Y'\to \widetilde Y$ is a finite \'etale morphism. 
Let us consider $V\times _Y\widetilde Y=
\Spec _V \bigoplus _{i=0}^{N_j-1}(h^*\mathcal A)^{-i}$. 
Note that the $\mathcal O_V$-algebra structure on 
$\bigoplus _{i=0}^{N_j-1}(h^*\mathcal A)^{-i}$ is defined by 
the dual of $h^*s\in \Gamma(V, (h^*\mathcal A)^{N_j})$, that 
is, $(h^*s)^{\vee}: (h^*\mathcal A)^{-N_j}\to \mathcal O_V$. 
We put 
$$\widetilde Z:=\Spec _V \bigoplus _{i=0}^{N_j-1}(h^*\mathcal A)^{-i}
\otimes \mathcal O_V\left(\left\lfloor \frac{ih^*P_j}{N_j}\right\rfloor\right).$$
Of course, the $\mathcal O_V$-algebra structure on 
$$\bigoplus _{i=0}^{N_j-1}(h^*\mathcal A)^{-i}
\otimes \mathcal O_V\left(\left\lfloor \frac{ih^*P_j}{N_j}\right\rfloor\right)$$ 
is defined by the isomorphism 
$$
(h^*\mathcal A)^{-N_j}\otimes \mathcal O_V\left( \left\lfloor 
\frac{N_jh^*P_j}{N_j}\right\rfloor\right)\overset{\sim}
\longrightarrow\mathcal O_V, 
$$  
which is induced by $h^*s\in \Gamma (V, (h^*\mathcal A)^{N_j})$. 
Then we get a morphism $\widetilde Z\to V\times _Y\widetilde Y$, 
which is an isomorphism over $\widetilde Y\setminus 
\alpha^{-1}(P_j)$,  
induced by the natural map of $\mathcal O_V$-algebras 
$$
\bigoplus _{i=0}^{N_j-1}(h^*\mathcal A)^{-i}\to 
\bigoplus _{i=0}^{N_j-1}(h^*\mathcal A)^{-i}
\otimes \mathcal O_V\left(\left\lfloor \frac{ih^*P_j}{N_j}\right\rfloor\right). 
$$
We put $Z':=\widetilde Z\times _{\widetilde Y}Y'$ and 
take a suitable birational modification $a:V'\to Z'$. 
Then we get the following big commutative diagram: 
\begin{equation}\label{d-eq7.1}
\xymatrix{
& \widetilde Z\ar[dl]_-b \ar[d]& Z'\ar[l]_-{\beta_2} \ar[d]&V'\ar[dl]_-p\ar[l]_-a\ar[ddl]^-{h'} \\ 
V\ar[d]_-h& V\times _Y \widetilde Y\ar[d]\ar[l] & V\times _Y Y'\ar[d]
\ar[l]^-{\beta_1}& \\ 
Y & \widetilde Y \ar[l]^-\alpha&Y'\ar[l]^-\beta &
}
\end{equation}
where $\beta, \beta_1$, and $\beta_2$ are finite \'etale morphisms. 
We put $\widetilde P_j=\alpha^{-1}(P_j)$. 
We define $B_{\widetilde Z}$ by $K_{\widetilde Z}+B_{\widetilde Z}
=b^*(K_V+B_V)$. 
Similarly, we 
put $\beta_2^*(K_{\widetilde Z}+B_{\widetilde Z})=K_{Z'}+B_{Z'}$ 
and $a^*(K_{Z'}+B_{Z'})=K_{V'}+B_{V'}$. 
Without loss of generality, by shrinking $Y$ suitably, 
we may assume that 
$h(Q_l)=P_j$ holds if $h(Q_l)\subset P_j$. 
By the construction of $\alpha:\widetilde Y\to Y$ and the definition 
of $N_j$, $c^* \widetilde P_j$ is reduced 
(see \cite[Proposition 7.23]{kollar-mori}) and $(\widetilde Z, c^*\widetilde P_j)$ 
is semi-log canonical, where $c:\widetilde Z\to \widetilde Y$. 
We note that $Z'\to Y'$ in \eqref{d-eq7.1} 
is the base change of $c:\widetilde Z\to \widetilde Y$ by an 
\'etale morphism $\beta:Y'\to \widetilde Y$. 
Therefore, we can take a birational modification $a: V'\to Z'$ which 
is an isomorphism over $Y'\setminus \gamma^{-1}(P_j)$ such that 
$h': (V', B_{V'})\to Y'$ is a pre-basic slc-trivial fibration satisfying 
the desired properties. Although we constructed $h': (V', B_{V'})\to 
Y'$ after shrinking $Y$ around the generic point of $P_j$, 
we can construct a desired pre-basic slc-trivial 
fibration $h': (V', B_{V'})\to Y'$ globally 
without shrinking $Y$ by the above local description and 
\cite[Theorem 1.4]{bierstone-vera}. 
\end{proof}

\section{Pull-back of the moduli parts}\label{c-sec8}

In this section, we see that 
the moduli parts behave well under pull-back by 
generically finite surjective 
morphisms with some mild assumptions.  

\medskip
 
Let $$K_X+B+\frac{1}{b}(\varphi)=f^*(K_Y+B_Y+M_Y)$$ and 
$h:(V, B_V)\overset{g}\longrightarrow 
(X, B)\overset{f}\longrightarrow Y$ be as in Section \ref{c-sec6} 
which satisfies conditions (a)--(g) in Section \ref{c-sec6}.  
Let $\gamma:Y'\to Y$ be a generically finite surjective 
morphism 
from a smooth quasi-projective 
variety $Y'$. 
Assume that there is a simple normal crossing divisor 
$\Sigma_{Y'}$ which contains 
$\gamma^{-1}(\Sigma_Y)$. 
By base change, 
we have a commutative diagram: 
\begin{equation}\label{c-eq8.1}
\xymatrix{& V\ar[dd]\ar[dl]_-g &  &V' \ar[dl]_-{g'} 
\ar[ll]_-\nu \ar[dd]\\ 
X\ar[dr]_-f& & X' \ar[ll]|-\hole _(.35)\sigma\ar[dr]_-{f'}&\\
& Y && Y'\ar[ll]^-\gamma
}
\end{equation}
where $h':(V', B_{V'})\overset{g'}\longrightarrow 
(X', B_{X'})\overset{f'}\longrightarrow Y'$ satisfies 
the same properties, that is, (a)--(g) in Section \ref{c-sec6}, 
and it is nothing but the base change 
of $h:(V, B_V)\overset{g}\longrightarrow 
(X, B)\overset{f}\longrightarrow Y$ by 
$\gamma: Y'\to Y$ over $Y\setminus \Sigma_Y$. 
We note that 
$B_{X'}$ and $B_{V'}$ are induced by crepant 
pull-back, that is, 
$K_{X'}+B_{X'}=\sigma^*(K_X+B_X)$ and $K_{V'}+B_{V'}=
\nu^*(K_V+B_V)$, 
$\Sigma_{X'}\supset \sigma^{-1}(\Sigma_X)$, 
$\Sigma_{V'}\supset \nu^{-1}(\Sigma_V)$, and 
$\varphi'=\sigma^*\varphi$. 
In this situation, we say that the setup 
$h':(V', B_{V'})\overset{g'}\longrightarrow 
(X', B_{X'})\overset{f'}\longrightarrow Y'$ is induced from 
$h:(V, B_V)\overset{g}\longrightarrow (X, B_X)\overset{f}
\longrightarrow Y$ by the base change 
$\gamma:Y'\to Y$. 
\medskip

In the above setup, we have the following theorem, 
which is a generalization of 
\cite[Proposition 5.5]{ambro3}. 
Note that \cite[Proposition 5.5]{ambro3} is a generalization of 
\cite[Proposition 4.2]{fujino-certain}. 

\begin{thm}\label{c-thm8.1}
Let $h:(V, B_V)\overset{g}\longrightarrow 
(X, B)\overset{f}\longrightarrow Y$ be a setup 
as in Section \ref{c-sec6} 
which satisfies conditions {\em{(a)}}--{\em{(g)}} in Section \ref{c-sec6}. 
Let $\gamma:Y'\to Y$ be a generically finite 
projective surjective 
morphism from a smooth quasi-projective 
variety $Y'$. 
Assume that 
there exists a simple normal crossing divisor $\Sigma_{Y'}$ on $Y'$ which 
contains $\gamma^{-1}(\Sigma_Y)$. 
We consider an induced setup 
$h': (V', B_{V'})\overset{g'}\longrightarrow 
(X', B_{X'})\overset{f'}\longrightarrow Y'$ as in 
\eqref{c-eq8.1}. 
Let $M_{Y'}$ be the moduli part of the induced setup 
$h': (V', B_{V'})\overset{g'}\longrightarrow 
(X', B_{X'})\overset{f'}\longrightarrow Y'$. 
Then we obtain $\gamma ^*M_Y=M_{Y'}$. 
\end{thm}
\begin{proof}
We divide the proof into the following two steps. 
\begin{step}\label{c-thm8.1-step1}
In this step, we further 
assume that $h: (V, B_V)\to Y$ and $h':(V', B_{V'})\to 
Y'$ satisfy the 
assumptions in (iv) and (v) in Proposition \ref{d-prop6.3}. 
Then the moduli parts $M_Y$ and $M_{Y'}$ are both integral divisors. 
By Theorem \ref{c-thm3.1}, there exists a natural 
isomorphism 
$$
\gamma^*\left(h_*\mathcal O_V(K_{V/Y}+(B^h_V)^{=1}
\right)\simeq h'_*
\mathcal O_{V'}(K_{V'/Y'}+(B^h_{V'})^{=1}), 
$$
which is compatible 
with the action of the Galois group $G=\mathbb Z/b\mathbb Z$ 
(see Proposition \ref{d-prop6.3}). 
Therefore, we have an induced isomorphism 
of eigensheaves corresponding to the eigenvalue $\zeta^{-1}$. 
Thus we obtain the isomorphism $\gamma^*\mathcal O_Y(M_Y)
\simeq 
\mathcal O_{Y'}(M_{Y'})$. 
This means that $\gamma^*M_Y-M_{Y'}$ is linearly trivial. 
If $\gamma$ is finite, then we know 
that $\gamma^*M_Y=M_{Y'}$ holds by Lemma \ref{d-lem4.10}. 
Therefore, $\gamma^*M_Y-M_{Y'}$ is exceptional over $Y$. 
More precisely, $\codim _Y\gamma(E)\geq 2$ holds for 
$E=\gamma^*M_Y-M_{Y'}$. 
Thus we get $\gamma^*M_Y=M_{Y'}$ in this special case 
since $\gamma^*M_Y-M_{Y'}$ is linearly trivial. 
\end{step}
\begin{step}\label{c-thm8.1-step2}
In this step we treat the general case. 
By Lemma \ref{c-lem7.3}, 
we can construct a finite cover $\tau:\overline Y\to Y$ such that an induced 
setup $\overline h: (\overline V, B_{\overline V})\overset{\overline g}
\longrightarrow (\overline X, B_{\overline X}) 
\overset{\overline f}\longrightarrow \overline Y$ 
as in \eqref{c-eq8.1} 
satisfies the assumptions in (iv) and (v) in Proposition \ref{d-prop6.3}. 
By construction, we may assume that 
there is a simple normal crossing divisor $\Sigma_1$ on $Y$ 
such that $\Sigma_Y\subset \Sigma_1$, 
$\tau$ is \'etale over $Y\setminus \Sigma_1$, 
and $\gamma^{-1}(\Sigma_1)$ is a simple 
normal crossing divisor on $Y'$. 
We may further assume that $\gamma^{-1}(\Sigma_1)\cup \Sigma_{Y'}$ 
is contained in a simple normal crossing divisor.  
By Lemma \ref{c-lem7.2}, 
we can construct the following commutative diagram: 
$$
\xymatrix{
\overline Y \ar[d]_-\tau& \widetilde Y\ar[d]^-{\widetilde\tau}
\ar[l]_-{\widetilde \gamma}\\ 
Y & Y'\ar[l]^-\gamma
}
$$ 
where $\widetilde \tau:\widetilde Y\to Y'$ is a finite 
cover from a smooth quasi-projective 
variety $\widetilde Y$, and there is 
a simple normal crossing divisor $\Sigma_2$ on $Y'$ such that 
$\gamma^{-1}(\Sigma_1)\cup \Sigma_{Y'}\subset \Sigma_2$, 
$\widetilde \tau$ is \'etale over $Y'\setminus \Sigma_2$, 
and $(\widetilde \tau)^{-1}(\Sigma_2)$ is a simple 
normal crossing divisor on $\widetilde Y$. 
Then we apply Lemma \ref{c-lem7.3} again. 
We get a finite cover $\overline Y'\to \widetilde Y$ from a 
smooth quasi-projective 
variety $\overline Y'$ and the following 
commutative diagram: 
$$
\xymatrix{&& \overline Y'\ar[dll]_-{\gamma'}\ar[dl]\ar[ddl]^-{\tau'}\\ 
\overline Y\ar[d]_-\tau & \widetilde Y \ar[d]_-{\widetilde \tau}\ar[l]^
-{\widetilde \gamma}& \\ 
Y & \ar[l]^-\gamma Y' &
}
$$
such that an induced setup 
$\overline h': (\overline V', B_{\overline V'})\overset{\overline g'}
\longrightarrow (\overline X', B_{\overline X'}) \overset{\overline f'}
\longrightarrow \overline Y'$ 
satisfies the assumptions in (iv) and (v) in Proposition \ref{d-prop6.3}. 
Hence, $\overline h: (\overline V, B_{\overline V})\to 
\overline Y$ and $\overline h': (\overline V', B_{\overline V'})\to 
\overline Y'$ satisfy the assumptions in Step \ref{c-thm8.1-step1}. 
Therefore, we get $M_{\overline Y'}=(\gamma')^*M_{\overline Y}$, 
We note that $\tau^*M_Y=M_{\overline Y}$ and 
$(\tau')^*M_{Y'}=M_{\overline Y '}$ hold by Lemma \ref{d-lem4.10} 
because $\tau$ and $\tau'$ are both finite. 
Thus we get $(\tau')^*(M_{Y'}-\gamma^*M_Y)=0$. 
This implies that $M_{Y'}=\gamma^*M_Y$ holds. 
\end{step}
Thus, we obtain $\gamma^*M_Y=M_{Y'}$. 
\end{proof}

\section{Proof of Theorem \ref{c-thm1.2}}\label{c-sec9} 

In this section, we prove Theorem \ref{c-thm1.2}, which is the main 
theorem of this paper. 
Theorem \ref{c-thm1.2} obviously generalizes \cite[Theorem 0.2]{ambro3} 
and \cite[Theorem 3.6]{fujino-gongyo}. 
Since we have already checked that 
the moduli part of a given basic slc-trivial 
fibration behaves well under pull-back by generically finite 
surjective morphisms 
with some mild assumptions in Theorem \ref{c-thm8.1}, 
there are no difficulties to prove Theorem \ref{c-thm1.2}. 

\medskip

Let us prove Theorem \ref{c-thm1.2}. 

\begin{proof}[Proof of Theorem \ref{c-thm1.2}]
Let $f:(X, B)\to Y$ be a basic slc-trivial fibration. 
As in \ref{c-say4.5}, we can write 
$$
K_X+B\sim _{\mathbb Q} f^*(K_Y+B_Y+M_Y). 
$$
Without loss of generality, by taking a projective 
birational modification $\sigma:Y'\to Y$ from 
a smooth quasi-projective variety $Y'$ and considering 
an induced basic slc-trivial fibration $f': (X', B_{X'})\to Y'$ of 
$f:(X, B)\to Y$, we may assume that 
$Y$ is a smooth quasi-projective variety. 
By Remark \ref{d-rem4.9}, we may further assume that 
$$
K_X+B+\frac{1}{b}(\varphi)=f^*(K_Y+B_Y+M_Y)
$$ 
holds, 
where $b=b(F, B_F)$ and $\varphi\in \Gamma (X, \mathcal K^*_X)$. 
It is sufficient to prove that 
the moduli $\mathbb Q$-b-divisor $\mathbf M$ 
is b-potentially nef in the above setup. 
By taking a birational modification of $X$ which is an isomorphism 
over the generic point of every stratum of $X$, 
we may assume that $\Supp (B-f^*(B_Y+M_Y))$ is a simple 
normal crossing divisor on $X$, $(B^h)^{=1}$ 
is a Cartier divisor on $X$, and 
every stratum of $(X, (B^h)^{=1})$ is dominant 
onto $Y$ (see, for example, \cite[Theorem 1.4 and 
Section 8]{bierstone-vera} and 
\cite[Lemma 2.11]{fujino-projectivity}). 
As in Section \ref{c-sec6}, 
by taking the $b$-fold cyclic cover $\pi: (\widetilde X, B_{\widetilde X})\to 
(X, B)$ associated to $$K_{X/Y}+B-f^*(B_Y+M_Y)=\frac{1}{b}(\varphi^{-1})$$ 
and a suitable birational modification 
$d: (V, B_V)\to (\widetilde X, B_{\widetilde X})$, 
we get $$h: (V, B_V)\overset{g}\longrightarrow (X, B)\overset{f}\longrightarrow 
Y.$$ Then we take a projective birational morphism 
$\sigma: Y'\to Y$ from a smooth quasi-projective 
variety $Y'$ and obtain an induced setup 
$h':(V', B_{V'})\overset{g'}\longrightarrow 
(X', B_{X'})\overset{f'}\longrightarrow Y'$ which 
satisfies conditions (a)--(g) in Section \ref{c-sec6}. 

From now on, we will prove that $\nu^*(\mathbf M_{Y'})
=\mathbf M_{Y''}$ and $\nu^*(K_{Y'}+\mathbf B_{Y'})
=K_{Y''}+\mathbf B_{Y''}$ hold for every proper birational morphism 
$\nu: Y''\to Y'$ from a normal variety $Y''$. We take a common resolution 
$$
\xymatrix{
& Y'''\ar[dl]_-p\ar[dr]^-q& \\ 
Y'& & \ar[ll]^-\nu Y''
}
$$ 
such that $Y'''$ is a smooth 
quasi-projective 
variety and that $p^{-1}(\Sigma_{Y'})$ is a simple 
normal crossing divisor on $Y'''$. 
We consider an induced setup 
$h''': (V''', B_{V'''})\overset{g'''}\longrightarrow 
(X''', B_{X'''})\overset{f'''}\longrightarrow Y'''$ as in 
Section \ref{c-sec8}. 
By Theorem \ref{c-thm8.1}, 
we get $p^*\mathbf M_{Y'}=\mathbf M_{Y'''}$. 
Thus we obtain  
$p^*(K_{Y'}+\mathbf B_{Y'})=K_{Y'''}+\mathbf B_{Y'''}$. 
Since $q: Y'''\to Y''$ is birational, 
$\nu^*\mathbf M_{Y'}=\mathbf M_{Y''}$ and 
$\nu^*(K_{Y'}+\mathbf B_{Y'})=K_{Y''}+\mathbf B_{Y''}$ follow 
from the above relations by taking $q_*$. 

Finally, we will prove that $\mathbf M_{Y'}$ is potentially nef. 
By Lemma \ref{d-lem4.12}, we can compactify $f:(X, B)\to Y$ and 
may assume that $X$ and $Y$ are both complete varieties. 
Therefore, it is sufficient to prove that $\mathbf M_{Y'}$ is nef. 
Let $\tau:\overline Y'\to Y'$ be a suitable finite 
cover from a smooth projective variety $\overline Y'$ 
as in Lemma \ref{c-lem7.3}. 
More precisely, $\overline h': (\overline V', B_{\overline V'})\to 
\overline Y'$ satisfies the assumptions in (iv) and (v) in 
Proposition \ref{d-prop6.3}, where 
$\overline h': (\overline V', B_{\overline V'}) 
\overset {\overline g'}\longrightarrow (\overline X', B_{\overline X'})
\overset{\overline f'}\longrightarrow \overline Y'$ 
is an induced setup from 
$h':(V', B_{V'})\overset{g'}\longrightarrow (X', B_{X'})
\overset{f'}\longrightarrow Y'$ by $\tau:\overline Y'\to Y'$. 
Then $\tau^*\mathbf M_{Y'}=\mathbf M_{\overline Y'}$ holds by 
Lemma \ref{d-lem4.10} since $\tau$ is finite. 
By Proposition \ref{d-prop6.3}, 
$\mathbf M_{\overline Y'}$ is a nef Cartier divisor. 
Therefore, $\mathbf M_{Y'}$ is nef. 
Hence, we obtain that $\mathbf M$ is b-potentially nef.  
\end{proof}

\section{Quasi-log canonical pairs}\label{c-sec10}

In this section, let us recall the basic definitions of quasi-log canonical 
pairs and prove a result on normal irreducible 
quasi-log canonical pairs, which will play a crucial role in the 
proof of Theorem \ref{c-thm1.7}. 
For the details of the theory of quasi-log 
schemes, see \cite[Chapter 6]{fujino-foundations}. 
We note that our formulation in \cite[Chapter 6]{fujino-foundations} 
is slightly different from Ambro's original one (see \cite{ambro2}). 

\medskip 

Let us start with the definition of globally embedded simple normal 
crossing pairs. We will soon use it for the definition of 
quasi-log canonical pairs (see Definition \ref{c-def10.2}). 

\begin{defn}[Globally embedded simple normal crossing 
pairs]\label{c-def10.1} 
Let $Y$ be a simple normal crossing 
divisor on a smooth variety $M$ and let $B$ be an $\mathbb R$-divisor 
on $M$ such that 
$Y$ and $B$ have no common irreducible components and 
that the support of $Y+B$ is a simple normal crossing divisor on $M$. In this 
situation, the pair $(Y, B_Y)$, where 
$B_Y:=B|_Y$, is called a {\em{globally embedded simple 
normal crossing pair}}.
\end{defn}

Of course, a globally embedded simple normal crossing pair is 
a simple normal crossing pair in the sense of Definition \ref{d-def2.15}. 
We note that a simple normal crossing variety can not always 
be embedded as a simple normal crossing divisor on a smooth 
variety (see \cite[Example 5.2.7]{fujino-foundations}). 
Therefore, a simple normal crossing pair is not necessarily 
a globally embedded simple normal crossing pair. 

\medskip

Let us quickly look at the definition of 
quasi-log canonical pairs. 

\begin{defn}[Quasi-log canonical pairs]\label{c-def10.2}
Let $X$ be a scheme and let $\omega$ be an 
$\mathbb R$-Cartier divisor (or an $\mathbb R$-line bundle) on $X$. 
Let $f:Y\to X$ be a proper morphism from a globally embedded 
simple normal 
crossing pair $(Y, B_Y)$. If the natural map 
$$\mathcal O_X\to f_*\mathcal O_Y(\lceil -(B_Y^{<1})\rceil)$$ is an 
isomorphism, $B_Y$ is a subboundary $\mathbb R$-divisor, 
and $$f^*\omega\sim _{\mathbb R} K_Y+B_Y$$ holds,  
then $(X, \omega, f:(Y, B_Y)\to X)$ 
or simply $[X, \omega]$ is called a {\em{quasi-log canonical pair}} 
({\em{qlc pair}}, for short). 

We say that $\left(X, \omega, f: (Y, B_Y)\to X\right)$ or 
$[X, \omega]$ has a {\em{$\mathbb Q$-structure}} if 
$B_Y$ is a $\mathbb Q$-divisor, $\omega$ is a $\mathbb Q$-Cartier 
divisor (or a $\mathbb Q$-line bundle), and 
$f^*\omega\sim _{\mathbb Q} K_Y+B_Y$ holds 
in the above definition. 
\end{defn}

Let $\left(X, \omega, f:(Y, B_Y)\to X\right)$ be a quasi-log canonical 
pair as in 
Definition \ref{c-def10.2}. 
Let $\nu:Y^\nu\to Y$ be the normalization. 
We put $K_{Y^\nu}+\Theta=\nu^*(K_Y+B_Y)$, that is, 
$\Theta$ is the sum of the inverse images of $B_Y$ and 
the singular locus of $Y$. 
Then $(Y^\nu, \Theta)$ is sub log canonical in the usual sense 
(see \ref{c-say2.2}). 
Let $W$ be a log canonical center of $(Y^\nu, \Theta)$ or 
an irreducible component of $Y^\nu$. 
Then $f\circ \nu(W)$ is called a {\em{qlc stratum}} of 
$\left(X, \omega, f:(Y, B_Y)\to X\right)$. If there is no danger 
of confusion, we simply call it a qlc stratum of $[X, \omega]$. 
If $C$ is a qlc stratum of $[X, \omega]$ but is not 
an irreducible component of $X$, then $C$ is called 
a {\em{qlc center}} of $\left(X, \omega, f:(Y, B_Y)\to X\right)$ or 
simply of $[X, \omega]$. 
The union of all qlc centers of $[X, \omega]$ is denoted by 
$\Nqklt(X, \omega, f:(Y, B_Y)\to X)$ or simply by  
$\Nqklt (X, \omega)$. It is important that by 
adjunction (see \cite[Theorem 6.3.5 (i)]{fujino-foundations}) 
$[\Nqklt (X, \omega), \omega|_{\Nqklt(X, \omega)}]$ has a 
natural quasi-log canonical structure induced by 
$\left(X, \omega, f:(Y, B_Y)\to X\right)$. 
\medskip 

The following theorem is the main result of this section. 
Although this is a special case of \cite[Theorem 1.1]{fujino-haidong}, 
we give a detailed proof for the reader's convenience. 

\begin{thm}[{see \cite[Theorem 1.1]{fujino-haidong}}]\label{c-thm10.3} 
Let 
$
\left(X, \omega, f:(Y, B_Y)\to X\right)
$ 
be a quasi-log canonical pair. 
Assume that $X$ is a normal irreducible variety. 
Then we can construct a projective surjective morphism 
$f':Y'\to X$ with the following properties: 
\begin{itemize}
\item[(i)] $(Y', B_{Y'})$ is a globally embedded simple 
normal crossing pair, $Y'$ is quasi-projective, $B_{Y'}$ is a subboundary 
$\mathbb R$-divisor, 
and $K_{Y'}+B_{Y'}\sim _{\mathbb R} (f')^*\omega$, 
\item[(ii)] the natural map 
$\mathcal O_X\to f'_*\mathcal O_{Y'}(\lceil 
-(B^{<1}_{Y'})\rceil)
$ is an isomorphism, and 
\item[(iii)] every stratum of $Y'$ is dominant onto $X$. 
\end{itemize}
Therefore, $\left(X, \omega, f': (Y', B_{Y'})\to X\right)$ 
is also a quasi-log canonical pair. 
Moreover, we have: 
\begin{itemize}
\item[(iv)] if $C$ is a qlc stratum of 
$\left(X, \omega, f':(Y', B_{Y'})\to X\right)$ then 
$C$ is a qlc stratum of  
$\left(X, \omega, f:(Y, B_{Y})\to X\right)$, and 
\item[(v)] $\Nqklt(X, \omega, f':(Y', B_{Y'})\to X)
=\Nqklt(X, \omega, f:(Y, B_{Y})\to X)$. 
\end{itemize}
Furthermore, if $K_Y+B_Y\sim _{\mathbb Q}f^*\omega$, 
then $K_{Y'}+B_{Y'}\sim _{\mathbb Q}(f')^*\omega$ holds 
by construction. 
We note that if $\left(X, \omega, f':(Y', B_{Y'})\to X\right)$ has a 
$\mathbb Q$-structure then 
$f:(Y', B_{Y'})\to X$ is a basic slc-trivial fibration 
in the sense of Definition \ref{c-def4.1}. 
\end{thm}
\begin{proof}
By \cite[Proposition 6.3.1]{fujino-foundations}, 
we may assume that $Y$ is quasi-projective 
and that the union of all strata of $(Y, B_Y)$ mapped 
to $\Nqklt(X, \omega, f:(Y, B_Y)\to X)$, which is denoted by 
$Y''$, is a union of some irreducible components of $Y$ by 
taking some suitable blow-ups of 
the ambient space $M$ of $Y$. We put $Y'=Y-Y''$ and 
$K_{Y'}+B_{Y'}=(K_Y+B_Y)|_{Y'}$. 
Then we obtain the following 
commutative diagram: 
$$
\xymatrix{
Y' \ar[d]_{f'}\ar@{^{(}->}[r]^\iota& Y \ar[d]^f\\ 
V \ar[r]_p& X
}
$$
where $\iota:Y'\hookrightarrow Y$ is a natural closed immersion 
and 
$$
\xymatrix{
Y' \ar[r]^{f'}& V\ar[r]^p & X
}
$$ is the Stein factorization of $f\circ \iota:Y'\to X$. 
By construction, the natural map $\mathcal O_V\to 
f'_*\mathcal O_{Y'}$ is an isomorphism and 
every stratum of 
$Y'$ is dominant onto $V$. 
By construction again, 
$\iota: Y'\hookrightarrow Y$ is an isomorphism over the 
generic point of $X$. 
Therefore, $p$ is birational. 
Thus, $p: V\to X$ is an isomorphism by Zariski's main theorem since $X$ is 
normal and $p$ is a finite birational morphism. 
So we have the following commutative diagram. 
$$
\xymatrix{
Y' \ar[d]_{f'}\ar@{^{(}->}[r]^\iota& Y \ar[d]^f\\ 
X \ar@{=}[r]& X
}
$$
By construction, it is obvious that 
$B_{Y'}$ is a subboundary $\mathbb R$-divisor and that 
$K_{Y'}+B_{Y'}\sim _{\mathbb R} (f')^*\omega$ holds. 
Of course, if $K_Y+B_Y\sim _{\mathbb Q}f^*\omega$, then 
$K_{Y'}+B_{Y'}\sim _{\mathbb Q}(f')^*\omega$. 
\begin{claim}
The natural map 
$$
\alpha: \mathcal O_X\to f'_*\mathcal O_{Y'}(\lceil -(B^{<1}_{Y'})\rceil)
$$ 
is an isomorphism. 
\end{claim}
\begin{proof}[Proof of Claim]
Since $X$ is normal and $f'_*\mathcal O_{Y'}(\lceil -(B^{<1}_{Y'})\rceil)$ is 
torsion-free, it is sufficient to see that $\alpha$ is an isomorphism 
in codimension one. 
Let $P$ be a prime divisor on $X$ such that 
$P\subset \Nqklt (X, \omega, f:(Y, B_Y)\to X)$. 
We note that every fiber of $f$ is connected by $f_*\mathcal O_Y\simeq 
\mathcal O_X$. 
Thus, by construction, there exists an irreducible 
component of $B^{=1}_{Y'}$ which maps onto $P$. 
Therefore, the effective divisor $\lceil -(B^{<1}_{Y'})\rceil$ does not 
contain the whole fiber of $f'$ over the generic point of $P$. Thus, 
$\alpha$ is an isomorphism at the generic point of $P$. 
This implies that the natural map $\alpha$ is an isomorphism. 
\end{proof}
By Claim, 
$(X, \omega, f':(Y', B_{Y'})\to X)$ is a quasi-log canonical pair. 
By construction, if $C$ is a qlc stratum of $(X, \omega, f':(Y', B_{Y'})\to X)$ 
then $C$ is a qlc stratum of $(X, \omega, f:(Y, B_{Y})\to X)$. 
By construction again, it is easy to see that 
$$
\mathcal I_{\Nqklt \left(X, \omega, f:(Y, B_{Y})\to X\right)}
=f'_*\mathcal O_{Y'}(\lceil -(B^{<1}_{Y'})\rceil-Y''|_{Y'})
=\mathcal I_{\Nqklt \left(X, \omega, f':(Y', B_{Y'})\to X\right)}
$$ 
(see the proof of \cite[Theorem 6.3.5 (i)]{fujino-foundations}). 
Therefore, this new quasi-log canonical pair 
$(X, \omega, f':(Y', B_{Y'})\to X)$ is the desired one. 
We note that $f':(Y', B_{Y'})\to X$ is a basic slc-trivial fibration 
in the sense of Definition \ref{c-def4.1} when 
$\left(X, \omega, f': (Y', B_{Y'})\to X\right)$ has a 
$\mathbb Q$-structure. 
\end{proof}

Theorem \ref{c-thm10.3} is one of the main motivations 
to introduce the notion of basic slc-trivial 
fibrations. 

\medskip 

We close this section with an important remark on {\em{embedded}} 
qlc centers. 

\begin{rem}\label{c-rem10.4}
In Theorem \ref{c-thm10.3}, let $C$ be 
an embedded qlc center of $(X, \omega, f:(Y, B_Y)\to X)$, that is, 
$C$ is a qlc center of $(X, \omega, f:(Y, B_Y)\to X)$ that is 
not an irreducible component of $\Nqklt(X, \omega, f:(Y, B_Y)\to X)$. 
Then it is not clear whether $C$ is also a qlc center of 
$(X, \omega, f':(Y', B_{Y'})\to X)$ or not by the above construction of 
$f': (Y', B_{Y'})\to X$. In Theorem \ref{c-thm10.3}, we just claim that 
the equality 
$$\Nqklt(X, \omega, f':(Y', B_{Y'})\to X)
=\Nqklt(X, \omega, f:(Y, B_{Y})\to X)$$ holds. 
\end{rem}

\section{Structure theorem for normal qlc 
pairs}\label{c-sec11}
 
In this section, we prove Theorem \ref{c-thm1.7}. 
We believe that Theorem \ref{c-thm1.7} will make the theory of quasi-log 
schemes more powerful and flexible. We treat various nontrivial 
applications 
of Theorem \ref{c-thm1.7} in \cite{fujino-haidong2}, 
\cite{fujino-haidong3}, and 
\cite{fujino-wenfei2}. 

\medskip 

Let us start with the following elementary lemma. 

\begin{lem}\label{c-lem11.1}
Let $\left(X, \omega, f: (Y, B_Y)\to X\right)$ be a quasi-log canonical 
pair such that $X$ is a normal irreducible variety and that every stratum 
of $Y$ is dominant onto $X$. 
Then we obtain a $\mathbb Q$-divisor $D_i$ on $Y$, 
a $\mathbb Q$-Cartier divisor $\omega_i$ on $X$, and a positive 
real number $r_i$ for $1\leq i\leq k$ such that 
\begin{itemize}
\item[(i)] $\sum _{i=1}^k r_i=1$, 
\item[(ii)] $D_i=D^{\leq 1}_i$, $\Supp D_i=\Supp B_Y$, 
$D^{=1}_i=B^{=1}_Y$, and $\lceil -(D^{<1}_i)\rceil=\lceil -(B^{<1}_Y)\rceil$ 
for every $i$, 
\item[(iii)] $\omega=\sum _{i=1}^k r_i\omega_i$ and 
$B_Y=\sum _{i=1}^k r_i D_i$, and 
\item[(iv)] $\left(X, \omega_i, f:(Y, D_i)\to X\right)$ is a quasi-log canonical 
pair with $K_Y+D_i\sim _{\mathbb Q} f^*\omega_i$ for every $i$. 
\end{itemize}
\end{lem}
\begin{proof}
We put $B_Y=\sum _j b_j B_j$, where $B_j$ is a simple normal crossing divisor 
on $Y$ for every $j$, $b_{j_1}\ne b_{j_2}$ for $j_1\ne j_2$, and 
$\Supp B_{j_1}$ and $\Supp B_{j_2}$ have no common irreducible components 
for $j_1\ne j_2$. We may assume that $b_j\in \mathbb R\setminus \mathbb Q$ 
for $1\leq j\leq l$ and 
$b_j\in \mathbb Q$ for $j\geq l+1$. 
We put $\omega=\sum _{p=1}^m a_p \omega_p$, where 
$a_p\in \mathbb R$ and $\omega_p$ is a Cartier divisor 
on $X$ for every $p$. 
We can write 
$$
K_Y+B_Y+\sum _{q=1}^n c_q (\varphi_q)=\sum _{p=1}^m a_p f^*\omega_p
$$ 
where $c_q\in \mathbb R$ and $\varphi_q\in \Gamma (Y, \mathcal K^*_Y)$ for 
every $q$. We consider the following linear map 
$$
\psi: \mathbb R^{l+m+n}  \longrightarrow 
\Gamma (Y, \mathcal K^*_Y/\mathcal O^*_Y)
\otimes _{\mathbb Z}\mathbb R
$$
defined by 
$$
\psi(x_1, \ldots, x_{l+m+n})=\sum _{\alpha=1}^m x_\alpha f^*\omega_\alpha
-\sum _{\beta=1}^n x_{m+\beta}(\varphi_\beta)-\sum _{\gamma=1}^l 
x_{m+n+\gamma}B_\gamma. 
$$
We note that $\psi$ is defined over $\mathbb Q$. 
By construction, 
$$\mathcal A:=\psi^{-1}\left(K_Y+\sum _{j\geq l+1} b_j B_j\right)$$ 
is a nonempty affine subspace of $\mathbb R^{l+m+n}$ defined over 
$\mathbb Q$. We put 
$$
P:=(a_1, \ldots, a_m, c_1, \ldots, c_n, b_1, \ldots, b_l)\in \mathcal A. 
$$
We can take $P_1, \ldots, P_k \in \mathcal A\cap \mathbb Q^{l+m+n}$ and 
$r_1, \ldots, r_k\in \mathbb R_{>0}$ such that 
$\sum _{i=1}^k r_i=1$ and $\sum _{i=1}^k r_i P_i=P$ in $\mathcal A$. 
Note that we can make $P_i$ arbitrary close to $P$ for every $i$. 
So we may assume that $P_i$ is sufficiently close to $P$ for every $i$. 
For each $P_i$, we obtain 
\begin{equation}\label{c-eq11.1}
K_Y+D_i\sim _{\mathbb Q} f^*\omega_i
\end{equation}
which satisfies (ii) by using $\psi$. 
By construction, (i) and (iii) hold. 
By \eqref{c-eq11.1} and (ii), $$\left(X, \omega_i, f: (Y, D_i)\to X\right)$$ 
is a quasi-log canonical pair for every $i$. 
Therefore, we get (iv). 
\end{proof}

We prepare one more lemma for the proof of Theorem \ref{c-thm1.7}, 
which is essentially contained in \cite[Chapter 6]{fujino-foundations}. 

\begin{lem}\label{c-lem11.2}
Let $\left(X, \omega, f:(Y, B_Y)\to X\right)$ be a quasi-log canonical 
pair such that $X$ is a normal irreducible variety. 
We assume that every stratum of $Y$ is dominant onto $X$. 
Let $P$ be a prime divisor on $X$ which is Cartier. 
We put 
$$
b_P:=\max \left\{t \in \mathbb R\, \left|\, 
\begin{array}{l}  {\text{$(Y, B_Y+tf^*P)$ is sub slc over}}\\
{\text{the generic point of $P$}} 
\end{array}\right. \right\}.  
$$ 
Then $b_P\leq 1$ holds. 
\end{lem}

\begin{proof}
If $P$ is a qlc center of $[X, \omega]$, then $b_P=0$. 
Therefore, from now on, we assume that $P$ is not a qlc center of 
$[X, \omega]$. 
By shrinking $X$ around the generic point of $P$, 
we may assume that $X$ is quasi-projective 
and that $(Y, B_Y+b_Pf^*P)$ is sub slc. 
By taking a suitable birational modification of $Y$ (see 
\cite[Theorem 1.4]{bierstone-vera}), 
we may further assume that $(Y, \Supp B_Y+\Supp f^*P)$ is a 
simple normal crossing pair. 
In this situation, $\left(X, \omega+b_PP, f: (Y, B_Y+b_Pf^*P)\to X\right)$ 
has a natural quasi-log canonical structure. 
In order to prove $b_P\leq 1$, 
we may further assume that 
$X$ is a smooth curve and $P$ is a point of $X$ by taking 
general hyperplanes of $X$ and by using adjunction. 
If $b_P>1$, then $((B_Y+f^*P)^v)^{<1}=(B_Y+f^*P)^v$ holds 
over $P$. 
This implies that $f^*P\leq \lceil -(B^{<1}_Y)\rceil$. 
Thus we get 
$$
\mathcal O_X\subsetneq \mathcal O_X(P)\subset 
f_*\mathcal O_Y(\lceil -(B^{<1}_Y)\rceil)
$$ in a neighborhood of $P$. 
This is a contradiction because the natural map 
$$\mathcal O_X\to 
f_*\mathcal O_Y(\lceil -(B^{<1}_Y)\rceil)$$ is an isomorphism. 
Therefore, we obtain $b_P\leq 1$. 
\end{proof}

Let us start the proof of Theorem \ref{c-thm1.7}. 

\begin{proof}[Proof of Theorem \ref{c-thm1.7}]
By Theorem \ref{c-thm10.3}, 
we may assume that there exists a projective 
surjective morphism 
$f:(Y, B_Y)\to X$ from a simple normal crossing 
pair $(Y, B_Y)$ such that every stratum of 
$Y$ is dominant onto $X$ and that $\left(X, \omega, 
f: (Y, B_Y)\to X\right)$ is a quasi-log canonical pair. 
By taking some more blow-ups, 
we may further assume that $(B^h_Y)^{=1}$ is 
Cartier and that every stratum of $(Y, (B^h_Y)^{=1})$ 
is dominant onto $X$ (see, for example, 
\cite[Theorem 1.4 and Section 8]{bierstone-vera} 
and \cite[Lemma 2.11]{fujino-projectivity}). 
\setcounter{step}{0}
\begin{step}\label{a-thm1.4-step1}
In this step, we treat the case where 
$[X, \omega]$ has a $\mathbb Q$-structure. 
In this situation, 
$f:(Y, B_Y)\to X$ is a basic slc-trivial 
fibration (see Theorem \ref{c-thm10.3}). 
Let $\mathbf B$ be the discriminant $\mathbb Q$-b-divisor and let 
$\mathbf M$ be the moduli $\mathbb Q$-b-divisor associated to 
$f:(Y, B_Y)\to X$. Since $(Y, B_Y)$ is sub slc, $\mathbf B_X$ is 
a subboundary $\mathbb Q$-divisor on $X$, that is, 
$\mathbf B_X=(\mathbf B_X)^{\leq 1}$ . 
By Lemma \ref{c-lem11.2}, we obtain that 
$\mathbf B_X$ is an effective $\mathbb Q$-divisor on $X$. 
By the definition of qlc centers, 
we have $f((B^v_Y)^{=1})=\nqklt (X, \omega)$. 
We take a projective 
birational morphism $p:X'\to X$ from a smooth quasi-projective 
variety $X'$. 
Let $f': (Y', B_{Y'})\to X'$ be an induced basic slc-trivial 
fibration with the following commutative diagram. 
$$
\xymatrix{
(Y, B_Y)\ar[d]_-f & (Y', B_{Y'})\ar[d]^-{f'}\ar[l]_-q\\
X & \ar[l]^-p X' 
}
$$ 
By Theorem \ref{c-thm1.2}, 
we may assume that 
there exists a simple normal crossing divisor $\Sigma_{X'}$ on $X'$ 
such that $\mathbf M=\overline{\mathbf M_{X'}}$, $\Supp \mathbf M_{X'}$ and 
$\Supp \mathbf B_{X'}$ are contained in $\Sigma_{X'}$, and that every stratum of 
$(Y', \Supp B^h_{Y'})$ is smooth over $X'\setminus \Sigma_{X'}$. 
Of course, 
we may assume that 
$M_{X'}:=\mathbf M_{X'}$ is potentially nef by Theorem \ref{c-thm1.2}. 
We may further assume that every irreducible 
component of $q^{-1}_*\left((B^v_Y)^{=1}\right)$ is mapped onto a prime 
divisor in $\Sigma_{X'}$ 
with the aid of the flattening theorem 
(see \cite[Th\'eor\`eme (5.2.2)]{raynaud-g}). 
We put $B_{X'}:=\mathbf B_{X'}$. 
Note that $B_{X'}$ is a subboundary $\mathbb Q$-divisor on $X'$ 
since $(Y', B_{Y'})$ is sub slc. 
In the above setup, 
$f'(q^{-1}_*(B^v_Y)^{=1})\subset B^{=1}_{X'}$ by the definition of 
$\mathbf B$. 
Thus, we get $\Nqklt (X, \omega)\subset p(B^{=1}_{X'})$. 
On the other hand, we can easily see that 
$p(B^{=1}_{X'})\subset \Nqklt(X, \omega)$ by definition. 
Therefore, $p(B^{=1}_{X'})=\Nqklt (X, \omega)$ holds. 
Since $p_*B_{X'}=\mathbf B_X$ and $\mathbf B_X$ is effective, 
$B^{<0}_{X'}$ is $p$-exceptional. 
Hence, $B_{X'}$ and $M_{X'}$ satisfy the desired properties. 
We note that $B_{X'}$ and $M_{X'}$ are obviously $\mathbb Q$-divisors 
by construction. 
\end{step}
\begin{step}\label{a-thm1.4-step2} 
In this step, we treat the general case. 
We first use Lemma \ref{c-lem11.1} and get a positive real number 
$r_i$ and $\left(X, \omega_i, 
f:(Y, D_i)\to X\right)$ for $1\leq i\leq k$ with 
the properties in Lemma \ref{c-lem11.1}. 
Then we apply the argument in 
Step \ref{a-thm1.4-step1} to $\left(X, \omega_i, f:(Y, D_i)\to X\right)$ for 
every $i$.  
By Theorem \ref{c-thm1.2}, we can take a projective birational morphism 
$p:X'\to X$ from a smooth quasi-projective variety $X'$ 
which works for $\left(X, \omega_i, f:(Y, D_i)\to X\right)$ for every $i$. 
By summing them up with weight $r_i$, we get 
$\mathbb R$-divisors $B_{X'}$ and $M_{X'}$ with the desired properties. 
In this case, we do not 
claim that $B_{X'}$ is the discriminant of $f': (Y', B_{Y'})\to X'$. 
\end{step} 
Therefore, we get $p:X'\to X$, $B_{X'}$, and $M_{X'}$ with 
the desired properties. 
\end{proof}

As we mentioned in Remark \ref{c-rem1.9}, 
$(X, B_X+M_X)$, where $B_X:=p_*B_{X'}$ and 
$M_X:=p_*M_{X'}$, is generalized lc in the 
sense of \cite[Definition 4.1]{birkar-zhang}. 
Moreover, if $\Nqklt(X, \omega)=\emptyset$, then 
$(X, B_X+M_X)$ is generalized klt in the sense of 
\cite[Definition 4.1]{birkar-zhang}. 
We note that $M_{X'}$ is a finite $\mathbb R_{>0}$-linear 
combination of relatively nef Cartier divisors. 
Hence $(X, B_X+M_X)$ is an NQC 
g-pair in the sense of \cite[Definition 2.13]{han-li}. 

\medskip

Finally, we prove Corollary \ref{c-cor1.10}. 

\begin{proof}[Proof of Corollary \ref{c-cor1.10}]
By adjunction (see \cite[Theorem 6.3.5]{fujino-foundations}), 
$[W, \omega|_W]$ is a quasi-log canonical 
pair. 
Since $W$ is a minimal qlc stratum of $[X, \omega]$, 
$W$ is a normal irreducible variety 
and $\Nqklt(W, \omega|_W)=\emptyset$ holds (see \cite[Theorem 6.3.5 and 
Lemma 6.3.9]{fujino-foundations}). 
By Theorem \ref{c-thm1.7}, 
we can take a projective birational morphism $p:W'\to W$ 
from a smooth quasi-projective 
variety $W'$, a subboundary $\mathbb R$-divisor $B_{W'}$ whose 
support is a simple normal crossing divisor on $W'$, 
a potentially 
nef $\mathbb R$-divisor $M_{W'}$ on $W'$ such that $p^*(\omega|_W)
=K_{W'}+B_{W'}+M_{W'}$. 
Since $\Nqklt (W, \omega|_W)=\emptyset$ holds, 
we may assume that $B_{W'}=B^{<1}_{W'}$. 
By taking some more blow-ups, if necessary, 
we may further assume that 
there exists an effective $p$-exceptional 
Cartier divisor $E$ on $W'$ such that 
$\Supp B_{W'}\cup \Supp E$ is contained in  
a simple normal crossing divisor and that $-E$ is $p$-ample. 
We note that $-\varepsilon E+p^*H+M_{W'}$ is semi-ample for 
any $0<\varepsilon \ll 1$. 
Therefore, we can take a general effective $\mathbb R$-divisor 
$G\sim _{\mathbb R} -\varepsilon E+p^*H+M_{W'}$ such that 
$\Supp (B_{W'}+\varepsilon E+G)$ is a simple normal crossing divisor 
on $W'$ and $\lfloor B_{W'}+\varepsilon E+G\rfloor\leq 0$. 
By construction, $K_{W'}+B_{W'}+M_{W'}+p^*H\sim _{\mathbb R} 
K_{W'}+B_{W'}+\varepsilon E+G$ holds. 
We put $\Delta_W=p_*(B_{W'}+\varepsilon E+G)$. 
Then $\Delta_W$ satisfies the desired properties. 

When $[X, \omega]$ has a $\mathbb Q$-structure and 
$H$ is an ample $\mathbb Q$-divisor, it is easy to 
see that we can make $\Delta_W$ a $\mathbb Q$-divisor with $K_W+
\Delta_W\sim _{\mathbb Q} \omega|_W+H$ by the above construction 
of $\Delta_W$.   
\end{proof}

\section{On the basepoint-freeness}\label{c-sec12}

In this section, we give a small remark on the 
basepoint-free theorem for quasi-log canonical pairs. 

\medskip 

The following theorem is a special case of 
the basepoint-free theorem for quasi-log schemes (see \cite[Theorem 6.5.1]
{fujino-foundations}). 
We can quickly reduce Theorem \ref{c-thm12.1} 
to the usual Kawamata--Shokurov basepoint-free theorem 
for kawamata log terminal pairs by Corollary \ref{c-cor1.10}. 
Note that the general basepoint-free theorem 
for quasi-log schemes (see \cite[Theorem 6.5.1]{fujino-foundations}) 
easily follows from Theorem \ref{c-thm12.1}. 
For the details, see Claims 1, 3, and 4 in the proof of \cite[Theorem 6.5.1]
{fujino-foundations}. 

\begin{thm}[{Basepoint-free theorem, see 
\cite[Theorem 6.5.1]{fujino-foundations}}]\label{c-thm12.1}
Let $[X, \omega]$ be a quasi-log canonical pair 
with $\Nqklt(X, \omega)=\emptyset$ and let $\pi:X\to S$ be 
a projective morphism 
between schemes. 
Let $L$ be a $\pi$-nef Cartier divisor on $X$. 
Assume that 
$qL-\omega$ is $\pi$-ample for some real number 
$q>0$. 
Then there exists a positive number $m_0$ such that 
$\mathcal O_X(mL)$ is $\pi$-generated 
for every integer $m\geq m_0$. 
\end{thm}

\begin{proof}
Without loss of generality, we may assume that $S$ is 
quasi-projective. Then $X$ is also quasi-projective. 
Therefore, we can take an ample 
$\mathbb Q$-divisor $H$ on $X$ such that 
$qL-(\omega+H)$ is still $\pi$-ample. 
By Corollary \ref{c-cor1.10}, 
we can take an effective 
$\mathbb R$-divisor $\Delta_X$ on $X$ such that 
$\omega+H\sim _{\mathbb R} K_X+\Delta_X$ and that 
$(X, \Delta_X)$ is kawamata log terminal. 
Therefore, by the usual Kawamata--Shokurov basepoint-free 
theorem for kawamata log terminal pairs, we obtain a positive 
number $m_0$ such that $\mathcal O_X(mL)$ is $\pi$-generated 
for every integer $m\geq m_0$. 
\end{proof}

\section{Supplements to \cite{fujino-fujisawa}}\label{c-sec13} 
In this section, we give some 
supplementary remarks on \cite{fujino-fujisawa} for the 
reader's convenience. We believe that there are no serious 
troubles in \cite{fujino-fujisawa}. 
However, we found that 
it contains some minor mistakes and ambiguities. 
So we fix them here. For a completely different approach to 
the results in \cite{fujino-fujisawa} based on 
Saito's theory of mixed Hodge modules, see \cite{ffs}. 

\begin{say}[Base change theorem]\label{c-say13.1}
We note that the statement of \cite[Lemma 3.4 (iv)]{fujino-fujisawa} 
is correct. However, the proof of \cite[Lemma 3.4 (iv)]{fujino-fujisawa} 
is somewhat misleading. Therefore, we recommend the interested 
reader to see 
\cite[Lemma 2.20]{fujisawa2} and its proof. 
We think that 
\cite[Lemma 3.4]{fujino-fujisawa} is an easy exercise. 
\end{say}

\begin{say}[Semipositivity theorem]\label{c-say13.2}
In \cite[Section 5]{fujino-fujisawa}, 
we discussed a generalization of the Fujita--Zucker--Kawamata 
semipositivity theorem (see \cite[Theorem 5.21]{fujino-fujisawa}), 
which plays a crucial role in this paper. 
We used \cite[Corollary 5.23]{fujino-fujisawa}, 
which is an easy consequence of \cite[Theorem 5.21]{fujino-fujisawa}, 
in Theorem \ref{c-thm3.1} (ii). 
Unfortunately, there are some ambiguities in the 
arguments in \cite[Section 5]{fujino-fujisawa}. 
In \cite[5.8]{fujino-fujisawa}, we defined the condition $(m\mathrm{MH})$. 
It was not precise enough because the real structure 
was not mentioned explicitly. 
In \cite[Section 2]{fujisawa}, Taro Fujisawa, who is one of the authors 
of \cite{fujino-fujisawa}, removed the ambiguities 
from \cite[Section 5]{fujino-fujisawa}. 
We recommend the reader to see \cite{fujisawa}. 
We also recommend the interested reader to see \cite[Theorem 3]{ffs} and 
\cite{fujino-fujisawa2}. 
In \cite{fujino-fujisawa2}, we give an analytic generalization of 
the Fujita--Zucker--Kawamata semipositivity 
theorem whose proof is completely 
different from the 
arguments in \cite[Section 5]{fujino-fujisawa}. 
\end{say}

\begin{say}[Lemma on two filtrations]\label{c-say13.3} 
In Section 4 of \cite{fujino-fujisawa},
the lemma on two filtrations
\cite[Propositions (7.2.5) and (7.2.8)]{deligne}
(see also \cite[Theorem 3.12]{peters-steenbrink})
was used
several times
(explicitly stated at
p.~608, the proof of Lemma 4.5,
p.~610, Remark 4.6,
p.~618, Step 1 of the proof of Lemma 4.10
and p.~623, the proof of Lemma 4.12,
and implicitly used
at p.~611, the proof of Lemma 4.8).
However, there are missing points
in the arguments.

Let $K$ be a complex,
$W$ a finite increasing filtration on $K$
and $F$ a finite decreasing filtration on $K$.
In order to apply the lemma on two filtrations
for the spectral sequence
\begin{equation*}
(E_r^{p,q}(K,W), F_{\rec}),
\end{equation*}
it is necessary to discuss about the $E_0$-terms.
More precisely,
it has to be checked that
the strictness of the filtration $F$
on the complex $\Gr_m^WK$ holds true for all $m$.
Here we will explain how to check this strictness
for the case of Lemma 4.10 of \cite{fujino-fujisawa}
mentioned above.
For the other cases, the similar arguments are valid.

In Step 1 of the proof of Lemma 4.10,
the bifiltered complex
\begin{equation*}
(Rf_{\ast}\Omega_{X_{\bullet}/\Delta}(\log E_{\bullet}), L, F)
\end{equation*}
is studied.
Thus the strictness of the filtration $F$
on the complex
\begin{equation*}
\Gr_m^LRf_{\ast}\Omega_{X_{\bullet}/\Delta}(\log E_{\bullet})
\end{equation*}
has to be checked for all $m$.
Under the canonical isomorphism
\begin{equation*}
\begin{split}
\Gr_m^LRf_{\ast}\Omega_{X_{\bullet}/\Delta}(\log E_{\bullet})
&\simeq
Rf_{\ast}\Gr_m^L\Omega_{X_{\bullet}/\Delta}(\log E_{\bullet}) \\
&\simeq
Rf_{-m \ast}\Omega_{X_{-m}/\Delta}(\log E_{-m})[m],
\end{split}
\end{equation*}
the filtration $F$ coincides
with the filtration induced from the stupid filtration, 
which is denoted by $F$ again, 
on the complex $\Omega_{X_{-m}/\Delta}(\log E_{-m})$. 
Therefore it suffices to prove the strictness of the filtration $F$
on $Rf_{-m \ast}\Omega_{X_{-m}/\Delta}(\log E_{-m})$ 
that is induced by the stupid filtration $F$ on 
$\Omega_{X_{-m}/\Delta}(\log E_{-m})$.
The strictness of $F$
on $Rf_{-m \ast}\Omega_{X_{-m}/\Delta}(\log E_{-m})$
is equivalent to the $E_1$-degeneracy
of the spectral sequence
$E_r^{p,q}(Rf_{-m \ast}\Omega_{X_{-m}/\Delta}(\log E_{-m}), F)$.
We note that the morphism of $E_r$-terms
\begin{equation*}
\begin{split}
d_r:
E_r^{p,q}(Rf_{-m \ast}\Omega&_{X_{-m}/\Delta}(\log E_{-m}), F) \\
&\longrightarrow
E_r^{p+r,q-r+1}(Rf_{-m \ast}\Omega_{X_{-m}/\Delta}(\log E_{-m}), F)
\end{split}
\end{equation*}
is zero on $\pd$ for all $p,q$ and for all $r \ge 1$
because $X_{-m} \longrightarrow \Delta$ is smooth and projective
over $\pd$.
On the other hand,
\begin{equation*}
\begin{split}
E_1^{p,q}(Rf_{-m \ast}\Omega_{X_{-m}/\Delta}(\log E_{-m}), F)
&=R^{p+q}f_{-m \ast}\Gr_F^p\Omega_{X_{-m}/\Delta}(\log E_{-m}) \\
&=R^qf_{-m \ast}\Omega^p_{X_{-m}/\Delta}(\log E_{-m}),
\end{split}
\end{equation*}
is a locally free $\mathcal O_{\Delta}$-module of finite rank
by \cite[(2.11) Theorem]{steenbrink}.
Therefore the morphism of $E_1$-terms $d_1$
is zero on the whole $\Delta$ for all $p,q$.
Inductively on $r$,
we obtain that
$E_r^{p,q}(Rf_{-m \ast}\Omega_{X_{-m}/\Delta}(\log E_{-m}), F)$
is a locally free $\mathcal O_{\Delta}$-module of finite rank
and that $d_r$ is zero on the whole $\Delta$
for all $p,q$ and for all $r \ge 1$.
Thus the $E_1$-degeneracy is proved.
\end{say}


\begin{thebibliography}{FLh3} 

\bibitem[A1]{ambro1} 
F.~Ambro, The adjunction conjecture and its applications, 
Thesis (Ph.D.)--The Johns Hopkins University. 1999. 
arXiv:math/9903060 [math.AG]

\bibitem[A2]{ambro-mld} 
F.~Ambro, On minimal log discrepancies, 
Math. Res. Lett. \textbf{6} (1999), no. 5-6, 573--580. 

\bibitem[A3]{ambro2} 
F.~Ambro, Quasi-log varieties, 
Tr. Mat. Inst. Steklova \textbf{240} (2003), 
Biratsion. Geom. Line\u{\i}n. Sist. Konechno 
Porozhdennye Algebry, 220--239; translation 
in Proc. Steklov Inst. Math. 2003, no. 1(240), 214--233.  

\bibitem[A4]{ambro3}
F.~Ambro, 
Shokurov's boundary property, 
J. Differential Geom. \textbf{67} (2004), no. 2, 229--255.

\bibitem[A5]{ambro5} 
F.~Ambro, The moduli b-divisor of an lc-trivial fibration, 
Compos. Math. \textbf{141} (2005), no. 2, 385--403. 

\bibitem[BVP]{bierstone-vera}
E.~Bierstone, F.~Vera Pacheco, 
Resolution of singularities of pairs preserving 
semi-simple normal crossings, 
Rev. R. Acad. Cienc. Exactas F\'is. Nat. Ser. A 
Math. RACSAM \textbf{107} (2013), no. 1, 159--188. 

\bibitem[BZ]{birkar-zhang} 
C.~Birkar, D.-Q.~Zhang, 
Effectivity of Iitaka fibrations 
and pluricanonical systems of 
polarized pairs, 
Publ. Math. Inst. Hautes \'Etudes Sci. \textbf{123} 
(2016), 283--331.
 
\bibitem[C]{corti} 
A.~Corti, $3$-fold flips after Shokurov, 
{\em{Flips for $3$-folds and $4$-folds}}, 18--48, 
Oxford Lecture Ser. Math. Appl., \textbf{35}, Oxford 
Univ. Press, Oxford, 2007.

\bibitem[D]{deligne}
P.~Deligne, 
Th\'eorie de Hodge. III, 
Inst. Hautes \'Etudes Sci. Publ. Math. No. \textbf{44} (1974), 5--77.

\bibitem[EV]{esnault-viehweg} 
H.~Esnault, E.~Viehweg, {\em{Lectures on vanishing theorems}}, 
DMV Seminar, \textbf{20}. Birkh\"auser Verlag, Basel, 1992.

\bibitem[Fi]{fischer} 
G.~Fischer, {\em{Complex analytic geometry}}, 
Lecture Notes in Mathematics, Vol. \textbf{538}. Springer-Verlag, 
Berlin-New York, 1976. 

\bibitem[Fl]{floris} 
E.~Floris, 
Inductive approach to effective b-semiampleness, 
Int. Math. Res. Not. IMRN 2014, no. 6, 1465--1492.

\bibitem[Fn1]{fujino-applications} 
O.~Fujino, Applications of Kawamata's positivity theorem, 
Proc. Japan Acad. Ser. A Math. Sci. \textbf{75} (1999), no. 6, 75--79.

\bibitem[Fn2]{fujino-certain} 
O.~Fujino, A canonical bundle formula for certain algebraic 
fiber spaces and its applications, 
Nagoya Math. J. \textbf{172} 
(2003), 129--171.

\bibitem[Fn3]{fujino-higher-pre} 
O.~Fujino, Higher direct images of log canonical divisors 
and positivity theorems, preprint (2003). 
arXiv:math/0302073 [math.AG]

\bibitem[Fn4]{fujino-higher} 
O.~Fujino, Higher direct images of log canonical divisors, 
J. Differential Geom. \textbf{66} (2004), no. 3, 453--479.

\bibitem[Fn5]{fujino-lmmp} 
O.~Fujino, Vanishing and injectivity theorems for LMMP, 
preprint (2007). arXiv:0705.2075 [math.AG]

\bibitem[Fn6]{fujino-fund}
O.~Fujino, Fundamental theorems for the log minimal model program, 
Publ. Res. Inst. Math. Sci. \textbf{47} (2011), no. 3, 727--789.

\bibitem[Fn7]{fujino-slc}
O.~Fujino, Fundamental theorems for semi log canonical 
pairs, Algebr. Geom. \textbf{1} (2014), no. 2, 194--228.

\bibitem[Fn8]{fujino-some} 
O.~Fujino, 
Some remarks on the minimal model program for log canonical pairs, 
J. Math. Sci. Univ. Tokyo \textbf{22} (2015), no. 1, 149--192. 

\bibitem[Fn9]{fujino-vanishing}
O.~Fujino, Vanishing theorems, 
{\em{Minimal models and extremal rays (Kyoto, 2011)}}, 299--321, 
Adv. Stud. Pure Math., \textbf{70}, Math. Soc. Japan, [Tokyo], 2016. 

\bibitem[Fn10]{fujino-foundations}
O.~Fujino, {\em{Foundations of the minimal model program}}, 
MSJ Memoirs, \textbf{35}. Mathematical Society of Japan, Tokyo, 2017.

\bibitem[Fn11]{fujino-slc-surface} 
O.~Fujino, Effective basepoint-free theorem for semi-log canonical surfaces, 
Publ. Res. Inst. Math. Sci. \textbf{53} (2017), no. 3, 349--370.

\bibitem[Fn12]{fujino-zucker65} 
O.~Fujino, On semipositivity, injectivity and vanishing theorems, 
{\em{Hodge theory and $L^2$-analysis}}, 245--282, 
Adv. Lect. Math. (ALM), \textbf{39}, Int. Press, Somerville, MA, 2017. 

\bibitem[Fn13]{fujino-injectivity} 
O.~Fujino, Injectivity theorems, 
{\em{Higher Dimensional Algebraic Geometry 
in honour of Professor Yujiro Kawamata's sixtieth birthday}}, 
131--157, 
Adv. Stud. Pure Math., \textbf{74}, Math. Soc. Japan, [Tokyo], 2017. 

\bibitem[Fn14]{fujino-projectivity} 
O.~Fujino, Semipositivity theorems for moduli problems, 
Ann. of Math. (2) \textbf{187} (2018), no. 3, 639--665.

\bibitem[Fn15]{fujino-vani-semi}
O.~Fujino, 
Vanishing and semipositivity theorems for semi-log canonical pairs, 
Publ. Res. Inst. Math. Sci. \textbf{56} (2020), no. 1, 15--32.  

\bibitem[Fn16]{fujino-wenfei2} 
O.~Fujino, Subadjunction for quasi-log canonical 
pairs and its applications, in preparation. 

\bibitem[FF1]{fujino-fujisawa}
O.~Fujino, T.~Fujisawa, 
Variations of mixed Hodge structure and semipositivity theorems, 
Publ. Res. Inst. Math. Sci. \textbf{50} (2014), no. 4, 589--661. 

\bibitem[FF2]{fujino-fujisawa3} 
O.~Fujino, T.~Fujisawa, A short remark on [2], preprint 
(2017). 

\bibitem[FF3]{fujino-fujisawa2}
O.~Fujino, T.~Fujisawa, 
On semipositivity theorems, 
Math. Res. Lett. \textbf{26} (2019), no. 5, 1359--1382. 

\bibitem[FFL]{fujino-fujisawa-liu} 
O.~Fujino, T.~Fujisawa, H.~Liu, 
Fundamental properties of basic slc-trivial fibrations, II, preprint (2018). 

\bibitem[FFS]{ffs} 
O.~Fujino, T.~Fujisawa, M.~Saito, 
Some remarks on the semipositivity theorems, 
Publ. Res. Inst. Math. Sci. \textbf{50} (2014), no. 1, 85--112. 

\bibitem[FG1]{fujino-gongyo-sub}
O.~Fujino, Y.~Gongyo, 
On canonical bundle formulas and subadjunctions, 
Michigan Math. J. \textbf{61} (2012), no. 2, 255--264. 

\bibitem[FG2]{fujino-gongyo} 
O.~Fujino, Y.~Gongyo, 
On the moduli b-divisors of lc-trivial fibrations, 
Ann. Inst. Fourier (Grenoble) \textbf{64} (2014), no. 4, 1721--1735.

\bibitem[FM]{fujino-mori}
O.~Fujino, S.~Mori, 
A canonical bundle formula, 
J. Differential Geom. \textbf{56} (2000), no. 1, 167--188. 

\bibitem[FLh1]{fujino-haidong}
O.~Fujino, H.~Liu, On normalization of quasi-log canonical pairs, 
Proc. Japan Acad. Ser. A Math. Sci. \textbf{94} (2018), no. 10, 97--101. 

\bibitem[FLh2]{fujino-haidong2} 
O.~Fujino, H.~Liu, Quasi-log canonical pairs are Du Bois, 
to appear in  J. Algebraic Geom.  

\bibitem[FLh3]{fujino-haidong3} 
O.~Fujino, H.~Liu, Fujita-type freeness for 
quasi-log canonical curves and surfaces, 
to appear in Kyoto J. Math. 

\bibitem[FLw]{fujino-wenfei} 
O.~Fujino, W.~Liu, 
Simple connectedness of Fano log pairs with semi-log canonical 
singularities, 
to appear in Math. Z.  

\bibitem[Fs1]{fujisawa2} 
T.~Fujisawa, 
Limits of Hodge structures in several variables, II, 
preprint (2015). 
arXiv:1506.02271 [math.AG]

\bibitem[Fs2]{fujisawa} 
T.~Fujisawa, A remark on semipositivity theorems, preprint (2017). 
arXiv:1710.01008 [math.AG]

\bibitem[HL]{han-li} 
J.~Han, Z.~Li, Weak Zariski decompositions 
and log terminal models for generalized polarized pairs, 
preprint (2018). arXiv:1806.01234 [math.AG]

\bibitem[H]{hartshorne} 
R.~Hartshorne, {\em{Algebraic geometry}}, 
Graduate Texts in Mathematics, No. \textbf{52}. Springer-Verlag, 
New York-Heidelberg, 1977.

\bibitem[Ka1]{kawamata-abel}
Y.~Kawamata, Characterization of abelian varieties, 
Compositio Math. \textbf{43} (1981), no. 2, 253--276.

\bibitem[Ka2]{kawamata-sub1} 
Y.~Kawamata, Subadjunction of log canonical 
divisors for a subvariety of codimension $2$, 
{\em{Birational algebraic geometry (Baltimore, MD, 1996)}}, 
79--88, Contemp. Math., \textbf{207}, Amer. Math. Soc., 
Providence, RI, 1997.

\bibitem[Ka3]{kawamata} 
Y.~Kawamata, 
Subadjunction of log canonical 
divisors. II, 
Amer. J. Math. \textbf{120} (1998), no. 5, 893--899. 

\bibitem[Ko1]{kollar-projectivity} 
J.~Koll\'ar, Projectivity of complete moduli, 
J. Differential Geom. \textbf{32} (1990), no. 1, 235--268. 

\bibitem[Ko2]{kollar1} 
J.~Koll\'ar, Kodaira's canonical bundle formula and adjunction, 
{\em{Flips for $3$-folds and $4$-folds}}, 134--162, 
Oxford Lecture Ser. Math. Appl., \textbf{35}, 
Oxford Univ. Press, Oxford, 2007.

\bibitem[Ko3]{kollar}
J.~Koll\'ar, {\em{Singularities of the minimal model program}}. With 
a collaboration of S\'andor Kov\'acs, 
Cambridge Tracts in Mathematics, \textbf{200}. Cambridge 
University Press, Cambridge, 2013.

\bibitem[KK]{kollar-kovacs}
J.~Koll\'ar, S.~J.~Kov\'acs, 
Log canonical singularities are Du Bois, 
J. Amer. Math. Soc. \textbf{23} (2010), no. 3, 791--813. 

\bibitem[KM]{kollar-mori}
J.~Koll\'ar, S.~Mori, {\em{Birational geometry of 
algebraic varieties}}. With the 
collaboration of C.~H.~Clemens and A.~Corti. Translated 
from the 1998 Japanese original. Cambridge Tracts in 
Mathematics, \textbf{134}. Cambridge University Press, Cambridge, 1998.

\bibitem[L]{lut} 
W.~L\"utkebohmert, 
On compactification of schemes, 
Manuscripta Math. \textbf{80} (1993), no. 1, 95--111. 

\bibitem[M]{mori} 
S.~Mori, Classification of higher-dimensional varieties, 
{\em{Algebraic geometry, Bowdoin, 1985 (Brunswick, Maine, 1985)}}, 
269--331, Proc. Sympos. Pure 
Math., \textbf{46}, Part 1, Amer. Math. Soc., Providence, RI, 1987.

\bibitem[PeSt]{peters-steenbrink}
C.~A.~M.~Peters, J.~H.~M.~Steenbrink, 
{\em{Mixed Hodge structures}}, 
Ergebnisse der Mathematik und ihrer Grenzgebiete. 3. Folge. A 
Series of Modern Surveys in Mathematics [Results 
in Mathematics and Related Areas. 3rd Series. A 
Series of Modern Surveys in Mathematics], \textbf{52}. Springer-Verlag, 
Berlin, 2008. 

\bibitem[PrSh]{prokhorov-shokurov}
Y.~Prokhorov, V.~V.~Shokurov, 
Towards the second main theorem on complements, 
J. Algebraic Geom. \textbf{18} (2009), no. 1, 151--199. 

\bibitem[RG]{raynaud-g} 
M.~Raynaud, L.~Gruson, 
Crit\`eres de platitude et de projectivit\'e. Techniques 
de \lq\lq platification\rq\rq \ d'un module, 
Invent. Math. \textbf{13} (1971), 1--89. 

\bibitem[Sh]{shokurov} 
V.~V.~Shokurov, $3$-fold log models, 
Algebraic geometry, 4. J. Math. Sci. \textbf{81} (1996), no. 3, 2667--2699.

\bibitem[St]{steenbrink}
J.~H.~M.~Steenbrink, 
Mixed Hodge structure on the vanishing cohomology, 
{\em{Real and complex 
singularities (Proc. Ninth Nordic 
Summer School/NAVF Sympos. Math., Oslo, 1976)}}, 525--563. Sijthoff 
and Noordhoff, Alphen aan den Rijn, 1977. 

\bibitem[V]{viehweg} 
E.~Viehweg, {\em{Quasi-projective moduli for polarized manifolds}}, 
Ergebnisse der Mathematik und ihrer Grenzgebiete (3) [Results 
in Mathematics and Related Areas (3)], \textbf{30}. Springer-Verlag, 
Berlin, 1995.

\end{thebibliography}
\end{document}